\documentclass[twoside,leqno]{article}
\usepackage[paperwidth=18cm, paperheight=25.8cm,textwidth=13cm, textheight=21cm,left=2.5cm,top=2.5cm]{geometry}
\usepackage{indentfirst}
\usepackage[numbers,sort&compress]{natbib}
\usepackage{titletoc}
\titlecontents{section}
[1cm]
{\rm \small}%
{\contentslabel{2em}}%
{}%
{\titlerule*[0.5pc]{$\cdot$}\contentspage\hspace*{0.5cm}}%

\titlecontents{subsection}
[1.6cm]
{\rm \small}%
{\contentslabel{2em}}%
{}%
{\titlerule*[0.5pc]{$\cdot$}\contentspage\hspace*{0.5cm}}%

\usepackage{graphicx}
\usepackage{mathrsfs,amsfonts,amssymb,amsmath}
\usepackage{stmaryrd,cases}
\usepackage{hyperref}
\newcommand{\norm}[1]{\|#1\|}
\newcommand{\lnorm}[1]{\left\|#1\right\|}
\newcommand{\hnorm}[1]{\|#1\mathring{\|}}
\newcommand{\abs}[1]{\left|#1\right|}

\newcommand{\R}{\mathbb R}
\newcommand{\D}{\partial}
\newcommand{\Pdv}{\mathbb{P}}
\newcommand{\Id}{\mathbb{I}}
\newcommand{\M}{\mc{M}}
\newcommand{\eps}{\varepsilon}

\newcommand{\dv}{\mathrm{div}\,}

\newcommand{\nb}{\nabla}

\newcommand{\curl}{\mathrm{curl}\,}
\newcommand{\dist}{\mathrm{dist}\,}
\newcommand{\ls}{\leqslant\,}
\newcommand{\gs}{\geqslant\,}

\newcommand{\sgn}{\mathrm{sgn}}

\newcommand{\no}{\nonumber}

\newcommand{\bic}{\complement}

\newcommand{\mc}[1]{{\mathcal{#1}}}
\newcommand{\Lie}{\mathcal{L}}
\newcommand{\A}{\mathcal{A}}
\allowdisplaybreaks
\numberwithin{equation}{section}
\usepackage{amsthm}   
\newtheorem{theorem}{Theorem}[section]

\newtheorem{lemma}[theorem]{Lemma}
\newtheorem{proposition}[theorem]{Proposition}
\theoremstyle{definition}
\newtheorem{definition}[theorem]{Definition}
\theoremstyle{remark}
\newtheorem{remark}[theorem]{Remark}

\begin{document}

\title{\Large\bf Well-Posedness for the Linearized Free Boundary Problem of Incompressible Ideal Magnetohydrodynamics Equations}

\author{\normalsize CHENGCHUN HAO\\[-3pt]
	\small\it Academy of Mathematics \& Systems Science\\[-4pt]
\small\it	and Hua Loo-Keng Key Laboratory of Mathematics\\[-4pt]
\small\it	Chinese Academy of Sciences\\[1mm]
\normalsize 	 TAO LUO\\[-3pt]
\small\it City University of Hong Kong}
\date{}
\maketitle

\begin{abstract}
	We study the well-posedness theory for the linearized free boundary problem of incompressible ideal magnetohydrodynamics equations in a bounded domain. We express the magnetic field in terms of the velocity field and the deformation tensors in the Lagrangian coordinates, and substitute the magnetic field into the momentum equation to get an equation of the velocity in which the initial magnetic field serves only as a parameter. Then, we linearize this equation with respect to the position vector field whose time derivative is the velocity, and obtain the local-in-time well-posedness of the solution by using energy estimates of the tangential derivatives and the curl with the help of Lie derivatives and the smooth-out approximation.
\end{abstract}
\tableofcontents

\pagestyle{myheadings}
\markboth{\small\rm C. HAO AND T. LUO\hfil}{\hfil\small\rm LWP OF LINEARIZED INCOMPRESSIBLE IDEAL MHD}

\section{Introduction}

This paper is concerned with the well-posedness of the linearized motion of the following incompressible ideal magnetohydrodynamics (MHD) with free boundary
\begin{align}
	&v_t+v\cdot\nb v+\nabla p=\mu H\cdot\nb H,  &\text{ in } \mathcal{D}, \label{mhd.1}\\
	&H_t+v\cdot \nb H=H\cdot \nb v,  &\text{ in } \mathcal{D}, \label{mhd.2}\\
	&\dv  v=0,  \quad \dv  H=0, &\text{ in } \mathcal{D}, \label{mhd.3}
\end{align}
where $v$ is the velocity field, $H$ is the magnetic field, $p$ is the total pressure including the fluid pressure and the magnetic pressure, and $\mu>0$ is the vacuum permeability, $\mathcal{D}=\cup_{0\ls t\ls T}\{t\}\times \Omega_t$, $\Omega_t\subset \R^n$ is the domain occupied by the fluid at time $t$. 

We also require boundary conditions on the free boundary $\D\mathcal{D}$:
\begin{align}
	 &H\cdot\mathcal{N}=0,\quad p=0, \quad\text{ on } \D\mathcal{D},\label{bddry}\\
	 &(\D_t+v^k\D_k)|_{\D\mathcal{D}}\in T(\D\mathcal{D}), \label{comoving}
\end{align}
where $\mathcal{N}$ is the exterior unit normal to $\Gamma_t:=\D\Omega_t$. The condition $p=0$ indicates that the total pressure vanishes outside the domain. Here the fluid considered is an incompressible ideal case.  The incompressibility condition  prevents the body from expanding, and the fact that the pressure is positive  prevents the body from breaking up in the interior.  From a physical point of view,   the total pressure can be thought of alternatively as being a small positive constant on the boundary instead of vanishing.  The challenge of the problem is that  the regularity of the boundary enters to the highest order.  Roughly speaking, the velocity determines the motion of  the boundary, and the boundary is the level set of the total pressure that determines the acceleration together with the magnetic tension. The condition $H\cdot\mathcal{N}=0$ comes from the assumption that the boundary $\Gamma_t$ is a perfect conductor, and should be understood as the constraints on the initial data since it will hold true for all $t\in [0,T]$ if it holds initially as showed in \cite{HLarma}. The condition \eqref{comoving} means that the boundary moves with the velocity $v$ of the fluid particles at the boundary.

Given a domain $\Omega\subset \R^n$ that is homeomorphic to the unit ball, and initial data $(v_0,H_0)$ satisfying the constrain \eqref{mhd.3}, we expect to find a set $\mathcal{D}\subset [0,T]\times \R^n$ and vector fields $(v,H)$ solving \eqref{mhd.1}-\eqref{comoving} with initial conditions
\begin{align}\label{idata}
\{x: (0,x)\in\mathcal{D}\}=\Omega;\quad  v=v_0, \; H=H_0, \text{ on } \{0\}\times \Omega.
\end{align}
Then, let $\Omega_t=\{x:\, (t,x)\in\mathcal{D}\}$. Motivated by the Taylor sign condition on the fluid pressure for the Euler equations, we raised an analogous condition based on the total pressure for ideal MHD in \cite{HLarma}:
\begin{align}
\nb_\mathcal{N} p\ls -c_0<0 \text{ on } \D\mathcal{D}, \label{nbNPcond}
\end{align}
where $\nb_{\mathcal{N}}=\mathcal{N}^i\D_{x^i}$. Here we have used the summation convection over repeated upper and lower indices. 
In \cite{HLarma}, we have proved a priori estimates in standard Sobolev spaces  for the free boundary problem of incompressible ideal MHD system \eqref{mhd.1}-\eqref{idata} under the condition \eqref{nbNPcond}. We also showed in \cite{HLip} that the above free boundary problem \eqref{mhd.1}-\eqref{idata} under consideration would be ill-posed at least for the case $n=2$ if the condition \eqref{nbNPcond} was violated. Thus, it will be much reasonable and necessary to require this condition \eqref{nbNPcond} in the studies of well-posedness of the considering free boundary problem of incompressible ideal MHD equations.

However, the a priori estimates in \cite{HLarma} used all the symmetries of the nonlinear equations and so only holds for perturbations of the equations that preserved all the symmetries. Thus, we can not use those a priori estimates for solutions of linearized equations that do not preserve the symmetries. Of course, the results in \cite{HLarma} are  important for the raise of the meaningful and reasonable condition \eqref{nbNPcond} for the well-posedness. 

In this paper, we prove the existence of solutions in Sobolev spaces for linearized equations using a new type of estimate by using some ideas in \cite{L1}. Existence for the linearized equations or some modification are crucially  important to any existence proof for the nonlinear problem by putting the nonlinear problem in some iteration schemes. In the most general way, it is to linearize the equations with respect to both the velocity field and the magnetic field. However, it is almost impossible to get the well-posedness for the linearized system of this type, where many operators can not be controlled and the relations between the velocity field and the magnetic field is also destroyed. In order to preserve these important relations, we  seek another way to linearize the equations. Since the magnetic field can be expressed in terms of the velocity field and the deformation tensors in the Lagrangian coordinates, we can first solve the equation \eqref{mhd.2} and substitute the magnetic field into the equation \eqref{mhd.1} to get an equation of the velocity in which the initial magnetic field serves only as a parameter. Then, we linearize this equation with respect to the position vector field whose time derivative is the velocity in the Lagrangian coordinates. As in \cite{L1}, we project the linearized equation onto an equation in the interior using the orthogonal projection onto divergence-free vector fields in the $L^2$ inner product, which removes a difficult term, the differential of linearization of the pressure, and reduces a higher-order term, the linearization of the free boundary, to an unbounded symmetric operator on divergence-free vector fields. Thus, the linearized equation turns to an evolution equation in the interior for this so-called normal operator that is positive due to the condition \eqref{nbNPcond} and leads to energy bounds. Because the operator is time dependent and nonelliptic, we can not obtain the existence of regular solutions by standard energy methods or semigroup methods. As practiced in \cite{L1}, to use Lie derivatives with respect to divergence-free vector fields tangential at the boundary is an effective way. The estimates of all derivatives can be got from those of tangential derivatives, the divergence and the curl. We replace the normal operator by a sequence of bounded operators converging to it that are still symmetric and positive and have uniformly estimates in order to get the existence of solutions.

Fluids free boundary problems arising from  physical, engineering and medical models are both important in applications and challenging in PDE theory. Examples include water waves, evolution of boundaries of stars, vortex sheets, multi-phase flow,  reacting flow, shock waves, biomedical modeling such as tumor growth, cell deformation and etc.  The most fundamental and simplest setting is for incompressible fluids for which the local well-posedness in Sobolev spaces for inviscid irrotational flow was obtained first in \cite{wu1, wu2} for 2D and 3D, respectively.  Substantial progresses for the cases without the irrotational assumption, finite depth water waves, lower regularity, uniform estimates with respect to surface tension and etc have been made in \cite{AZ, ABZ, AM, BG, CHS1, CL, Coutand,Ebin,LD2,LN, PN, SZ, SWZ, zhang} and etc.  For more references, one may refer to the excellent survey  \cite{LD2}. For compressible inviscid  flow, the local-in-time well-posedness of smooth solutions was established for liquids  in  \cite{L2,35} (see also \cite{CHS} for the zero surface tension limits),  the study of the effects of heat-conductivity to fluid free surface can be found in \cite{LZ}.

However, only few results are available  for free boundary problems of  MHD equations.  Indeed,  magnetic fields are essential in many important physical situations (\cite{6, 26, 27}), for example,  solar flares  in astrophysics due to the coupling between magnetic and thermomechanical degrees of freedom for which magnetic
reconnection is thought to be the mechanism responsible for the conversion of
magnetic energy into heat and fluid motion (\cite{6}). Moreover,  interface problems in MHD are crucial to the theoretical and practical study of producing energy by fusion. In the study of the ideal MHD free boundary problems,  a priori estimates were derived in \cite{HLarma} with a bounded initial domain homeomorphic to a ball, provided that the size of the magnetic field to be invariant on the free boundary. A priori estimates for the low regularity solution of this problem were given in \cite{LZ19} for the bounded domain with small volume.  Ill-posedness was showed in \cite{HLip} for the two-dimensional problem if the condition \eqref{nbNPcond} was violated. A local existence result was established in \cite{GW} for which the detailed proof is given for an initial flat domain of the form $\mathbb{T}^2 \times (0, 1)$, where $\mathbb{T}^2$ is a two-dimensional period box in $x_1$ and  $x_2$.   The aim of the present paper is to study the ideal MHD free surface problem with a free surface being a closed curved surface with large curvature by the geometric approach motivated by \cite{CL}, \cite{L1} and \cite{L2}. 
For the special case where the magnetic field is zero on the free boundary and in vacuum,  the local existence and uniqueness of the free boundary problem of incompressible viscous-diffusive MHD flow in three-dimensional space with infinite and finite depth setting was proved in \cite{Lee} and \cite{Lee2} where also a local unique solution was obtained  for the  free boundary MHD without kinetic viscosity and magnetic diffusivity via zero kinetic viscosity-magnetic diffusivity limit. The convergence rates of inviscid limits for the free boundary problems of the three-dimensional incompressible MHD with or without surface tension was studied in \cite{CD19}, where the magnetic field is constant on the surface and outside of the unbounded domain. 
For the incompressible viscous MHD equations, a free boundary problem  in a simply connected domain of $\R^3$ was studied by a linearization technique and the construction of a sequence of successive approximations in \cite{PS10} with an irrotational condition for magnetic fields in a part of the domain. 
The plasma-vacuum system was investigated  in \cite{Hao17} where the a priori estimates were derived in a bounded domain. The well-posedness of the linearized plasma-vacuum interface problem in incompressible ideal MHD was studied in \cite{MTT14} in an unbounded plasma domain. For other related results of  MHD equations with free boundaries or interfaces, one may refer to \cite{chenwang, ST, Trakhinin, WangXin17, WangYu13}.

The rest of the paper is organized as follows. In Section \ref{sec.Lag}, we introduce the Lagrangian coordinates and reformulate the free boundary problem to a fixed boundary problem, and then linearize the equation.  We project the linearized equation onto the divergence-free vector fields in Section \ref{sec.proj}, and derive the lowest-order energy estimates in Section \ref{sec.lowestenergy}. In Section \ref{sec.homodata}, we change the linearized problem into the case of homogeneous initial data and an inhomogeneous term that vanishes to any order as time tends to zero. Next, we derive the a priori estimates of the linearized equation with homogeneous initial data in Section \ref{sec.apriori} including those of tangential derivatives and the curl. Then, we study a smoothed-out equation according to the normal operator and prove the existence of weak solution of it in Section \ref{sec.weak} and the existence of smooth solutions for the linearized equation in Section \ref{sec.exist}. We turn to the energy estimates of the original linearized equation with inhomogeneous initial data and an inhomogeneous term  in Section \ref{sec.inhomo}, and give the main result and its proof in Section \ref{sec.mainresult}. Finally, some preliminaries about the Lie derivative are given in the appendix.

\section{Lagrangian coordinates and the linearization of equations}\label{sec.Lag}
\subsection{Lagrangian reformulation}
In this section, we introduce the Lagrangian coordinates and reformulate the free boundary problem to a fixed boundary problem. 
Lagrangian coordinates $x=x(t,y)=f_t(y)$ are given by
\begin{align}\label{trajectory}
\frac{dx}{dt}=v(t,x(t,y)), \quad x(0,y)=f_0(y), \quad y\in \Omega.
\end{align}
Then $f_t: \Omega\to \Omega_t$ is a volume-preserving diffeomorphism because of $\dv v=0$, and the free boundary becomes fixed in the new $y$-coordinates. For simplicity, we take $f_0$ is the identity operator, that is, $x(0,y)=y$ and $\Omega$ is just the unit ball. For convenience, the letters $a,b,c,d,e$, and $f$ will refer to quantities in the Lagrangian frame, whereas the letters $i,j,k,l,m$, and $n$ will refer to ones in the Eulerian frame, e.g., $\D_a=\D/\D y^a$ and $\D_i=\D/\D x^i$.

Denote
\begin{align}
D_t=&\D_t+v^k\D_k,\quad 
 \D_k=\frac{\D}{\D x^k}=\frac{\D y^a}{\D x^k}\frac{\D}{\D y^a}.\label{Di}
\end{align}
Then, we get
\begin{align}
D_t\frac{\D x^i}{\D y^a}=\frac{\D D_t x^i}{\D y^a}=\frac{\D v^i}{\D y^a}=\frac{\D x^k}{\D y^a}\frac{\D v^i}{\D x^k},
\end{align}
and
\begin{align}\label{Dtyx}
D_t\frac{\D y^a}{\D x^i}=&D_t(\delta^a_b\frac{\D y^b}{\D x^i})=\frac{\D x^j}{\D y^b}\frac{\D y^a}{\D x^j}D_t\frac{\D y^b}{\D x^i}\\
=&\frac{\D y^a}{\D x^j}D_t(\frac{\D x^j}{\D y^b}\frac{\D y^b}{\D x^i})-\frac{\D y^a}{\D x^j}\frac{\D y^b}{\D x^i}D_t\frac{\D x^j}{\D y^b}\no\\
=&-\frac{\D y^a}{\D x^j}\frac{\D y^b}{\D x^i}\frac{\D x^k}{\D y^b}\frac{\D v^j}{\D x^k}\no\\
=&-\frac{\D y^a}{\D x^j}\frac{\D v^j}{\D x^i}.\no
\end{align}
From \eqref{mhd.2} and \eqref{Dtyx}, we have
\begin{align*}
D_t\left(H^i\frac{\D y^a}{\D x^i}\right)=&D_tH^i\frac{\D y^a}{\D x^i}+H^iD_t\frac{\D y^a}{\D x^i}
=H^j\D_jv^i\frac{\D y^a}{\D x^i}-H^i\D_iv^k\frac{\D y^a}{\D x^k}=0,
\end{align*}
which yields 
\begin{align*}
H^i(t,x(t,y))\frac{\D y^a}{\D x^i}= H^i(0,x(0,y))\left.\frac{\D y^a}{\D x^i}\right|_{t=0}=\bar{H}_0^i(y)\delta_i^a=\bar{H}_0^a(y),
\end{align*}
and
\begin{align}\label{Hasinitial}
H^j(t,x(t,y))=\bar{H}_0^a(y)\frac{\D x^j(t,y)}{\D y^a},
\end{align}
where $\bar{H}_0^a(y)=H^a_0(x(0,y))$.
Then, 
\begin{align*}
H^k\D_kH^i=\bar{H}_0^a\frac{\D x^k}{\D y^a}\frac{\D y^c}{\D x^k}\D_c(\bar{H}_0^b\frac{\D x^i}{\D y^b})=\bar{H}_0^a\D_a(\bar{H}_0^b\D_b x^i),
\end{align*}
For convenience, denote the differential operator 
$$B:=B^a(y)\frac{\D}{\D y^a}, \text{ with } B^a(y):=\sqrt{\mu}\bar{H}_0^a(y),$$ 
then \eqref{mhd.1}-\eqref{comoving} can be written as
\begin{align}\label{mhd1}
	\left\{\begin{aligned}
	&D_t^2x^i+\D_i P =  B^2 x^i, &&\text{ in } [0,T]\times\Omega,\\
	&\kappa:=\det\left(\frac{\D x}{\D y}\right)=1, &&\text{ in }  [0,T]\times\Omega,\\
	&P=0, &&\text{ on }  \Gamma,
	\end{aligned}\right.
\end{align}
where $P=P(t,y)=p(t,x(t,y))$, $\D_i$ is thought of as the differential operator in $y$ given in \eqref{Di} and $D_t$ is the time derivative.  The initial conditions read
\begin{align}
x|_{t=0}=y,\quad D_tx|_{t=0}=v_0,
\end{align}
satisfying the constraint $\dv v_0=0$. 
Taking the divergence of \eqref{mhd1} gives the Laplacian of $P$:
\begin{align}\label{RT}
\Delta P=-(\D_iD_t x^k)(\D_kD_t x^i)+ \D_i ( B^2 x^i).
\end{align}
The condition \eqref{nbNPcond} turns to be
\begin{align}
\nb_N P\ls -c_0<0, \text{ on } \Gamma,
\end{align}
where $N$ is the exterior unit normal to $\Gamma_t$ parametrized by $x(t,y)$.

\subsection{Linearization}
Now, we derive the linearized equations of \eqref{mhd1}. We assume that $(x(t,y)$, $P(t,y))$ is a given smooth solution of \eqref{mhd1} satisfying \eqref{RT} for $t\in[0,T]$.

Let $\delta$ be a variation with respect to some parameter $r$ in the Lagrangian coordinates:
\begin{align}
\delta=\left.\frac{\D}{\D r}\right|_{(t,y)=\mathrm{const}}.
\end{align}
We think of $x(t,y)$ and  $P(t,y)$ as depending on 
$r$ and differentiating with respect to $r$, say, $\bar{x}(t,y,r)$ and $\bar{P}(t,y,r)$ respectively. Namely, $(\bar{x},\bar{P})|_{r=0}=(x,P) $.
Differentiating \eqref{Di} and using the formula for the derivative of 
the inverse of a matrix, $\delta M^{-1}=-M^{-1}(\delta M)M^{-1}$, we get the commutator
\begin{align}\label{del.Di}
[\delta,\D_i]=-(\D_i \delta x^k)\D_k.
\end{align}
Let 
\begin{align}
(\delta x,\delta P)=\left.\left(\frac{\D \bar{x}}{\D r},\frac{\D \bar{P}}{\D r}\right)\right|_{r=0},
\end{align}
which satisfies $\dv \delta x=0$ and $\delta P|_{\Gamma}=0$.

From \eqref{mhd1} and \eqref{del.Di}, we get  by noticing $[D_t,\delta]=0$ and $[\delta,B]=0$
\begin{align}\label{1}
D_t^2\delta x^i=&-\delta \D_i P+  B^2 \delta x^i\\
=&(\D_i\delta x^k)\D_kP-\D_i\delta P+ B^2 \delta x^i.\no
\end{align}
From \eqref{mhd1} again, we have
\begin{align}
\D_iP=-D_t^2x^i+  B^2 x^i=-D_t v^i+  B^2 x^i,
\end{align}
and then
\begin{align}\label{2}
(\D_i\delta x^k)\D_kP=&\D_i(\delta x^k\D_kP)-\delta x^k\D_i\D_kP\\
=&\D_i(\delta x^k\D_kP)-\delta x^k\D_k(-D_t v^i+  B^2 x^i)\no\\
=&\D_i(\delta x^k\D_kP)+\delta x^k(\D_kD_t v^i- \D_k( B^2 x^i)).\no
\end{align}
It follows from \eqref{1} and \eqref{2} that
\begin{align}\label{3}
D_t^2\delta x^i
&+\D_i\delta P-\D_i(\delta x^k\D_kP)-\delta x^k(\D_kD_t v^i- \D_k( B^2 x^i))-  B^2 \delta x^i=0.
\end{align}

Now, we introduce new variables. Let 
\begin{align}
W^a=&\delta x^i \frac{\D y^a}{\D x^i}, \quad \delta x^i=W^b\frac{\D x^i}{\D y^b}, \quad q=\delta P. 
\end{align}
And recall
$$\D_i=\frac{\D y^a}{\D x^i}\D_a, \quad \D_a=\frac{\D x^i}{\D y^a}\D_i.$$

Let 
\begin{align}
g_{ab}=\delta_{ij}\frac{\D x^i}{\D y^a}\frac{\D x^j}{\D y^b}
\end{align}
be the metric $\delta_{ij}$ expressed in the Lagrangian coordinates. Let $g^{ab}$ be the inverse of $g_{ab}$, 
\begin{align}\label{gomega}
\dot{g}_{ab}=D_t g_{ab}=\frac{\D x^i}{\D y^a}\frac{\D x^k}{\D y^b}(\D_k v_i+\D_i v_k),  \text{ and } \omega_{ab}=\frac{\D x^i}{\D y^a}\frac{\D x^k}{\D y^b}(\D_iv_k-\D_k v_i)
\end{align}
be the time derivatives of the metric and the vorticity in the Lagrangian coordinates, respectively. It follows that
\begin{align}\label{Dv}
\frac{\D x^i}{\D y^a}\frac{\D x^k}{\D y^b}\D_k v_i=\frac{1}{2}(\dot{g}_{ab}-\omega_{ab}).
\end{align}

Multiplying \eqref{3} by $\frac{\D x^i}{\D y^a}$ and summing over $i$, we obtain
\begin{align}\label{4}
\delta_{il}\frac{\D x^l}{\D y^a}D_t^2\delta x^i
&-\D_a(W^c\D_cP)-W^b\delta_{il}\frac{\D x^l}{\D y^a}\D_bD_t v^i+\D_a q\no\\
&+ \delta_{il}\D_a x^lW^d\D_d( B^2 x^i) - \delta_{il}\D_a x^lB^2( W^c\D_c x^i)=0.
\end{align}

The first term in \eqref{4} can be written as
\begin{align*}
\delta_{il}\frac{\D x^l}{\D y^a}D_t^2\delta x^i
=&\delta_{il}\frac{\D x^l}{\D y^a}D_t^2(W^b\frac{\D x^i}{\D y^b})\\
=&\delta_{il}\frac{\D x^l}{\D y^a}D_t(D_tW^b\frac{\D x^i}{\D y^b}+W^bD_t\frac{\D x^i}{\D y^b})\no\\
=&g_{ab}D_t^2 W^b+2\delta_{il}\frac{\D x^l}{\D y^a}D_tW^bD_t\frac{\D x^i}{\D y^b}+\delta_{il}\frac{\D x^l}{\D y^a}W^bD_t^2\frac{\D x^i}{\D y^b}\no\\
=&g_{ab}D_t^2 W^b+2\frac{\D x^l}{\D y^a}\frac{\D x^k}{\D y^b}\D_kv_lD_tW^b+\delta_{il}\frac{\D x^l}{\D y^a}W^b\D_b D_tv^i\no\\
=&g_{ab}D_t^2 W^b+(\dot{g}_{ab}-\omega_{ab})D_tW^b+\frac{\D x^i}{\D y^a}W^b\D_b D_tv_i.
\end{align*}
It follows from  \eqref{4} that 
\begin{align}\label{lineqform}
g_{ab}D_t^2 W^b&+(\dot{g}_{ab}-\omega_{ab})D_tW^b-\D_a(W^c\D_cP)+\D_a q\no\\
&+ \delta_{il}\D_a x^lW^d\D_d( B^2 x^i) - \delta_{il}\D_a x^lB^2( W^c\D_c x^i)=0,
\end{align}
which yields by acting $g^{da}$
\begin{align}\label{5}
D_t^2 W^d&+g^{da}(\dot{g}_{ab}-\omega_{ab})D_tW^b-g^{da}\D_a(W^c\D_cP)+g^{da}\D_a q\no\\
&+ g^{da}\delta_{il}\D_a x^l[W^c\D_c( B^2 x^i) - B^2( W^c\D_c x^i)]=0.
\end{align}

From \eqref{4}, we see that the energies will include $\norm{BW}^2$. But it is very complicated due to $\dv(BW)\neq 0$. Indeed, we can regard $B$ as a tangential derivative since $B=B^a\D_a$ is independent of time and $\D_a B^a=0$. Thus, we can use the Lie derivative corresponding to $B$ given by
\begin{align}\label{Lie.B}
\Lie_BW^a=BW^a-\D_bB^a W^b,
\end{align}
which is divergence-free due to $\dv\Lie_B W=\D_a(B^b\D_bW^a-\D_bB^aW^b)=0$ if $\dv W=0$. We also have
\begin{align}\label{Lie.BDx}
\Lie_B \D_c x^i=B\D_c x^i+\D_c B^d\D_d x^i.
\end{align}
For more details and properties of Lie derivatives, one can see Appendix \ref{App.Lie}.

From \eqref{Lie.B}, it follows that
\begin{align*}
\Lie_B^2W^a=&\Lie_B (BW^a-(\D_cB^a)W^c)\\
=&B(BW^a-(\D_cB^a)W^c)-(\D_cB^a)\Lie_BW^c\\
=&B^2W^a-B^d\D_d\D_cB^a W^c-(\D_cB^a)BW^c-(\D_cB^a)\Lie_BW^c\\
=&B^2W^a-B^c\D_c\D_dB^a W^d-(\D_cB^a)(\D_dB^cW^d)-2(\D_cB^a)\Lie_BW^c\\
=&B^2W^a-2(\D_cB^a)\Lie_BW^c-W^d\D_d(BB^a),
\end{align*}
and then
\begin{align*}
B^2(\D_cx^iW^c)=&\D_cx^iB^2W^c+2B\D_cx^iBW^c+B^2\D_cx^iW^c\\
=&\D_cx^i(\Lie_B^2W^c+2(\D_dB^c)\Lie_BW^d+W^d\D_d(BB^c))\\
&+2B\D_cx^i(\Lie_BW^c+(\D_dB^c)W^d)+B^2\D_cx^iW^c\\
=&\D_cx^i\Lie_B^2W^c+2(\D_dBx^i)\Lie_BW^d+\D_cx^iW^d\D_d(BB^c)\\
&+2B\D_cx^i\D_dB^cW^d+B^2\D_dx^iW^d\\
=&\D_cx^i\Lie_B^2W^c+2(\D_cBx^i)\Lie_BW^c+W^c\D_c(B^2x^i)\\
=&\D_cx^i\Lie_B^2W^c+2((\D_cB^b)\D_bx^i+B\D_cx^i)\Lie_BW^c+W^c\D_c(B^2x^i).
\end{align*}
Hence, we get by \eqref{Lie.BDx}
\begin{align*}
&g^{da}\delta_{il}\D_a x^l[W^c\D_c( B^2 x^i) - B^2( W^c\D_c x^i)]\\
=&-g^{da}\delta_{il}\D_a x^l[\D_cx^i\Lie_B^2W^c+2(\D_cB^b\D_bx^i+B\D_cx^i)\Lie_BW^c]\\
=&-\Lie_B^2W^d-2\D_cB^d\Lie_BW^c-2g^{da}\delta_{il}\D_a x^lB\D_cx^i\Lie_BW^c\\
=&-\Lie_B^2W^d-2g^{da}\delta_{il}\D_a x^l\Lie_B\D_cx^i\Lie_BW^c.
\end{align*}

We introduce some new notations.  Denote
\begin{align}
\dot{W}^a(t,y):=D_tW^a(t,y),\quad \ddot{W}^a:=D_t^2 W^a.
\end{align}
Since $q=\delta P$, one has $q|_{\Gamma}=0$.
Thus, from \eqref{5}  and \eqref{mhd.3},  we have the following system
\begin{equation}\label{m}
	\left\{\begin{aligned}
&\ddot{W}^d-\Lie_B^2W^d+g^{da}\D_a q-g^{da}\D_a(W^c\D_cP)+g^{da}(\dot{g}_{ab}-\omega_{ab})\dot{W}^b\\
&\qquad-2g^{da}\delta_{il}\D_a x^l\Lie_B\D_cx^i\Lie_BW^c =0,\\
&\dv W=\kappa^{-1}\D_a(\kappa W^a)=0, \\
&q|_{\Gamma}=0,\\
&W|_{t=0}=W_0,\; \dot{W}|_{t=0}=W_1,
	\end{aligned}\right.
\end{equation}
where $\dv W_0=\dv W_1=0$. 

We can express \eqref{m} in one equation since $q=\delta P$ is determined as a functional of $W$ and $\dot{W}$. Thus, we derive an elliptic equation for $q$.

\subsection{The equation of $\Delta q$}

In order to get the equation of $\Delta q$, we have to derive $\dv \ddot{W}$ first. Denote 
\begin{align*}
u^a:=\frac{\D y^a}{\D x^i}v^i, \text{ and } u_a=g_{ab}u^b.
\end{align*}
From  $\dv W=0$, it follows that $\dv\ddot{W}=0$. 
Thus, taking the divergence of \eqref{m}, we have,
\begin{align}\label{Deltaq}
	\left\{\begin{aligned}
&\Delta q=\D_d\big(g^{da}\D_a(W^c\D_cP)-g^{da}(\dot{g}_{ab}-\omega_{ab})\dot{W}^b+2g^{da}\delta_{il}\D_a x^l\Lie_B\D_cx^i\Lie_BW^c\big),\\
&q|_{\Gamma}=0,
	\end{aligned}\right.
\end{align}
since $\dv \Lie_B^2W=0$.

We separate $q$ into four parts:
$$q=\sum_{i=1}^4 q_i,$$
where $q_i$'s satisfy the following Dirichlet problems of Poisson equations:
\begin{align*}
	\left\{\begin{aligned}
&\Delta q_1=\Delta(W^c\D_cP), &q_1|_\Gamma =0,\\
&\Delta q_2=-\D_d(g^{da}\dot{g}_{ab}\dot{W}^b),  &q_2|_\Gamma =0,\\
&\Delta q_3=\D_d(g^{da}\omega_{ab}\dot{W}^b), &q_3|_\Gamma =0,\\
&\Delta q_4=2\D_d(g^{da}\delta_{il}\D_a x^l\Lie_B\D_cx^i\Lie_BW^c), &q_4|_\Gamma =0.
	\end{aligned}\right.
\end{align*}

Then, we can write \eqref{m} as
\begin{align}\label{m1}
	L_1W:=\ddot{W}-\Lie_B^2W+\A W+\dot{\mc{G}}\dot{W}-\mc{C}\dot{W}+\mc{X}\Lie_BW=0,
\end{align}
where
	\begin{align}
	&\A W^d:=-g^{da}\D_a(\D_cP W^c-q_1),\label{oper.A}\\	
	&\dot{\mc{G}}\dot{W}^d:=g^{da}(\dot{g}_{ab}\dot{W}^b+\D_a q_2),\label{oper.dotG}\\
	&\mc{C}\dot{W}^d:=g^{da}(\omega_{ab}\dot{W}^b-\D_a q_3),\label{oper.C}\\
	&\mc{X}\Lie_BW^d:= -2g^{da}\delta_{il}\D_a x^l\Lie_B\D_cx^i\Lie_BW^c+g^{da}\D_a q_4.\label{oper.X}
	\end{align}

\section{The projection onto divergence-free vector field} \label{sec.proj}

In this section, we recall some definitions on the projection onto divergence-free vector field, one can see \cite{L1} for details.

Let $\Pdv$ be the orthogonal projection onto divergence-free vector fields in the inner product
\begin{align*}
\langle W,U\rangle =\int_\Omega g_{ab}W^a U^b dy.
\end{align*}
Then,
\begin{align*}
\Pdv U^a=&U^a-g^{ab}\D_b q,\\
\Delta q=&\D_a(g^{ab}\D_b q)=\dv U=\D_aU^a, \quad q|_{\Gamma}=0,
\end{align*}
because of $g_{ab}g^{bc}=\delta_a^c$ and 
\begin{align*}
\langle W, (\Id-\Pdv)U\rangle =&\int_\Omega g_{ab}W^a g^{bc}\D_c q dy\\
=&\int_{\Gamma} W^aN_a qdS-\int_\Omega q\dv W  dy=0, \text{ if } \dv W=0,\no
\end{align*}
where $N_a$ is the exterior unit conormal and $dS$ is the surface measure. In addition, for the function vanishing on the boundary, the projection of its gradient vanishes:
\begin{align*}
\Pdv (g^{ab}\D_b f)=0, \text{ if } f|_{\Gamma}=0.
\end{align*}

Denote $\norm{W}:=\langle W,W\rangle^{1/2}$. It is clear that 
\begin{align*}
\norm{\Pdv U}\ls \norm{U}, \quad \norm{(\Id-\Pdv)U}\ls \norm{U}.
\end{align*}
The projection is continuous on the Sobolev spaces $H^r(\Omega)$ if the metric is sufficiently regular: 
\begin{align*}
\norm{\Pdv U}_{H^r(\Omega)}\ls C_r \norm{U}_{H^r(\Omega)}.
\end{align*}
Furthermore, if the metric also depends smoothly on time $t$, then
\begin{align}\label{PdvHr}
\sum_{j=0}^k \norm{D_t^j \Pdv U}_{H^r(\Omega)}\ls C_{r,k}\sum_{j=0}^k \norm{D_t^j U}_{H^r(\Omega)}.
\end{align}

For functions $f$ vanishing on the boundary, we define operators on divergence-free vector fields ($\D_a W^a=0$)
\begin{align}\label{Afdef}
\A_f W^a=\Pdv\left(-g^{ab}\D_b(W^c\D_c f)\right).
\end{align}
$\A_f$ is symmetric, i.e., $\langle U,\A_fW\rangle=\langle \A_fU,W\rangle$.

Since $P$ is the total pressure, the normal operator $\A $ in \eqref{oper.A} is 
\begin{align*}
\A =\A_P\gs 0, \quad \langle W,\A W\rangle \gs 0 \text{ if } \nb_N P|_{\Gamma}\ls 0,
\end{align*} 
which is true in view of the condition \eqref{nbNPcond}. In fact,
\begin{align}\label{WAW}
\langle W,\A W\rangle =&\int_\Omega g_{ab}W^b \A W^a dy\\
=&-\int_\Omega g_{ab}W^b g^{ad}\D_d(W^c\D_cP -q_1) dy \no\\
=&-\int_\Omega \D_d(W^d(W^c\D_cP -q_1)) dy\no\\
=&-\int_{\Gamma} N_dW^dW^c\D_cP dS+\int_{\Gamma} N_dW^dq_1 dS\no\\
=&-\int_{\Gamma} N_dW^dg^{ac}\D_cPW_a dS\no\\
=&-\int_{\Gamma} N_dW^d\gamma^{ac}\D_cPW_a dS-\int_{\Gamma} N_dW^dN^aN^c\D_cPW_a dS\no\\
=&-\int_{\Gamma} |N\cdot W|^2\nb_NP dS\gs 0,\no
\end{align}
due to $P=0$ and $q_1=0$ on the boundary $\Gamma$.

By the definition in \eqref{Afdef}, we have
\begin{align*}
\A_{fP}W^a=&-g^{ab}\D_b(W^c\D_c(fP))+g^{ab}\D_b q,\\
\Delta q=&\Delta (W^c\D_c(fP)), \quad q|_{\Gamma}=0.
\end{align*}
Then, for the divergence-free vector field $U$, it follows that
\begin{align}
\langle U,\A_{fP}W\rangle=&\int_\Omega g_{ab}U^a\A_{fP}W^b dy\no\\
=&-\int_{\Omega}(U^d\D_d(W^c\D_c(fP))+U^d\D_d q)dy\no\\
=&-\int_{\Gamma} U_NW^c\D_c(fP)dS\no\\
=&-\int_{\Gamma} U_NW_N\nb_N(fP)dS\label{UAW}\\
=&-\int_{\Gamma} U_NW_Nf\nb_NPdS,\no
\end{align}
since $\nb_N(fP)=f\nb_NP$ and $fP=0$ on the boundary, where $U_N=N_aU^a=N\cdot U$.
From the Cauchy-Schwarz inequality and the identity \eqref{WAW}, it follows that
\begin{align}\label{est.A}
|\langle U,\A_{fP}W\rangle|\ls &\norm{f}_{L^\infty(\Gamma)}\left(\int_\Gamma |U_N|^2(-\nb_NP)dS\right)^{1/2}\left(\int_\Gamma |W_N|^2(-\nb_NP)dS\right)^{1/2}\no\\
=&\norm{f}_{L^\infty(\Gamma)}\langle U,\A U\rangle^{1/2}\langle W,\A W\rangle^{1/2}.
\end{align}
In addition, since $P$ vanishes on the boundary, so does $\dot{P}=D_tP$, and then we can define 
\begin{align*}
\dot{\A }=\A_{\dot{P}},\quad \dot{\A }W^a=-g^{ab}\D_b(W^c\D_c\dot{P}-q),\quad \Delta q=\Delta(W^c\D_c\dot{P}),\; q|_\Gamma=0,
\end{align*}
which satisfies by \eqref{UAW}
\begin{align}\label{dotA}
|\langle W,\dot{\A }W\rangle |=\abs{-\int_{\Gamma} |W_N|^2\nb_N\dot{P}dS}\ls \lnorm{\frac{\nb_N\dot{P}}{\nb_N P}}_{L^\infty(\Gamma)}\langle W,\A W\rangle.
\end{align}
In fact, $\dot{\A }$ is the time derivatives of the operator $\A $, considered as an operator with values in the $1$-forms.

For $2$-forms $\alpha$, we define bounded projected multiplication operators, as in \cite{L1}, given by
\begin{align}\label{Malpha}
\M_\alpha W^a=\Pdv (g^{ab}\alpha_{bc}W^c), \quad \norm{\M_\alpha W}\ls \norm{\alpha}_{L^\infty(\Omega)}\norm{W}.
\end{align}
In particular, the operators in \eqref{oper.dotG} and \eqref{oper.C} are bounded, projected multiplication operators:
\begin{align}\label{op}
\mc{G}=\M_g, \quad \mc{C}=\M_\omega, \quad \dot{\mc{G}}=\M_{\dot{g}},
\end{align}
for the metric $g$, the vorticity $\omega$, and the time derivative of the metric $\dot{g}$.

\section{The lowest-order energy estimates} \label{sec.lowestenergy}

Now, we derive the energy estimates for the linearized equations
\begin{align}\label{lineareq}
	L_1W=\ddot{W}-\Lie_B^2W+\A W+\dot{\mc{G}}\dot{W}-\mc{C}\dot{W}+\mc{X}\Lie_BW =F,
\end{align}
where $F$ is divergence-free.

We first compute the inner product of \eqref{lineareq} with $\dot{W}$ and $W$. Since $$D_t(g_{ab}\dot{W}^a\dot{W}^b)=\dot{g}_{ab}\dot{W}^a\dot{W}^b+2g_{ab}\dot{W}^a\ddot{W}^b,$$
 we get
\begin{align}\label{DtWt}
\langle \ddot{W}, \dot{W}\rangle=\frac{1}{2}\frac{d}{dt}\langle \dot{W}, \dot{W}\rangle-\frac{1}{2}\langle \dot{W}, \dot{\mc{G}}\dot{W}\rangle,
\end{align}
where $\dot{G}$ is given by \eqref{op}. From the symmetry of $\A $, it follows 
\begin{align*}
 \langle \A W,\dot{W}\rangle=\frac{1}{2}\frac{d}{dt}\langle \A W,W\rangle-\frac{1}{2}\langle \dot{\A }W,W\rangle,
\end{align*}
where $\dot{\A }W^a=\A_{\dot{P}}W^a$ is defined by \eqref{Afdef} with $f=\dot{P}=D_tP$. In addition,
\begin{align*}
\frac{1}{2}\frac{d}{dt}\langle W, W\rangle=\langle W,\dot{W}\rangle.
\end{align*}
Thus,
\begin{align}\label{DtAW}
\frac{1}{2}\frac{d}{dt}\langle (\A +I)W,W\rangle=\langle (\A +I)W,\dot{W}\rangle+\frac{1}{2}\langle \dot{\A }W,W\rangle.
\end{align}

We also have
\begin{align}\label{DtBW}
-\langle \Lie_B^2W,\dot{W}\rangle=&-\int_{\Omega} g_{ab}\Lie_B^2W^a\dot{W}^b dy\\
=&-\int_{\Omega}\Lie_B (g_{ab}\Lie_BW^a\dot{W}^b) dy+\int_{\Omega} g_{ab}\Lie_BW^a\Lie_B\dot{W}^b dy\no\\
&+\int_{\Omega} (\Lie_Bg_{ab})\Lie_BW^a\dot{W}^b dy\no\\
=&\frac{1}{2}\frac{d}{dt}\int_{\Omega}|\Lie_BW|^2dy-\frac{1}{2}\int_{\Omega} \dot{g}_{ab}\Lie_BW^a\Lie_BW^bdy\no\\
&+\int_{\Omega} \delta_{il}(\D_b x^l\Lie_B\D_ax^i+\Lie_B\D_b x^l\D_ax^i)\Lie_BW^a\dot{W}^b dy.\no
\end{align}

Since $q_4=0$ on $\Gamma$, it yields
\begin{align*}
\langle \mc{X}\Lie_BW,\dot{W}\rangle
=&-2\int_{\Omega} \delta_{il}\D_b x^l\Lie_B\D_ax^i\Lie_BW^a\dot{W}^b dy.
\end{align*}

Hence, we can define the energy as
\begin{align}\label{energy}
E_0^2(t)=E(t)=\langle \dot{W}, \dot{W}\rangle+\langle (\A +I)W,W\rangle+\langle \Lie_BW,\Lie_BW\rangle.
\end{align}
Then, we have the following energy estimates.

\begin{proposition}\label{prop.E0}
	Let
	\begin{align*}
	n_0(t)=&\frac{1}{2}\left(1+\lnorm{\frac{\nb_N\dot{P}}{\nb_N P}}_{L^\infty(\Gamma)}+2\norm{\dot{g}}_{L^\infty(\Omega)}+2\norm{\D x}_{L^\infty(\Omega)}\norm{\Lie_B\D x}_{L^\infty(\Omega)}\right).
	\end{align*}
It holds	
	\begin{align}\label{E0}
	E_0(t)\ls e^{\int_0^tn_0(\tau)d\tau}\left(E_0(0)+\int_0^t\norm{F(s)}e^{-\int_0^sn_0(\tau)d\tau}ds\right).
	\end{align}
\end{proposition}

\begin{proof}
Due to the antisymmetry of $\omega$, we have $\langle \mc{C}\dot{W},\dot{W}\rangle=0$. Then, we get
\begin{align*}
\frac{1}{2}\dot{E}(t)=&\langle -\frac{1}{2}\dot{\mc{G}}\dot{W}+F, \dot{W}\rangle+\langle W,\dot{W}\rangle+\frac{1}{2}\langle \dot{\A }W,W\rangle+\frac{1}{2}\langle\dot{\mc{G}}\Lie_BW,\Lie_BW\rangle\\
&+\int_{\Omega} \delta_{il}(\Lie_B\D_b x^l\D_ax^i-\D_b x^l\Lie_B\D_ax^i)\Lie_BW^a\dot{W}^b dy.
\end{align*}
Thus,  we obtain
\begin{align}\label{dotE0}
|\dot{E_0}|\ls &n_0(t)E_0+\norm{F},
\end{align}
which yields the desired estimates.
\end{proof}

\section{Turning initial data into an inhomogeneous divergence-free term} \label{sec.homodata}

In this section, we want to change the initial value problem \eqref{lineareq} and \eqref{m}, i.e.,
\begin{subequations}\label{IVP}
	\begin{align}
	&L_1W=\ddot{W}-\Lie_B^2W+\A W+\dot{\mc{G}}\dot{W}-\mc{C}\dot{W}+\mc{X}\Lie_BW=F,\label{eq.IVP}\\
	&W|_{t=0}=W_0,\; \dot{W}|_{t=0}=W_1,\label{data}
	\end{align}
\end{subequations}
into the case of homogeneous initial data and an inhomogeneous term $F$ that vanishes to any order as $t\to 0$. As in \cite{L1}, we can achieve it by subtracting off a power series solution in $t$ to \eqref{IVP}:
\begin{align}\label{W0r}
W_{0r}^a(t,y)=\sum_{s=0}^{r+2}\frac{t^s}{s!}W_s^a(y).
\end{align}
It is clear that $W_{0r}$ is divergence-free if so does $W_s$. Here $W_0$ and $W_1$ are the initial data given in \eqref{data}, $W_2$ is obtained from \eqref{IVP} at $t=0$:
\begin{align*}
W_2=\ddot{W}(0)=F(0)+\Lie_B^2W_0-\A (0)W_0-\dot{\mc{G}}(0)W_1+\mc{C}(0)W_1-\mc{X}(0)\Lie_BW_0.
\end{align*}
The higher-order terms can be obtained by  differentiating the equation with respect to time first and then taking the value at $t=0$. Indeed, by doing so, we can obtain an expression
\begin{align*}
D_t^{k+2}W=M_k(W,D_tW,\cdots,D_t^{k+1}W)+D_t^kF,
\end{align*}
from which we inductively define
\begin{align*}
W_{k+2}=M_k(W_0,W_1,\cdots,W_{k+1})|_{t=0}+D_t^kF|_{t=0},
\end{align*}
where $M_k$ is some linear operator of order at most $1$ and that is all we need to derive. Next, we calculate the explicit form of $M_k$ as a simple model case, since we will do similar derivations later on  for other operators.

It is convenient to differentiate the corresponding operator with values in $1$-forms, so we denote
\begin{align}\label{1formeq}
\underline{L}_1W_a:=&g_{ab}L_1W^b\\
=&g_{ab}\ddot{W}^b-g_{ab}\Lie_B^2W^b+\D_a q-\D_a(\D_cPW^c)+(\dot{g}_{ab}-\omega_{ab})\dot{W}^b\no\\
&-2\delta_{il}\D_a x^l\Lie_B\D_cx^i\Lie_BW^c=g_{ab}F^b,\no
\end{align}
where $q$ is chosen such that the last terms are divergence-free, and afterwards project the result to the divergence-free vector fields. Denote
\begin{align*}
q^s=D_t^s q,\; P^s=D_t^sP,\; g_{ab}^s=D_t^s g_{ab},\; \omega_{ab}^s=D_t^s \omega_{ab},\; F_s=D_t^s F.
\end{align*}
Applying the differential operator $D_t^r$ to \eqref{1formeq} and restricting $t$ to $0$, we get
\begin{align*}
&\sum_{s=0}^r\bic_r^s\left(g_{ab}^{r-s}W_{s+2}^b-g_{ab}^{r-s}\Lie_B^2W_{s}^b-\D_a(\D_cP^{r-s}W^c_s)\right)+\D_a q^r\no\\
&+\sum_{s=0}^r \bic_r^s\left(g_{ab}^{r-s+1}-\omega_{ab}^{r-s}\right)W_{s+1}^b\no\\
&-2\sum_{s=0}^r\bic_r^s\sum_{s_1=0}^{r-s}\bic_{r-s}^{s_1}\delta_{il}\D_a (D_t^{r-s-s_1}x^l)\Lie_B\D_c(  D_t^{s_1} x^i)\Lie_BW^c_s=\sum_{s=0}^r\bic_r^sg_{ab}^{r-s}F^b_s.
\end{align*}
Then, we need to project all terms onto divergence-free vector fields. Let
\begin{align*}
	\left\{\begin{aligned}
	&\A_sW^d:=\Pdv(-g^{da}\D_a(\D_cP^s W^c)),\\	&\mc{G}_sW^d:=\Pdv(g^{da}{g}_{ab}^sW^b),\\
	&\mc{C}_sW^d:=\Pdv(g^{da}\omega_{ab}^sW^b),\\
	&\mc{X}_s\Lie_BW^d:= -2 \Pdv\left(g^{da}\sum_{s_1=0}^s\bic_s^{s_1}\delta_{il}\D_a (D_t^{s-s_1}x^l)\Lie_B\D_c (D_t^{s_1}x^i)\Lie_B W^c\right).
	\end{aligned}\right.
\end{align*}
We have
\begin{align*}
W_{r+2}=&-\sum_{s=0}^{r-1}\bic_r^s\mc{G}_{r-s}W_{s+2}+\sum_{s=0}^r \bic_r^s\left(\mc{G}_{r-s}\Lie_B^2W_s-\A_{r-s}W_s-\mc{G}_{r-s+1}W_{s+1}\right)\no\\
&+\sum_{s=0}^r\bic_r^s\left(\mc{C}_{r-s}W_{s+1}-\mc{X}_{r-s}\Lie_BW_s+\mc{G}_{r-s}F_s\right),
\end{align*}
which inductively defines $W_{r+2}$ from $W_0,W_1, \cdots, W_{r+1}$. 

By the definition of $W_{0r}$ in \eqref{W0r}, it is obvious that 
\begin{align*}
D_t^s(L_1W_{0r}-F)|_{t=0}=0\; \text{ for } s\ls r,\quad W_{0r}|_{t=0}=W_0,\; \dot{W}_{0r}|_{t=0}=W_1.
\end{align*}
Thus, we reduces \eqref{IVP} to the desired case of vanishing initial data and an inhomogeneous term that vanishes to any order $r$ as $t\to 0$ by replacing $W$ by $W-W_{0r}$ and $F$ by $F-L_1W_{0r}$.

If the initial data are smooth, as similar as in \cite{L1}, we can also construct a smooth approximate solution $\tilde{W}$ that satisfies the equation to all orders as $t\to 0$. We can realize it by multiplying the $k$-th term in \eqref{W0r} by a smooth cutoff function $\chi(t/\eps_k)$  and summing up the infinite series where $\chi(s)=1$ for $|s|\ls \frac{1}{2}$ and $\chi(s)=0$ for $|s|\gs 1$.  If we take $(\norm{\tilde{W}_k}_k+1)\eps_k\ls \frac{1}{2}$, then the sequence $\eps_k>0$ can be chosen so small that the series converges in $C^m([0,T],H^m)$ for any $m$.

\section{A priori estimates of the linearized equation with homogeneous initial data} \label{sec.apriori}

\subsection{The estimates of the one more order derivatives for the linearized equation}
We take the time derivative to \eqref{1formeq}
to get 
\begin{align*}
&g_{ab}\dddot{W}^b-g_{ab}\Lie_B^2\dot{W}^b-\D_a(\dot{W}^c\D_cP)-\omega_{ab}\ddot{W}^b+\D_a \dot{q}\no\\
=&-2\dot{g}_{ab}\ddot{W}^b+\dot{g}_{ab}\Lie_B^2W^b+\D_a(W^c\D_c\dot{P})-(\ddot{g}_{ab}-\dot{\omega}_{ab})\dot{W}^b\no\\
&+2D_t(\delta_{il}\D_a x^l\Lie_B\D_cx^i)\Lie_BW^c+2\delta_{il}\D_a x^l\Lie_B\D_cx^i\Lie_B\dot{W}^c+\dot{g}_{ab}F^b+g_{ab}\dot{F}^b.
\end{align*}

Similar to \eqref{DtWt} and \eqref{DtAW} it holds
\begin{align*}
\langle \dddot{W}, \ddot{W}\rangle=\frac{1}{2}\frac{d}{dt}\langle \ddot{W}, \ddot{W}\rangle-\frac{1}{2}\langle \ddot{W}, \dot{\mc{G}}\ddot{W}\rangle,
\end{align*}
and
\begin{align*}
\frac{1}{2}\frac{d}{dt}\langle \A \dot{W},\dot{W}\rangle=\langle \A \dot{W},\ddot{W}\rangle+\frac{1}{2}\langle \dot{\A }\dot{W},\dot{W}\rangle.
\end{align*}
In view of \eqref{DtBW}, we have
\begin{align*}
-\langle \Lie_B^2\dot{W},\ddot{W}\rangle=&\frac{1}{2}\frac{d}{dt}\int_{\Omega}|\Lie_B\dot{W}|^2dy-\frac{1}{2}\int_{\Omega} \dot{g}_{ab}\Lie_B\dot{W}^a\Lie_B\dot{W}^bdy\no\\
&+\int_{\Omega} (\Lie_Bg_{ab})\Lie_B\dot{W}^a\ddot{W}^b dy.
\end{align*}

 Let
\begin{align}
E_{D_t}=E(D_tW)=\langle \ddot{W}, \ddot{W}\rangle +\langle \dot{W},\A \dot{W}\rangle+\langle \Lie_B\dot{W},\Lie_B\dot{W}\rangle.
\end{align}
Then, similarly, from  the antisymmetry of $\dot{\omega}$, we have 
\begin{subequations}\label{ED}
	\begin{align}
	\dot{E}_{D_t}
	=&2\langle \dot{F}+\dot{\mc{G}}F, \ddot{W}\rangle-2\langle\ddot{\mc{G}}\dot{W}, \ddot{W}\rangle+2\langle\dot{\mc{C}}\dot{W}, \ddot{W}\rangle-3\langle\dot{\mc{G}}\ddot{W}, \ddot{W}\rangle+\langle \dot{W},\dot{\A }\dot{W}\rangle\label{ED.1}\\
	&+\langle \dot{\mc{G}}\Lie_B\dot{W},\Lie_B\dot{W}\rangle+4\langle D_t(\delta_{il}\D_a x^l\Lie_B\D_cx^i)\Lie_BW^c, \ddot{W}^a\rangle \label{ED.2}\\
	&+2\langle\delta_{il}(\D_a x^l\Lie_B\D_cx^i-\Lie_B\D_a x^l\D_cx^i)\Lie_B\dot{W}^c, \ddot{W}^a\rangle\label{ED.3}\\
	&-2\langle\dot{\A }W, \ddot{W}\rangle.\label{ED.4}
	\end{align}
\end{subequations}
Thus,  we get
\begin{align*}
|\eqref{ED.3}|\ls 2\norm{\D x}_{L^\infty(\Omega)}\norm{\Lie_B\D x }_{L^\infty(\Omega)}E_{D_t}.
\end{align*}

By Cauchy-Schwartz' inequalities, we get
\begin{align*}
|\eqref{ED.1}|\ls& 2(\norm{\dot{F}}+\norm{\dot{g}}_{L^\infty(\Omega)}\norm{F})E_{D_t}^{1/2}+2(\norm{\ddot{g}}_{L^\infty(\Omega)}+\norm{\dot{\omega}}_{L^\infty(\Omega)})E_{D_t}^{1/2}E_0\\
&+\left(3\norm{\dot{g}}_{L^\infty(\Omega)}+\lnorm{\frac{\nb_N \dot{P}}{\nb_NP}}_{L^\infty(\Omega)}\right)E_{D_t},
\end{align*}
and
\begin{align*}
|\eqref{ED.2}|\ls& \norm{\dot{g}}_{L^\infty(\Omega)}E_{D_t}+4\norm{D_t(\delta_{il}\D x^l\Lie_B\D x^i)}_{L^\infty(\Omega)}E_0E_{D_t}^{1/2}.
\end{align*}

Now, it remainders to deal with the term $\langle\dot{\A }W, \ddot{W}\rangle$ in \eqref{ED.4}. From \eqref{est.A}, it follows that
\begin{align}\label{est.At}
|\langle\dot{\A }W, \ddot{W}\rangle| \ls \lnorm{\frac{\nb_N \dot{P}}{\nb_N P}}_{L^\infty(\Gamma)}\langle W, \A W\rangle^{1/2}\langle \ddot{W},\A \ddot{W} \rangle^{1/2}.
\end{align} 
But this does not imply that the norm of $\dot{\A }$ is bounded by the norm of $\A $ because $\langle \ddot{W},\A \ddot{W} \rangle$ is one more order derivative than the considering energies.
However, we have
\begin{align*}
\langle\dot{\A }W, \ddot{W}\rangle=\frac{d}{dt}\langle\dot{\A }W, \dot{W}\rangle-\langle\ddot{\A }W, \dot{W}\rangle-\langle\dot{\A }\dot{W}, \dot{W}\rangle,
\end{align*}
in which the last two terms can be bounded by $E_{D_t}$ and $E_0$. Thus, we have to deal with this term in an indirect way, by including them in the energies and using \eqref{est.At}.
Let
\begin{align*}
D_{D_t}=2\langle\dot{\A }W,\dot{W}\rangle,
\end{align*}
then we get
\begin{align*}
\dot{D}_{D_t}=&2\langle\ddot{\A }W,\dot{W}\rangle+2\langle\dot{\A }\dot{W},\dot{W}\rangle+2\langle\dot{\A }W,\ddot{W}\rangle,
\end{align*}
and 
\begin{align*}
|\dot{D}_{D_t}-2\langle\dot{\A }W,\ddot{W}\rangle|\ls & 2\lnorm{\frac{\nb_N\ddot{P}}{\nb_NP}}_{L^\infty(\Omega)}E_0E_{D_t}^{1/2}+2\lnorm{\frac{\nb_N\dot{P}}{\nb_NP}}_{L^\infty(\Omega)}E_{D_t},\\
|D_{D_t}|\ls& 2\lnorm{\frac{\nb_N\dot{P}}{\nb_NP}}_{L^\infty(\Omega)}E_0E_{D_t}^{1/2}.
\end{align*}

Denote
	\begin{align*}
\bar{n}_1(t)=&\lnorm{\frac{\nb_N\dot{P}}{\nb_NP}}_{L^\infty(\Omega)}, \\
n_1(t)=&\frac{5}{2}\norm{\dot{g}}_{L^\infty(\Omega)}+\frac{5}{2}\lnorm{\frac{\nb_N \dot{P}}{\nb_NP}}_{L^\infty(\Omega)}+\norm{\D x}_{L^\infty(\Omega)}\norm{\Lie_B\D x}_{L^\infty(\Omega)},\\
\tilde{n}_1(t)=&\norm{\ddot{g}}_{L^\infty(\Omega)}+\norm{\dot{\omega}}_{L^\infty(\Omega)}+2\norm{D_t(\delta_{il}\D x^l\Lie_B\D x^i)}_{L^\infty(\Omega)}+\lnorm{\frac{\nb_N\ddot{P}}{\nb_NP}}_{L^\infty(\Omega)},
\end{align*}
and
\begin{align*}
f_1(t)=\norm{\dot{F}}+\norm{\dot{g}}_{L^\infty(\Omega)}\norm{F}.
\end{align*}
Then, we have the following energy estimates.
\begin{proposition}\label{prop.E1}
	Let $E_1^2(t)=E_{D_t}$, and
	\begin{align*}
	M_1(t)=&4\bar{n}_1^2(t)E_0^2(t)
	+4e^{2\int_0^t n_1(\tau)d\tau}\int_0^t \left(\tilde{n}_1(\tau)E_0(\tau)+f_1(\tau)\right)^2d\tau,
	\end{align*}
	 it holds for \eqref{1formeq} with zero initial data
	\begin{align*}
E_1^2(t)\ls &M_1(t)+\int_0^t M_1(s)e^{t-s} ds.
	\end{align*}
\end{proposition}

\begin{proof}
	From the above argument, we have obtained
\begin{align*}
\abs{\frac{d}{dt}|E_{D_t}+D_{D_t}|}=&\abs{\frac{d}{dt}(E_{D_t}+D_{D_t})}\\
\ls& 2n_1(t)|E_{D_t}+D_{D_t}|+2(\tilde{n}_1(t)E_0(t)+f_1(t))E_1(t),
\end{align*}
which yields
\begin{align*}
|E_{D_t}+D_{D_t}|\ls&  2e^{2\int_0^t n_1(\tau)d\tau}\int_0^t \left(\tilde{n}_1(s)E_0(s)+f_1(s)\right)E_1(s)ds.
\end{align*}
Thus,
\begin{align*}
E_1^2(t)\ls &  2e^{2\int_0^t n_1(\tau)d\tau}\int_0^t \left(\tilde{n}_1(s)E_0(s)+f_1(s)\right)E_1(s)ds  +2\bar{n}_1(t)E_0E_1\\
\ls &e^{2\int_0^t n_1(\tau)d\tau}\Big( \frac{1}{2}e^{-2\int_0^t n_1(\tau)d\tau}\int_0^tE_1^2(s)ds\\
&+2e^{2\int_0^t n_1(\tau)d\tau}\int_0^t \left(\tilde{n}_1(s)E_0(s)+f_1(s)\right)^2ds \Big) +2\bar{n}_1^2(t)E_0^2+\frac{1}{2}E_1^2,
\end{align*}
and then
\begin{align*}
E_1^2(t)
\ls &\int_0^tE_1^2(s)ds+M_1(t),
\end{align*}
which implies the desired result by the Gronwall inequality.
\end{proof}

\subsection{The more one order energy estimates with respect to $\Lie_B$} \label{sec.B}

We now analyze the higher order energy functional. 
Let
\begin{align*}
\begin{aligned}
\A_B=\A_{BP},\quad \mc{G}_B=\mc{M}_{g^B}, \quad  g_{ab}^B=\Lie_B g_{ab},\\
\mc{\dot{G}}_B =\mc{M}_{\dot{g}^B},\quad \mc{C}_B=\mc{M}_{\omega^B},\quad \omega_{ab}^B=\Lie_B \omega_{ab}.
\end{aligned}
\end{align*}
From \eqref{Lieeq2}, it follows that
\begin{align*}
L_1\Lie_BW^d=&\Lie_B\ddot{W}^d-\Lie_B^3W^d+\A  \Lie_BW^d+\dot{\mc{G}}\Lie_B\dot{W}^d-\mc{C}\Lie_B\dot{W}^d+\mc{X}\Lie_B^2W^d\no\\
=&\Lie_BF^d-(\A_{B}W^d+\dot{\mc{G}}_{B}\dot{W}^d-\mc{C}_{B}\dot{W}^d+\mc{G}_{B}\ddot{W}^d-\mc{G}_{B}F^d)\no\\
&+ 2 g^{ad}(\delta_{il}\D_a x^l\Lie_B\D_c x^i)_{B}  \Lie_BW^c.
\end{align*}

As similar as for the lowest-order energies, we define 
\begin{align*}
E_B=E(\Lie_BW)=\langle \Lie_B\dot{W}, \Lie_B\dot{W}\rangle +\langle \Lie_BW,(\A +I)\Lie_BW\rangle+\langle \Lie_B^2W,\Lie_B^2W\rangle.
\end{align*}

From \eqref{DtBW}, \eqref{Lie.commut} and $B\cdot N|_\Gamma=0$, we get
\begin{align*}
&-\langle \Lie_B^3W, \Lie_B\dot{W}\rangle=-\int_{\Omega} g_{ad}\Lie_B^3W^d\Lie_B\dot{W}^a dy\\
=&\frac{1}{2}\frac{d}{dt}\int_{\Omega}|\Lie_B^2W|^2dy-\frac{1}{2}\int_{\Omega} \dot{g}_{ab}\Lie_B^2W^a\Lie_B^2W^bdy+\int_{\Omega} (\Lie_Bg_{ab})\Lie_B^2W^a\Lie_B\dot{W}^b dy.
\end{align*}

One has
\begin{align*}
\langle \mc{X}\Lie_B^2W,\dot{W}_T\rangle
=-2\int_{\Omega} \delta_{il}\D_a x^l\Lie_B\D_cx^i\Lie_B^2W^c\dot{W}^a_T dy.
\end{align*}

Thus, by \eqref{Lie.2form}, we get
\begin{align*}
&-\langle \Lie_B^3W, \Lie_B\dot{W}\rangle+\langle \mc{X}\Lie_B^2W,\Lie_B\dot{W}\rangle\no\\
=&\frac{1}{2}\frac{d}{dt}\int_{\Omega}|\Lie_B^2W|^2dy-\frac{1}{2}\int_{\Omega} \dot{g}_{ab}\Lie_B^2W^a\Lie_B^2W^bdy\no\\
&+\int_{\Omega} \delta_{il}(\Lie_B\D_a x^l\D_cx^i-\D_a x^l\Lie_B\D_cx^i)\Lie_B^2W^a\Lie_B\dot{W}^b dy.
\end{align*}

From the antisymmetry of $\dot{\omega}$, one has
\begin{subequations}\label{ET'}
	\begin{align}
	\dot{E}_B=&2\langle \Lie_B\ddot{W}+\A \Lie_BW+\dot{\mc{G}}\Lie_B\dot{W}, \Lie_B\dot{W}\rangle+2\langle \Lie_BW,\Lie_B\dot{W}\rangle +D_t\langle \Lie_B^2 {W},\Lie_B^2 {W}\rangle\no\\
	&-\langle \dot{\mc{G}}\Lie_B\dot{W}, \Lie_B\dot{W}\rangle+\langle \dot{\mc{G}}\Lie_BW,\Lie_B{W}\rangle+\langle \Lie_BW,\dot{\A }\Lie_BW\rangle+\langle \dot{\mc{G}}\Lie_BW,{\A }\Lie_BW\rangle\no\\
	=&2\langle L_1\Lie_BW, \Lie_B\dot{W}\rangle +2\langle \dot{\mc{C}}\Lie_B\dot{W}, \Lie_B\dot{W}\rangle+2\langle  \Lie_B^3W,\Lie_B\dot{W}\rangle-2\langle \mc{X}\Lie_B^2W, \Lie_B\dot{W}\rangle\no\\
	&+2\langle \Lie_BW,\Lie_B\dot{W}\rangle +D_t\langle \Lie_B^2 {W},\Lie_B^2 {W}\rangle-\langle \dot{\mc{G}}\Lie_B\dot{W}, \Lie_B\dot{W}\rangle+\langle \dot{\mc{G}}\Lie_BW,\Lie_B{W}\rangle\no\\
	&+\langle \Lie_BW,\dot{\A }\Lie_BW\rangle+\langle \dot{\mc{G}}\Lie_BW,{\A }\Lie_BW\rangle\no\\
	=&2\langle \Lie_BF+\mc{G}_BF, \Lie_B\dot{W}\rangle-2\langle\dot{\mc{G}}_B\dot{W}, \Lie_B\dot{W}\rangle+2\langle\mc{C}_B\dot{W}, \Lie_B\dot{W}\rangle\!-\!2\langle\mc{G}_B\ddot{W}, \Lie_B\dot{W}\rangle\label{ET'.1}\\ 
	&-4\langle \Lie_B(\delta_{il}\D_a x^l\Lie_B\D_c x^i)\Lie_BW^c, \Lie_B\dot{W}^a\rangle\label{ET'.2}\\
	&-\langle \dot{\mc{G}}\Lie_B\dot{W}, \Lie_B\dot{W}\rangle+\langle \dot{\mc{G}}\Lie_BW,\Lie_BW\rangle+\langle \Lie_BW,\dot{\A }\Lie_BW\rangle+\langle \dot{\mc{G}}\Lie_BW,{\A }\Lie_BW\rangle\label{ET'.5}\\
	&+\int_{\Omega} \dot{g}_{ab}\Lie_B^2W^a\Lie_B^2W^bdy+2\langle \Lie_BW,\Lie_B\dot{W}\rangle\label{ET'.5'}\\
	&-2\langle\A_BW, \Lie_B\dot{W}\rangle.\label{ET'.6}
	\end{align}
\end{subequations}

Now, we control the term $\langle\A_TW, \dot{W}_T\rangle$. As the same argument as in the estimates of $E_1(t)$,  we have to deal with it in an indirect way, by including it in the energies.
Let
\begin{align*}
D_B=2\langle\A_BW, \Lie_BW\rangle,
\end{align*}
then
\begin{align*}
\dot{D}_B=&2\langle\dot{\A }_BW, \Lie_BW\rangle+2\langle\A_B\dot{W},\Lie_BW\rangle+2\langle\A_BW, \Lie_B\dot{W}\rangle.
\end{align*}

Therefore, we obtain
\begin{align}\label{ET'.7}
\dot{E}_B+\dot{D}_B=&\eqref{ET.1}+\eqref{ET.2}+\eqref{ET.5}+\eqref{ET.5'}\\
&+2\langle\dot{\A }_BW, \Lie_BW\rangle+2\langle\A_B\dot{W}, \Lie_BW\rangle.\no 
\end{align}

From \eqref{Malpha}, \eqref{op} and \eqref{dotA}, it yields
\begin{align*}
|\eqref{ET'.1}|\ls & 2\big(\norm{\Lie_BF}+\norm{\Lie_Bg}_{L^\infty(\Omega)}\norm{F}+\norm{\Lie_B\dot{g}}_{L^\infty(\Omega)}E_0\\
&\qquad+\norm{\Lie_B\omega}_{L^\infty(\Omega)}E_0+\norm{\Lie_Bg}_{L^\infty(\Omega)}E_1\big)E_B^{1/2},
\end{align*}
\begin{align*}
|\eqref{ET'.2}|\ls &4\norm{\Lie_B(\delta_{il}\D x^l\Lie_B\D x^i)}_{L^\infty(\Omega)}E_0E_B^{1/2},
\end{align*}
\begin{align*}
|\eqref{ET'.5}+\eqref{ET'.5'}|\ls \left(1+\norm{\dot{g}}_{L^\infty(\Omega)}+\lnorm{\frac{\nb_N\dot{P}}{\nb_NP}}\right)E_B,
\end{align*}
and
\begin{align*}
|\eqref{ET'.7}|\ls 2\left(\lnorm{\frac{\nb_N(B\dot{P})}{\nb_NP}}E_0+\lnorm{\frac{\nb_N(BP)}{\nb_NP}}E_1\right)E_B^{1/2}.
\end{align*}

Let
\begin{align*}
E_1^{B}:= E_B^{1/2},
\end{align*}
and
\begin{align*}
\bar{n}_1^{B}(t)=&\lnorm{\frac{\nb_N(BP)}{\nb_NP}}_{L^\infty(\Omega)}E_0,\qquad
n_1^{B}(t)=\frac{1}{2}\left(1+\norm{\dot{g}}_{L^\infty(\Omega)}+\lnorm{\frac{\nb_N\dot{P}}{\nb_NP}}\right), \\
\tilde{n}_1^{B}(t)=&\norm{\Lie_B\dot{g}}_{L^\infty(\Omega)}E_0 +\norm{\Lie_B\omega}_{L^\infty(\Omega)}E_0+\norm{\Lie_Bg}_{L^\infty(\Omega)}E_1\\
&+2\norm{\Lie_B(\delta_{il}\D x^l\Lie_B\D x^i)}_{L^\infty(\Omega)}E_0+\lnorm{\frac{\nb_N(B\dot{P})}{\nb_NP}}E_0+\lnorm{\frac{\nb_N(BP)}{\nb_NP}}E_1,\\
f_1^{B}(t)=&\norm{\Lie_BF}+\norm{\Lie_Bg}_{L^\infty(\Omega)}\norm{F}.
\end{align*}
Then, we have the following estimates.

\begin{proposition}\label{prop.E1S'}
	Let
	\begin{align*}
	M_1^B(t)=2\bar{n}_1^B(t)+2\int_0^t (\tilde{n}_1^B(\tau)+f_1^{B}(\tau))d\tau,
	\end{align*}
	it holds
	\begin{align*}
	E_1^{B}(t)\ls M_1^B(t)+ 2\int_0^t n_1^B(s)M_1^B(s) \exp\left(2\int_s^t n_1^{B}(\tau)d\tau \right)ds.
	\end{align*}
\end{proposition}

\begin{proof}
	From the above argument, we have obtained
	\begin{align*}
	\dot{E}_B+\dot{D}_B
	\ls 2E_1^B(f_1^{B}+\tilde{n}_1^B+n_1^BE_1^B).
	\end{align*}
	Since $E_B(0)=D_B(0)=0$, the integration over $[0,t]$ in time gives
	\begin{align*}
	{E}_B 
	\ls &2E_1^B\bar{n}_1^B+2\int_0^t E_1^B(n_1^BE_1^B+\tilde{n}_1^B+f_1^{B})d\tau.
	\end{align*}
	Taking the supremum on $[0,t]$ in time and dividing by $\sup_{[0,t]}E_1^B$, we get
	\begin{align*}
	E_1^{B}(t)\ls &M_1^B(t)+2\int_0^t n_1^B(\tau)E_1^B(\tau)d\tau.
	\end{align*}
	By the Gronwall inequality, we can obtain the desired estimates.
\end{proof}

\subsection{Construction of tangential vector fields and the div-curl decomposition} \label{sec.tang}

A basic estimate in the Euclidean coordinates is that derivatives of vector fields can be estimated by derivatives of the curl, the divergence and the tangential derivatives, as proved in \cite[Lemma 11.1]{L1}. But that estimate is not invariant under changes of coordinates, so we also expect to replace it by an inequality that also holds in the Lagrangian coordinates. After that we need to derive its higher-order versions as well. Both the curl and the divergence are invariant, but the other terms are not. There are two ways to make these terms to be invariant. One is to replace the differentiation by covariant differentiation as used in \cite{CL,HLarma}, and the other is to replace it by Lie derivatives with respect to tangential vector fields introduced below, as the same as used in \cite{L1}. Both ways result in a lower-order term involving only the norm of the $1$-form itself multiplied by a constant relative to the coordinates.

\begin{definition}
	Let $c_1$ be a constant satisfying
	\begin{align*}
	\sum_{a,b} (|g_{ab}|+|g^{ab}|)\ls c_1^2, \quad \abs{\frac{\D x}{\D y}}^2+\abs{\frac{\D y}{\D x}}^2\ls c_1^2,
	\end{align*}
	and let $K_1$ denote a continuous function of $c_1$.
\end{definition}

Indeed, the bound for the Jacobian of the coordinate and its inverse follows from the bound for the metric and its inverse, and the bound for the former implies an equivalent bound for the latter with $c_1^2$ multiplied by $n$.

Following \cite{L1}, we now construct the tangential divergence-free vector fields  which are independent of time and expressed of the form $T^a(y)\frac{\D}{\D y^a}$ in the Lagrangian coordinates. Due to $\det(\frac{\D x}{\D y})=1$, the divergence-free condition reduces to
\begin{align*}
\D_a T^a=0.
\end{align*}

Because $\Omega$ is just the unit ball in $\R^n$, the vector fields can be explicitly expressed. The rotation vector fields
$y^a\D_b-y^b\D_a$ span the tangent space of the boundary and are divergence-free in the interior. It is clear that $B=B^a\D_a$ belongs to this space. Moreover, they also span the tangent space of the level sets of the distance function from the boundary in the Lagrangian coordinates $d(y)=\dist(y,\Gamma)=1-|y|$ for $y\neq 0$ away from the origin. We denote this set of vector fields by $\mc{S}_0$. Thus, $B\in \mc{S}_0$.

We can also construct a finite set of vector fields, as the same as in \cite{L1,L2}, which span the tangential space when $d\gs d_0$ and are compactly supported in the set where $d\gs d_0/2$. We denote this set of vector fields by $\mc{S}_1$. Let $\mc{S}=\mc{S}_0\cup\mc{S}_1$ denote the family of space tangential vector fields, and let $\mc{T}=\mc{S}\cup \{D_t\}$ denote the family of space-time tangential vector fields.

Let the radial vector field be $R=y^a\D_a$. Then, $\D_aR^a=n$ is not $0$ but it suffices for our purposes that it is constant. Let $\mc{R}=\mc{S}\cup \{R\}$, which spans the full tangent space of the space everywhere. Let $\mc{U}=\mc{S}\cup \{R\}\cup \{D_t\}$ denote the family of all vector fields. Note that the radial vector field commutes with the rotations, i.e.,
\begin{align*}
[R,S]=0, \quad S\in\mc{S}_0.
\end{align*}
Furthermore, the commutators of two vector fields in $\mc{S}_0$ is just another vector field in $\mc{S}_0$. For $i=0,1$, let $\mc{R}_i=\mc{S}_i\cup \{R\}$, $\mc{T}_i=\mc{S}_i\cup \{D_t\}$ and $\mc{U}_i=\mc{T}_i\cup\{R\}$.

Now, we recall some estimates as follows.

\begin{lemma}[\mbox{\cite[Lemma 11.3]{L1}}] \label{lem.11.3}
	In the Lagrangian frame,  with $\underline{W}_a=g_{ab}W^b$, we have
	\begin{align}
	|\Lie_U W|\ls& K_1\left(|\curl \underline{W}|+|\dv W|+\sum_{S\in\mc{S}}|\Lie_S W|+[g]_1|W|\right), \quad U\in\mc{R},\label{RW}\\
	|\Lie_U W|\ls& K_1\left(|\curl \underline{W}|+|\dv W|+\sum_{T\in\mc{T}}|\Lie_T W|+[g]_1|W|\right), \quad U\in\mc{U},\label{UW}
	\end{align}
	where $[g]_1=1+|\D g|$. Furthermore,
	\begin{align}\label{DW}
	|\D W|\ls K_1\left( |\Lie_RW|+\sum_{S\in\mc{S}}|\Lie_S W|+|W|\right).
	\end{align}
	When $d(y)\ls d_0$, we may replace the sums over $\mc{S}$ by the sums over $\mc{S}_0$ and the sum over $\mc{T}$ by the sum over $\mc{T}_0$.
\end{lemma}

Next, we recall the higher-order versions of the inequality in last lemma. We need to apply the lemma to $W$ replaced by $\Lie_U^JW$, and the divergence term will vanish in our applications. We will be able to control the curl of $(\Lie_U^J\underline{W})_a=\Lie_U^J(g_{ab}W^b)$, which is different from the curl of $(\underline{\Lie_U^JW})_a=g_{ab}\Lie_U^JW^b$, but the difference is lower order and can be easily controlled. We first introduce some notation.

\begin{definition}\label{def.betasV}
	Let $\beta$ be a function, a $1$- or $2$-form, or vector field, and let $\mc{V}$ be any of our families of vector fields. Set
	\begin{align*}
	\begin{aligned}
	|\beta|_s^{\mc{V}}=&\sum_{|J|\ls s, J\in \mc{V}}\abs{\Lie_S^J\beta},\\
	[\beta]_\mu^{\mc{V}}=&\sum_{s_1+\cdots+s_k\ls \mu,s_i\gs 1} |\beta|_{s_1}^{\mc{V}}\cdots|\beta|_{s_k}^{\mc{V}}, \quad [\beta]_0^{\mc{V}}=1.
	\end{aligned}
	\end{align*}
	In particular, $|\beta|_r^{\mc{R}}$ and $|\beta|_r^{\mc{U}}$ are equivalent to $\sum_{|\alpha|\ls r}\abs{\D_y^\alpha \beta}$ and   $\sum_{|\alpha|+k\ls r}\abs{D_t^k\D_y^\alpha \beta}$, respectively.
\end{definition}

\begin{lemma}[\mbox{\cite[Lemma 11.5]{L1}}] \label{lem.divcurl}
	With the convention that $|\curl \underline{W}|_{-1}^{\mc{V}}=|\dv W|_{-1}^{\mc{V}}$ $=0$, we have 
	\begin{align*}
	|W|_r^{\mc{R}}\ls& K_1\left(|\curl \underline{W}|_{r-1}^{\mc{R}}+|\dv W|_{r-1}^{\mc{R}}+|W|_r^{\mc{S}}+\sum_{s=1}^r|g|_s^{\mc{R}}|W|_{r-s}^{\mc{R}}\right),\\
	|W|_r^{\mc{R}}\ls& K_1\sum_{s=1}^r[g]_s^{\mc{R}}\left(|\curl \underline{W}|_{r-1-s}^{\mc{R}}+|\dv W|_{r-1-s}^{\mc{R}}+|W|_{r-s}^{\mc{S}}\right).
	\end{align*}
	The same inequalities also hold with $\mc{R}$ replaced by $\mc{U}$ everywhere and $\mc{S}$ replaced by $\mc{T}$:
	\begin{align*}
	|W|_r^{\mc{U}}\ls& K_1\left(|\curl \underline{W}|_{r-1}^{\mc{U}}+|\dv W|_{r-1}^{\mc{U}}+|W|_r^{\mc{T}}+\sum_{s=1}^r|g|_s^{\mc{U}}|W|_{r-s}^{\mc{U}}\right),\\
	|W|_r^{\mc{U}}\ls& K_1\sum_{s=1}^r[g]_s^{\mc{U}}\left(|\curl \underline{W}|_{r-1-s}^{\mc{U}}+|\dv W|_{r-1-s}^{\mc{U}}+|W|_{r-s}^{\mc{T}}\right).
	\end{align*}
\end{lemma}

\subsection{Commutators between the linearized equation and Lie derivatives with respect to $B$}

In order to get the higher-order energy estimates of tangential derivatives, we first commute tangential vector fields through the linearized equation. 

Let $T\in\mc{T}$ be a tangential vector field, and recall that $[\Lie_T,D_t]=0$ and that if $W$ is divergence-free, then so does $\Lie_TW$. Now, we apply  Lie derivatives $\Lie_T^I=\Lie_{T_{i_1}}\cdots\Lie_{T_{i_r}}$ with the multi-index $I=(i_1,\cdots,i_r)$ to \eqref{1formeq}.

From \eqref{A.Lie.Leib}, we have for $r=|I|$,
\begin{align*}
\Lie_T^I(g_{ab}\ddot{W}^b)=&\sum_{I_1+I_2=I}\bic_r^{|I_1|}\Lie_T^{I_1}g_{ab}\Lie_T^{I_2}\ddot{W}^b=:c_{I_1I_2}^I\Lie_T^{I_1}g_{ab}\Lie_T^{I_2}\ddot{W}^b,
\end{align*}
where we sum over all $I_1+I_2=I$ and $c_{I_1I_2}^I=1$ (only for the simplicity of summing over the repeated indice) in last expression.

From \eqref{A.Lieq} and the identity
\begin{align*}
T(\D_cPW^c)=&T^d\D_d(\D_cPW^c)
=T^d\D_d\D_cPW^c+T^d\D_cP\D_dW^c\no\\
=&\D_c(T^d\D_dP)W^c-(\D_cT^d)\D_dPW^c+\D_cPT^d\D_dW^c\no\\
=&\D_c(TP)W^c-\D_cP(\D_dT^c)W^d+\D_cPT^d\D_dW^c\no\\
=&\D_c(TP)W^c+\D_cP\Lie_TW^c,
\end{align*}
 one has
\begin{align*}
\Lie_T(\D_a(\D_cPW^c))=&\D_aT(\D_cPW^c)
=\D_a(\D_c(TP)W^c+\D_cP\Lie_TW^c).
\end{align*}
Then we have inductively
\begin{align}\label{LTA}
\Lie_T^I(\D_a(\D_cPW^c))=&\D_aT^I(\D_cPW^c)
=c_{I_1I_2}^I\D_a(\D_c(T^{I_1}P)\Lie_T^{I_2}W^c).
\end{align}
Hence, we obtain
\begin{align}\label{Lieeq}
&c_{I_1I_2}^I(\Lie_T^{I_1}g_{ab})\Lie_T^{I_2}\ddot{W}^b-c_{I_1I_2}^I(\Lie_T^{I_1}g_{ab})\Lie_T^{I_2}\Lie_B^2{W}^b-c_{I_1I_2}^I\D_a(\D_c(T^{I_1}P)\Lie_T^{I_2}W^c)\\
=&-\D_a T^Iq-c_{I_1I_2}^I(\Lie_T^{I_1}(\dot{g}_{ab}-\omega_{ab}))\Lie_T^{I_2}\dot{W}^b+c_{I_1I_2}^I(\Lie_T^{I_1}g_{ab})\Lie_T^{I_2}F^b\no\\
&+ 2c_{I_1I_2}^I(\Lie_T^{I_1}(\delta_{il}\D_a x^l\Lie_B\D_c x^i))  \Lie_T^{I_2}\Lie_BW^c.\no
\end{align}

Denote 
\begin{align*}
\begin{aligned}
W_I=&\Lie_T^I W,\; F_I=\Lie_T^I F,\;P_I=T^IP,\; q_I=T^Iq,\\ (\cdot)_I=&\Lie_T^I(\cdot),\; g_{ab}^I=\Lie_T^Ig_{ab},\; \omega_{ab}^I=\Lie_T^I\omega_{ab},
\end{aligned}
\end{align*}
and $\dot{g}_{ab}^I=D_t\Lie_T^Ig_{ab}=\Lie_T^I\dot{g}_{ab}$, $\dot{W}_I=D_tW_I=\Lie_T^I\dot{W}$, etc. Then, \eqref{Lieeq} can be written as
\begin{align}\label{Lieeq1}
&c_{I_1I_2}^I g_{ab}^{I_1}\ddot{W}^b_{I_2}-c_{I_1I_2}^I g_{ab}^{I_1}(\Lie_B^2{W})^b_{I_2}-c_{I_1I_2}^I\D_a(\D_cP_{I_1}W^c_{I_2})\no\\
=&-\D_a q_I-c_{I_1I_2}^I(\dot{g}_{ab}^{I_1}-\omega_{ab}^{I_1})\dot{W}^b_{I_2}\no\\
&+ 2c_{I_1I_2}^I(\delta_{il}\D_a x^l\Lie_B\D_c x^i)_{I_1}  (\Lie_BW)^c{I_2}+c_{I_1I_2}^Ig_{ab}^{I_1}F^b_{I_2}.
\end{align}

Next, we project each term onto the divergence-free vector fields and introduce some new notation for the operators
\begin{align}\label{notation}
\begin{aligned}
\A_I W^a=\A_{P_I}W^a,\quad \mc{G}_IW^a=\Pdv (g^{ac}g_{cb}^I W^b), \\ \dot{\mc{G}}_I W^a=\Pdv(g^{ac}\dot{g}_{cb}^I W^b),\quad \mc{C}_IW^a=\Pdv (g^{ac}\omega_{cb}^I W^b),
\end{aligned}
\end{align}
and $\tilde{c}_I^{I_1I_2}=c_{I_1I_2}^I$ if $I_2\neq I$ while $\tilde{c}_I^{I_1I_2}=0$ if $I_2= I$. Then, we can write \eqref{LTA}  as
\begin{align}\label{LTA'}
\Pdv(g^{ba}\Lie_T^I(g_{ac}\A W^c))=\A W_I^b+\tilde{c}_I^{I_1I_2} \A_{I_1}W_{I_2}^b.
\end{align}
 Thus, we are able to rewrite \eqref{Lieeq1} as
\begin{align}\label{Lieeq2}
L_1W_I^d=&\ddot{W}_I^d-(\Lie_B^2W)_I^d+\A  W_I^d+\dot{\mc{G}}\dot{W}_I^d-\mc{C}\dot{W}_I^d+\mc{X}(\Lie_BW)_I^d\\
=&F_I^d-\tilde{c}_I^{I_1I_2}(\A_{I_1}W_{I_2}^d+\dot{\mc{G}}_{I_1}\dot{W}_{I_2}^d-\mc{C}_{I_1}\dot{W}_{I_2}^d+\mc{G}_{I_1}\ddot{W}_{I_2}^d-\mc{G}_{I_1}F_{I_2}^d)\no\\
&+ 2\tilde{c}_I^{I_1I_2} g^{ad}(\delta_{il}\D_a x^l\Lie_B\D_c x^i)_{I_1}  (\Lie_BW)^c_{I_2}.\no
\end{align}

Now, we define higher-order energies. For $I\in\mc{V}$ with $|I|=r\gs 2$,  let
\begin{align*}
E_I=&E(W_I)=\langle \dot{W}_I, \dot{W}_I\rangle +\langle W_I,(\A +I)W_I\rangle+\langle \Lie_B W_I,\Lie_B W_I\rangle.
\end{align*}

For $\mc{V}\in\{\{D_t\}, \{B\}, \mc{B}, \mc{S}, \mc{T}, \mc{R}, \mc{U}\}$ where $\mc{B}=\{B, D_t\}$, let
\begin{align}
\abs{W}_s^{\mc{V}}=&\sum_{|I|\ls s,\, T\in\mc{V}}\abs{\Lie_T^IW}, \; \abs{W}_{s,B}^{\mc{V}}=\sum_{|I|\ls s,\, T\in\mc{V}}\abs{\Lie_B\Lie_T^IW},\\
\hnorm{\D q}_{s,\infty,P^{-1}}^{\mc{V}}=&\!\!\!\!\sum_{|I|=s, T\in\mc{V}}\!\lnorm{\frac{\nb_N \Lie_T^Iq}{\nb_NP}}_{L^\infty(\Omega)},\; \norm{\D q}_{s,\infty,P^{-1}}^{\mc{V}}=\sum_{0\ls l\ls s}\hnorm{\D q}_{l,\infty,P^{-1}}^{\mc{V}}, \label{Dq}\\
\norm{f}_{s,\infty}^{\mc{V}}=&\!\!\!\!\sum_{|I|\ls s,\, T\in\mc{V}}\norm{\Lie_T^If}_{L^\infty(\Omega)},\;F_s^{\mc{V}}=\sum_{|I|\ls s,\,  T\in\mc{V}} \norm{F_I}, \label{EsV}\\
\mathring{E}_s^{\mc{V}}=&\sum_{|I|=s,\, T\in\mc{V}}\sqrt{E_I},\quad 
E_s^{\mc{V}}=\sum_{0\ls l\ls s}\mathring{E}_l^{\mc{V}}.
\end{align}

\subsection{The higher-order energy estimates for time  and $\Lie_B$ derivatives} \label{sec.tang.r}

From \eqref{Lieeq2}, we have for $I\in \mc{B}$,
\begin{align}
\dot{E}_I=&2\langle \dot{W}_I, \ddot{W}_I+\A  W_I\rangle+\langle \dot{W}_I, \dot{\mc{G}}\dot{W}_I\rangle +2\langle \dot{W}_I, W_I\rangle+ \langle \dot{\mc{G}}W_I, (\A +I)W_I\rangle\no\\
&+\langle W_I, \dot{\A }W_I\rangle +\langle \dot{\mc{G}}\Lie_B W_I,\Lie_B W_I\rangle+2\langle \Lie_B \dot{W}_I,\Lie_B W_I\rangle\no\\
=&2\langle \dot{W}_I,F_I\rangle-\langle \dot{\mc{G}}\dot{W}_I,\dot{W}_I \rangle+2\langle \dot{W}_I, W_I\rangle+ \langle \dot{\mc{G}}W_I, (\A +I)W_I\rangle\no\\
&+\langle W_I, \dot{\A }W_I\rangle +\langle \dot{\mc{G}}\Lie_B W_I,\Lie_B W_I\rangle\label{EI.1}\\
&-2\tilde{c}_I^{I_1I_2}(\langle \dot{W}_I,\A_{I_1}W_{I_2}\rangle+\langle \dot{W}_I,\dot{\mc{G}}_{I_1}\dot{W}_{I_2}\rangle-\langle \dot{W}_I,\mc{C}_{I_1}\dot{W}_{I_2}\rangle\no\\
&+\langle \dot{W}_I,\mc{G}_{I_1}\ddot{W}_{I_2}\rangle-\langle \dot{W}_I,\mc{G}_{I_1}F_{I_2}\rangle)\label{EI.2}\\
&+ 4\tilde{c}_I^{I_1I_2} \langle \dot{W}_I^a, (\delta_{il}\D_a x^l\Lie_B\D_c x^i)_{I_1}  (\Lie_BW)^c_{I_2}\rangle\label{EI.3}\\
&+2\langle \Lie_B \dot{W}_I,\Lie_B W_I\rangle+2\langle \dot{W}_I,(\Lie_B^2W)_I \rangle-2\langle \dot{W}_I,\mc{X}(\Lie_BW)_I\rangle.\label{EI.4}
\end{align}

It is clear that
\begin{align*}
\abs{\eqref{EI.1}}\ls 2E_I^{1/2}\norm{F_I}+\left(1+\norm{\dot{g}}_{L^\infty(\Omega)}+\lnorm{\frac{\nb_N \dot{P}}{\nb_NP}}_{L^\infty(\Omega)}\right)E_I,
\end{align*}
and
\begin{align*}
\abs{\eqref{EI.3}}\ls 4\tilde{c}_I^{I_1I_2}\norm{(\delta_{il}\D_a x^l\Lie_B\D_c x^i)_{I_1}}_{L^\infty(\Omega)}E_{I_2}^{1/2}E_I^{1/2}.
\end{align*}

To deal with the term $\langle \dot{W}_I,\A_{I_1}W_{I_2}\rangle$, we introduce
\begin{align*}
D_I=2\tilde{c}_I^{I_1I_2}\langle W_I,\A_{I_1}W_{I_2}\rangle,
\end{align*}
then
\begin{align*}
\dot{D}_I=2\tilde{c}_I^{I_1I_2}(\langle \dot{W}_I,\A_{I_1}W_{I_2}\rangle+\langle W_I,\A_{I_1}\dot{W}_{I_2}\rangle+\langle W_I,\dot{\A}_{I_1}W_{I_2}\rangle).
\end{align*}
Thus,
\begin{align*}
&\abs{\dot{D}_I+\eqref{EI.2}}\no\\
\ls& 2\tilde{c}_I^{I_1I_2}\left(\lnorm{\frac{\nb_N P_{I_1}}{\nb_NP}}_{L^\infty(\Omega)}+\lnorm{\frac{\nb_N \dot{P}_{I_1}}{\nb_NP}}_{L^\infty(\Omega)}\!\!\!\!\!\!+\norm{\dot{g}^{I_1}}_{L^\infty(\Omega)}+\norm{\omega^{I_1}}_{L^\infty(\Omega)}\right)E_{I_2}^{1/2}E_I^{1/2}\no\\
&+2\tilde{c}_I^{I_1I_2}\norm{{g}^{I_1}}_{L^\infty(\Omega)}(\norm{\ddot{W}_{I_2}}+\norm{F_{I_2}})E_I^{1/2},
\end{align*}
where the term $\norm{\ddot{W}_{I_2}}$ can be controlled by the  energy norm taking one  $T=D_t$.

Since $B\cdot N=0$ on $\Gamma$, we get by  \eqref{Lie.2form},
\begin{align*}
\eqref{EI.4}=&2\int_{\Omega} g_{ab}\Lie_B (\dot{W}_I^a\Lie_B W_I^b)dy+4\int_{\Omega}\delta_{il}\D_a x^l\Lie_B\D_bx^i \dot{W}_I^a(\Lie_BW^b)_Idy\\
=&-2\langle (\Lie_Bg_{ab})\dot{W}_I^a,\Lie_B W_I^b\rangle+4\langle\delta_{il}\D_a x^l\Lie_B\D_bx^i \dot{W}_I^a,\Lie_BW_I^b\rangle\\
=&2\langle \delta_{il}(\D_a x^l\Lie_B\D_bx^i-\Lie_B\D_a x^l\D_bx^i) \dot{W}_I^a,\Lie_BW_I^b\rangle.
\end{align*}
Then, 
\begin{align*}
\abs{\eqref{EI.4}}\ls 4\norm{\D x}_{L^\infty(\Omega)}\norm{\Lie_B\D x}_{L^\infty(\Omega)} E_I.
\end{align*}
Therefore,
\begin{align}\label{EDI.1}
&\dot{E}_I+\dot{D}_I\\
\ls &2E_I^{1/2}(\norm{F_I}+\tilde{c}_I^{I_1I_2}\norm{{g}^{I_1}}_{L^\infty(\Omega)}(\norm{\ddot{W}_{I_2}}+\norm{F_{I_2}}))\no\\
&+\left(1+\norm{\dot{g}}_{L^\infty(\Omega)}+\lnorm{\frac{\nb_N \dot{P}}{\nb_NP}}_{L^\infty(\Omega)}+4\norm{\delta_{il}\D x^l\Lie_B\D x^i}_{L^\infty(\Omega)}\right)E_I\no\\
&+2\tilde{c}_I^{I_1I_2}\left(\lnorm{\frac{\nb_N P_{I_1}}{\nb_NP}}_{L^\infty(\Omega)}+\lnorm{\frac{\nb_N \dot{P}_{I_1}}{\nb_NP}}_{L^\infty(\Omega)}+\norm{\dot{g}^{I_1}}_{L^\infty(\Omega)}+\norm{\omega^{I_1}}_{L^\infty(\Omega)}\right)\no\\
&\cdot E_{I_2}^{1/2}E_I^{1/2}+4\tilde{c}_I^{I_1I_2}\norm{(\delta_{il}\D x^l\Lie_B\D x^i)_{I_1}}_{L^\infty(\Omega)}E_{I_2}^{1/2}E_I^{1/2}.\no
\end{align}
Noticing that $E_I(0)=D_I(0)=0$, the integration over $[0,t]$ in time implies
\begin{align}\label{EI}
E_I\ls &|D_I|+\int_0^t\eqref{EDI.1} d\tau\\
\ls &2E_I\sum_{s=0}^{r-1}\bic_r^s \hnorm{\D P}_{r-s,\infty,P^{-1}}\mathring{E}_s^{\mc{B}}+\int_0^t\eqref{EDI.1} d\tau.\no
\end{align}
Let
\begin{align*}
n_r^{\mc{B}}=&1+(2r-1)\norm{\dot{g}}_{L^\infty(\Omega)}+2(r-1)\norm{\Lie_B{g}}_{L^\infty(\Omega)}+\lnorm{\frac{\nb_N \dot{P}}{\nb_NP}}_{L^\infty(\Omega)}\\
&+ 4\norm{\delta_{il}\D_a x^l\Lie_B\D_bx^i}_{L^\infty(\Omega)},\\
\tilde{n}_r^{\mc{B}}=&\int_0^t(\norm{\D P}_{r,\infty,P^{-1}}+\norm{\D \dot{P}}_{r,\infty,P^{-1}}+\norm{\dot{g}}_{r,\infty}^{\mc{B}}+\norm{\omega}_{r,\infty}^{\mc{B}}+\norm{g}_{r,\infty}^{\mc{B}}\\
&\qquad+\norm{\delta_{il}\D_a x^l\Lie_B\D_c x^i}_{r,\infty}^{\mc{B}})
d\tau,\\
f_r^{\mc{B}}=&\int_0^t (1+\norm{g}_{r,\infty}^{\mc{B}})F_r^{\mc{B}}d\tau.
\end{align*}
Taking the supremum on $[0,t]$ in time of \eqref{EI}, then summing up the order from $0$ to $r$, and dividing by $\bar{E}_r^{\mc{B}}=\sup_{[0,t]}E_r^{\mc{B}}$, we get 
\begin{align*}
\bar{E}_r^{\mc{B}}
\ls &C (\sup_{[0,t]}\norm{\D P}_{r-1,\infty,P^{-1}}+\tilde{n}_r^{\mc{B}})\bar{E}_{r-1}^{\mc{B}}+Cf_r^{\mc{B}}+\int_0^t n_r^{\mc{B}}\bar{E}_r^{\mc{B}} d\tau.
\end{align*}
By the Gronwall inequality, we can obtain the following estimates.
\begin{proposition}\label{prop.Er.recur}
	Let 
\begin{align*}
M_r^{\mc{B}}=C \Big[(\sup_{[0,t]}\norm{\D P}_{r-1,\infty,P^{-1}}+\tilde{n}_r^{\mc{B}})\bar{E}_{r-1}^{\mc{B}}+f_r^{\mc{B}}\Big],
\end{align*}
it holds	
\begin{align*}
\bar{E}_r^{\mc{B}}(t)\ls &M_r^{\mc{B}}(t) +\int_0^t M_r^{\mc{B}}(s)n_r^{\mc{B}}(s)   \exp\left(\int_s^tn_r^{\mc{B}}(\tau)d\tau\right) ds.
\end{align*}
\end{proposition}

This is a recursion formula between $\bar{E}_r^{\mc{B}}$ and $\bar{E}_{r-1}^{\mc{B}}$, thus we can obtain inductively the estimates of $\bar{E}_r^{\mc{B}}$ and $E_r^{\mc{B}}$  since we have proved the estimates of $\bar{E}_1^{\mc{B}}=\sup_{[0,t]}(E_0+E_1+E_1^{B})$ in Propositions \ref{prop.E0}, \ref{prop.E1} and \ref{prop.E1S}. Indeed, we have the following:

\begin{proposition}\label{prop.ErT}
	Assume that $x,\,P\in C^{r+2}([0,T]\times\Omega)$, $B\in C^{r+2}(\Omega)$,  $P|_{\Gamma}=0$, $\nb_NP|_\Gamma\ls -c_0<0$, 
	$B^aN_a|_\Gamma = 0$ and $\dv V=0$, where $V=D_t x$. Suppose that $W$ is a solution of \eqref{1formeq} where $F$ is divergence-free and vanishing to order $r$ as $t\to 0$.  Then, there is a constant $C=C(x,P,B)$ depending only on the norm of $(x,P,B)$, a lower bound for $c_0$, and an upper bound for $T$ such that if $E_s^{\mc{B}}(0)=0$ for $s\ls r$, then
	\begin{align}\label{ErT}
	E_r^{\mc{B}}(t)\ls C\int_0^t\norm{F}_r^{\mc{B}}d\tau, \quad \text{ for } t\in [0,T].
	\end{align}
\end{proposition}

\subsection{Estimates for the tangential derivatives}

We have got the higher-order time and $\Lie_B$ derivatives that are some kinds of tangential derivatives due to $\dv B=0$ and $B\cdot N|_\Gamma=0$, but they do not give the estimates for all tangential derivatives. Thus, we need to derive the estimates for tangential derivatives $\sum_{T\in\mc{T}} |\Lie_T W|$ of $W$.

Let $T\in\mc{T}$, $W_T=\Lie_TW$, $F_T=\Lie_T F$, and similar notation as in \eqref{notation}:
\begin{align*}
\begin{aligned}
\A_T=\A_{TP},\quad \mc{G}_T=\mc{M}_{g^T}, \quad  g_{ab}^T=\Lie_T g_{ab},\\
\mc{\dot{G}}_T =\mc{M}_{\dot{g}^T},\quad \mc{C}_T=\mc{M}_{\omega^T},\quad \omega_{ab}^T=\Lie_T \omega_{ab}.
\end{aligned}
\end{align*}
Then, from \eqref{Lieeq2}, it follows that
\begin{align*}
L_1W_T^d=&\ddot{W}_T^d-\Lie_T\Lie_B^2W^d+\A  W_T^d+\dot{\mc{G}}\dot{W}_T^d-\mc{C}\dot{W}_T^d+\mc{X}\Lie_T\Lie_BW^d\no\\
=&F_T^d-(\A_{T}W^d+\dot{\mc{G}}_{T}\dot{W}^d-\mc{C}_{T}\dot{W}^d+\mc{G}_{T}\ddot{W}^d-\mc{G}_{T}F^d)\no\\
&+ 2 g^{ad}(\delta_{il}\D_a x^l\Lie_B\D_c x^i)_{T}  \Lie_BW^c.
\end{align*}

As the arguments in the lowest-order energies, we define 
\begin{align*}
E_T=E(W_T)=\langle \dot{W}_T, \dot{W}_T\rangle +\langle W_T,(\A +I)W_T\rangle+\langle \Lie_BW_T,\Lie_BW_T\rangle.
\end{align*}

From \eqref{DtBW}, \eqref{Lie.commut} and $B\cdot N|_\Gamma=0$, we get
\begin{align*}
&-\langle \Lie_T\Lie_B^2W, \dot{W}_T\rangle=-\int_{\Omega} g_{ad}\Lie_T\Lie_B^2W^d\dot{W}_T^a dy\\
=&-\int_{\Omega} g_{ad}[\Lie_T,\Lie_B]\Lie_BW^d\dot{W}_T^a dy-\int_{\Omega} g_{ad}\Lie_B[\Lie_T,\Lie_B]W^d\dot{W}_T^a dy\\ &-\int_{\Omega} g_{ad}\Lie_B^2W_T^d\dot{W}_T^a dy\\
=&-\int_{\Omega} g_{ad}\Lie_{[T,B]}\Lie_BW^d\dot{W}_T^a dy-\int_{\Omega} g_{ad}\Lie_B\Lie_{[T,B]}W^d\dot{W}_T^a dy\\
&+\frac{1}{2}\frac{d}{dt}\int_{\Omega}|\Lie_BW_T|^2dy-\frac{1}{2}\int_{\Omega} \dot{g}_{ab}\Lie_BW_T^a\Lie_BW_T^bdy+\int_{\Omega} (\Lie_Bg_{ab})\Lie_BW_T^a\dot{W}_T^b dy.
\end{align*}
One has
\begin{align*}
\langle \mc{X}\Lie_T\Lie_BW,\dot{W}_T\rangle
=&-2\int_{\Omega} \delta_{il}\D_a x^l\Lie_B\D_cx^i\Lie_T\Lie_BW^c\dot{W}^a_T dy\no\\
=&-2\int_{\Omega} \delta_{il}\D_a x^l\Lie_B\D_cx^i\Lie_{[T,B]}W^c\dot{W}^a_T dy\\
&-2\int_{\Omega} \delta_{il}\D_a x^l\Lie_B\D_cx^i\Lie_BW^c_T\dot{W}^a_T dy.
\end{align*}
Thus, by \eqref{Lie.2form}, we get
\begin{align*}
&-\langle \Lie_T\Lie_B^2W, \dot{W}_T\rangle+\langle \mc{X}\Lie_T\Lie_BW,\dot{W}_T\rangle\no\\
=&-2\int_{\Omega} g_{ad}\Lie_{[T,B]}\Lie_BW^d\dot{W}_T^a dy+\int_{\Omega} g_{ad}\Lie_{[[T,B],B]}W^d\dot{W}_T^a dy+\frac{1}{2}\frac{d}{dt}\int_{\Omega}|\Lie_BW_T|^2dy\\
&-\frac{1}{2}\int_{\Omega} \dot{g}_{ab}\Lie_BW_T^a\Lie_BW_T^bdy-2\int_{\Omega} \delta_{il}\D_a x^l\Lie_B\D_cx^i\Lie_{[T,B]}W^c\dot{W}^a_T dy\\
&+\int_{\Omega} \delta_{il}(\Lie_B\D_a x^l\D_cx^i-\D_a x^l\Lie_B\D_cx^i)\Lie_BW^c_T\dot{W}^a_T dy.
\end{align*}
From the antisymmetry of $\dot{\omega}$, one has
\begin{subequations}\label{ET}
	\begin{align}
	\dot{E}_T=&2\langle \ddot{W}_T+\A W_T+\dot{\mc{G}}\dot{W}_T, \dot{W}_T\rangle+2\langle W_T,\dot{W}_T\rangle +D_t\langle \Lie_B {W}_T,\Lie_B {W}_T\rangle\no\\
	&-\langle \dot{\mc{G}}\dot{W}_T, \dot{W}_T\rangle+\langle \dot{\mc{G}}W_T,{W}_T\rangle+\langle W_T,\dot{\A }W_T\rangle+\langle \dot{\mc{G}}W_T,{\A }W_T\rangle\no\\
	=&2\langle L_1W_T, \dot{W}_T\rangle +2\langle \dot{\mc{C}}\dot{W}_T, \dot{W}_T\rangle+2\langle \Lie_T \Lie_B^2W,\dot{W}_T\rangle-2\langle \mc{X}\Lie_T\Lie_BW, \dot{W}_T\rangle\no\\
	&+2\langle W_T,\dot{W}_T\rangle +D_t\langle \Lie_B {W}_T,\Lie_B {W}_T\rangle\no\\
	&-\langle \dot{\mc{G}}\dot{W}_T, \dot{W}_T\rangle+\langle \dot{\mc{G}}W_T,{W}_T\rangle+\langle W_T,\dot{\A }W_T\rangle+\langle \dot{\mc{G}}W_T,{\A }W_T\rangle\no\\
	=&2\langle F_T+\mc{G}_TF, \dot{W}_T\rangle-2\langle\dot{\mc{G}}_T\dot{W}, \dot{W}_T\rangle+2\langle\mc{C}_T\dot{W}, \dot{W}_T\rangle-2\langle\mc{G}_T\ddot{W}, \dot{W}_T\rangle\label{ET.1}\\ 
	&-4\langle (\delta_{il}\D_a x^l\Lie_B\D_c x^i)_{T}\Lie_BW^c, \dot{W}_T^a\rangle\no\\
	&-2\int_{\Omega} \delta_{il}(\Lie_B\D_a x^l\D_cx^i-\D_a x^l\Lie_B\D_cx^i)\Lie_BW^c_T\dot{W}^a_T dy\label{ET.2}\\
	&+4\int_{\Omega} \delta_{il}\D_a x^l\Lie_B\D_cx^i\Lie_{[T,B]}W^c\dot{W}^a_T dy\label{ET.3}\\
	&+4\int_{\Omega} g_{ad}\Lie_{[T,B]}\Lie_BW^d\dot{W}_T^a dy\label{ET.3'}\\
	&-2\int_{\Omega} g_{ad}\Lie_{[[T,B],B]}W^d\dot{W}_T^a dy\label{ET.4}
	\\
	&-\langle \dot{\mc{G}}\dot{W}_T, \dot{W}_T\rangle+\langle \dot{\mc{G}}W_T,{W}_T\rangle+\langle W_T,\dot{\A }W_T\rangle+\langle \dot{\mc{G}}W_T,{\A }W_T\rangle\label{ET.5}\\
	&+\int_{\Omega} \dot{g}_{ab}\Lie_BW_T^a\Lie_BW_T^bdy+2\langle W_T,\dot{W}_T\rangle\label{ET.5'}\\
	&-2\langle\A_TW, \dot{W}_T\rangle.\label{ET.6}
	\end{align}
\end{subequations}

Now, we control the term $\langle\A_TW, \dot{W}_T\rangle$. As the same argument as in the estimates of $E_1(t)$,  we have to deal with it in an indirect way, by including it in the energies.
Let
\begin{align*}
D_T=2\langle\A_TW, {W}_T\rangle,
\end{align*}
then 
\begin{align*}
\dot{D}_T=&2\langle\dot{\A }_TW, {W}_T\rangle+2\langle\A_T\dot{W}, {W}_T\rangle+2\langle\A_TW, \dot{W}_T\rangle.
\end{align*}

Therefore, we obtain
\begin{align}
\dot{E}_T+\dot{D}_T=&\eqref{ET.1}+\eqref{ET.2}+\eqref{ET.3}+\eqref{ET.3'}+\eqref{ET.4}+\eqref{ET.5}+\eqref{ET.5'}\no\\
&+2\langle\dot{\A }_TW, {W}_T\rangle+2\langle\A_T\dot{W}, {W}_T\rangle. \label{ET.7}
\end{align}

From \eqref{Malpha}, \eqref{op} and \eqref{dotA}, it yields
\begin{align*}
|\eqref{ET.1}|\ls & 2\big(\norm{F_T}+\norm{g^T}_{L^\infty(\Omega)}\norm{F}+\norm{\dot{g}^T}_{L^\infty(\Omega)}E_0\\
& +\norm{\omega^T}_{L^\infty(\Omega)}E_0+\norm{g^T}_{L^\infty(\Omega)}E_1\big)E_T^{1/2},
\end{align*}
\begin{align*}
|\eqref{ET.2}|\ls &4\norm{(\delta_{il}\D x^l\Lie_B\D x^i)_T}_{L^\infty(\Omega)}E_0E_T^{1/2}+2\norm{\delta_{il}\D x^l\Lie_B\D x^i}_{L^\infty(\Omega)}E_T,
\end{align*}
\begin{align*}
|\eqref{ET.5}+\eqref{ET.5'}|\ls \left(1+\norm{\dot{g}}_{L^\infty(\Omega)}+\lnorm{\frac{\nb_N\dot{P}}{\nb_NP}}\right)E_T,
\end{align*}
\begin{align*}
|\eqref{ET.7}|\ls 2\left(\lnorm{\frac{\nb_NT\dot{P}}{\nb_NP}}E_0+\lnorm{\frac{\nb_NTP}{\nb_NP}}E_1\right)E_T^{1/2},
\end{align*}
\begin{align*}
|\eqref{ET.3}|\ls& 4\norm{\delta_{il}\D x^l\Lie_B\D x^i}_{L^\infty(\Omega)}\norm{\Lie_{[T,B]}W}E_T^{1/2},
\end{align*}
\begin{align*}
|\eqref{ET.4}|\ls& 2\norm{g}_{L^\infty(\Omega)}\norm{\Lie_{[[T,B],B]}W}E_T^{1/2},
\end{align*}
and
\begin{align*}
|\eqref{ET.3'}|\ls&4\norm{g}_{L^\infty(\Omega)}\norm{\Lie_{[T,B]}\Lie_BW}E_T^{1/2}. 
\end{align*}

Since $B\in \mc{S}$ and for $T\in\mc{S}$,
\begin{align*}
\dv [T,B]=\D_b(T^a\D_a B^b-B^a\D_aT^b)=\D_bT^a\D_a B^b-\D_bB^a\D_aT^b=0,
\end{align*}
we get $[T,B]\in \mc{S}$. Similarly, $[[T,B],B]\in \mc{S}$. Thus, from the above estimates and observation, we see that the energies should include $E_T$ for any $T\in \mc{T}$ in order to deal with the commutators. 
 Thus, we define the energy as
\begin{align*}
E_1^{T}:= E_T^{1/2} \text{ for } T\in\mc{T}, \quad E_1^{\mc{T}}=\sum_{T\in\mc{T}}E_1^T.
\end{align*}

Let
\begin{align*}
\bar{n}_1^{T}(t)=&\lnorm{\frac{\nb_NTP}{\nb_NP}}_{L^\infty(\Omega)}E_0,\\
n_1(t)=&\frac{1}{2}\left(1+\norm{\dot{g}}_{L^\infty(\Omega)}+3\norm{g}_{L^\infty(\Omega)}+\lnorm{\frac{\nb_N\dot{P}}{\nb_NP}}\right)+2\norm{\delta_{il}\D x^l\Lie_B\D x^i}_{L^\infty(\Omega)}, \\
\tilde{n}_1^{T}(t)=&\norm{\dot{g}^T}_{L^\infty(\Omega)}E_0 +\norm{\omega^T}_{L^\infty(\Omega)}E_0+\norm{g^T}_{L^\infty(\Omega)}E_1\\
&+2\norm{(\delta_{il}\D x^l\Lie_B\D x^i)_T}_{L^\infty(\Omega)}E_0+\lnorm{\frac{\nb_NT\dot{P}}{\nb_NP}}E_0+\lnorm{\frac{\nb_NTP}{\nb_NP}}E_1,\\
f_1^{T}(t)=&\norm{F_T}+\norm{g^T}_{L^\infty(\Omega)}\norm{F},\\
\tilde{n}_1^{\mc{T}}(t)=&2\sum_{T\in\mc{T}}\sup_{[0,t]}\bar{n}_1^T+2\int_0^t\sum_{T\in\mc{T}}\tilde{n}_1^Td\tau,\\
f_1^{\mc{T}}(t)=&2\int_0^t\sum_{T\in\mc{T}}f_1^T d\tau.
\end{align*}
Then, we have the following estimates.

\begin{proposition}\label{prop.E1S}
	It holds
	\begin{align*}
	E_1^{\mc{T}}\ls\tilde{n}_1^{\mc{T}}(t)+f_1^{\mc{T}}(t)+\int_0^t \left(\tilde{n}_1^{\mc{T}}(s)+f_1^{\mc{T}}(s)\right)n_1(s)\exp\left(\int_s^t n_1(\tau)d\tau\right)ds.
	\end{align*}
\end{proposition}

\begin{proof}
	From the above argument, we have obtained
	\begin{align*}
	&\dot{E}_T+\dot{D}_T\\
	\ls &2E_1^T\Bigg\{\norm{F_T}+\norm{g^T}_{L^\infty(\Omega)}\norm{F}+\norm{\dot{g}^T}_{L^\infty(\Omega)}E_0 +\norm{\omega^T}_{L^\infty(\Omega)}E_0+\norm{g^T}_{L^\infty(\Omega)}E_1\\
	&\quad+2\norm{(\delta_{il}\D x^l\Lie_B\D x^i)_T}_{L^\infty(\Omega)}E_0
    +\left(\lnorm{\frac{\nb_NT\dot{P}}{\nb_NP}}E_0+\lnorm{\frac{\nb_NTP}{\nb_NP}}E_1\right)\\
	&\quad+\frac{1}{2}\left(1+\norm{\dot{g}}_{L^\infty(\Omega)}+3\norm{g}_{L^\infty(\Omega)}+\lnorm{\frac{\nb_N\dot{P}}{\nb_NP}}+4\norm{\delta_{il}\D x^l\Lie_B\D x^i}_{L^\infty(\Omega)}\right)E_1^T\Bigg\}.
	\end{align*}
	Since $E_T(0)=D_T(0)=0$, the integration over $[0,t]$ in time gives
	\begin{align*}
	{E}_T 
	\ls &2E_1^T\bar{n}_1^T+2\int_0^t E_1^T\big[n_1E_1^T+\tilde{n}_1^T+f_1^T\big]d\tau.
	\end{align*}
	Taking the supremum on $[0,t]$ in time and dividing by $\sup_{[0,t]}E_1^T$, we  sum over $T\in\mc{T}$ to get
	\begin{align}\label{est.E1S}
	E_1^{\mc{T}}\ls &2\sum_{T\in\mc{T}}\sup_{[0,t]}\bar{n}_1^T+2\int_0^t\Big[n_1E_1^{\mc{T}}+\sum_{T\in\mc{T}}\tilde{n}_1^T\Big]d\tau+f_1^{\mc{T}}\\
	\ls &\int_0^tn_1E_1^{\mc{T}}d\tau+\tilde{n}_1^{\mc{T}}+f_1^{\mc{T}}.\no
	\end{align}
	By the Gronwall inequality, we can obtain the desired estimates.
\end{proof}

\subsection{Estimates for the curl and the full derivatives of the first order} \label{sec.curl}

Now, we will derive the estimates of normal derivatives close to the boundary by using the estimates of the curl and the estimates of the tangential derivatives in view of Lemma \ref{lem.divcurl}. Thus, we have to derive the estimates of the curl and the time derivatives of the curl. For this reason, we need to use the $1$-form of $W$ and $\dot{W}$, denoted by $w$ and $\dot{w}$ respectively, i.e.,
\begin{align*}
w_a=g_{ab}W^b, \quad \dot{w}_a=g_{ab}\dot{W}^b,
\end{align*}
in which the latter notation is slightly confusing and $\dot{w}$ is not equal to $D_tw$, but we only try to indicate that $\dot{w}$ is the corresponding 1-form obtained by lowering the indices of the vector field $\dot{W}$. 

Let
\begin{align*}
\curl w_{ab}=\D_a w_b-\D_b w_a, \quad \underline{F}_a=g_{ab}F^b.
\end{align*}
Since $D_tw_a=D_t(g_{ab}W^b)=\dot{g}_{ab}W^b+g_{ab}\dot{W}^b$, we have
\begin{align}\label{Dtcurlw}
D_t\curl w_{ab}=&D_t(\D_aw_b-\D_bw_a)\no\\
=&\D_a(\dot{g}_{bc}W^c+g_{bc}\dot{W}^c)-\D_b(\dot{g}_{ac}W^c+g_{ac}\dot{W}^c)\no\\
=&(\D_a\dot{g}_{bc}-\D_b\dot{g}_{ac})W^c+\dot{g}_{bc}\D_aW^c-\dot{g}_{ac}\D_bW^c+\D_a\dot{w}_b-\D_b\dot{w}_a\no\\
=&\curl \dot{w}_{ab}+\D_c\omega_{ab}W^c+\dot{g}_{bc}\D_aW^c-\dot{g}_{ac}\D_bW^c\no\\
&+[(\dot{g}_{eb}-\omega_{eb})\D_a{\D_c x^k}-(\dot{g}_{ea}-\omega_{ea})\D_b{\D_c x^k}]\frac{\D y^e}{\D x^k}W^c,
\end{align}
since from \eqref{Dv} we have $2\D_bv_i=(\dot{g}_{cb}-\omega_{cb})\frac{\D y^c}{\D x^i}$ and 
\begin{align*}
\D_a \dot{g}_{db}-\D_d\dot{g}_{ab}=&\D_a[{\D_d x^i}{\D_b x^k}(\D_k v_i+\D_i v_k)]-\D_d[{\D_a x^i}{\D_b x^k}(\D_k v_i+\D_i v_k)]\\
=&{\D_d x^i}{\D_b x^k}\D_ax^l\D_l\D_k v_i-{\D_a x^i}{\D_b x^k}\D_dx^l\D_l\D_k v_i\\
&+({\D_d x^i}\D_a{\D_b x^k}-{\D_a x^k}\D_d{\D_b x^i})(\D_k v_i+\D_i v_k)\\
=&{\D_d x^i}\D_ax^k\D_b(\D_k v_i-\D_i v_k)+({\D_d x^i}\D_a{\D_b x^k}+{\D_a x^k}\D_d{\D_b x^i})(\D_k v_i-\D_i v_k)\\
&+2{\D_d x^i}\D_a{\D_b x^k}\D_i v_k-2{\D_a x^k}\D_d{\D_b x^i}\D_k v_i\\
=&\D_b\omega_{ad}+2(\D_d v_k\D_a{\D_b x^k}-\D_a v_k\D_d{\D_b x^k})\\
=& \D_b\omega_{ad}+[(\dot{g}_{cd}-\omega_{cd})\D_a{\D_b x^k}-(\dot{g}_{ca}-\omega_{ca})\D_d{\D_b x^k}]\frac{\D y^c}{\D x^k}.
\end{align*}

Due to $\dv W=0$, we can get from Lemma \ref{lem.11.3} and  \eqref{Dtcurlw} that
\begin{align*}
\abs{D_t\curl w}\ls& \abs{\curl\dot{w}}+\abs{\D \omega}\abs{W}+2\abs{\dot{g}}\abs{\D W}+(\abs{\dot{g}}+\abs{\omega})\abs{\D^2x}\abs{\frac{\D y}{\D x}}\abs{W}\no\\
\ls & \abs{\curl\dot{w}}+K_1 \abs{\dot{g}}\Big( |\curl w|+\sum_{S\in\mc{S}}|\Lie_S W|+[g]_1|W|\Big)\no\\
&+\left[\abs{\D \omega}+(\abs{\dot{g}}+\abs{\omega})\abs{\D^2x}\abs{\frac{\D y}{\D x}}\right]\abs{W}.
\end{align*}

Thus, we have to derive the estimates of $\curl\dot{w}$. From \eqref{1formeq}, we get 
\begin{align*}
D_t\dot{w}_a=&D_t(g_{ab}\dot{W}^b)=\dot{g}_{ab}\dot{W}^b+g_{ab}\ddot{W}^b\\
=&g_{ab}\Lie_B^2W^b-\D_a q+\D_a(W^c\D_cP)+\omega_{ab}\dot{W}^b+2\delta_{il}\D_a x^l\Lie_B\D_cx^i\Lie_BW^c+g_{ab}F^b.
\end{align*}
Note that the above equation can  be also formulated as
\begin{align}\label{Dtw}
D_t\dot{w}_a-g_{ab}\Lie_B^2W^b+g_{ab}(\A W^b-\mc{C}\dot{W}^b+\mc{X}\Lie_B W^b)=\underline{F}_a.
\end{align}
Then, we have
\begin{align}
D_t\curl\dot{w}_{ad}=&D_t(\D_a\dot{w}_d-\D_d\dot{w}_a)=\D_aD_t\dot{w}_d-\D_dD_t\dot{w}_a\no\\
=&\D_a(g_{db}\Lie_B^2W^b-\D_d q+\D_d(W^c\D_cP)+\omega_{db}\dot{W}^b+2\delta_{il}\D_d x^l\Lie_B\D_cx^i\Lie_BW^c\no\\
&+\underline{F}_d)-\D_d(g_{ab}\Lie_B^2W^b-\D_a q+\D_a(W^c\D_cP)\no\\
&+\omega_{ab}\dot{W}^b+2\delta_{il}\D_a x^l\Lie_B\D_cx^i\Lie_BW^c+\underline{F}_a)\no\\
=&\curl\underline{\Lie_B^2W}_{ad}+\D_b\omega_{da}\dot{W}^b+\omega_{db}\D_a\dot{W}^b-\omega_{ab}\D_d\dot{W}^b\no+\curl\underline{F}_{ad}\\
&+2\delta_{il}\D_d x^l\D_a(\Lie_B\D_cx^i\Lie_BW^c)-2\delta_{il}\D_a x^l\D_d(\Lie_B\D_cx^i\Lie_BW^c),\label{DtcurlDtw}
\end{align}
where we have used the identity $\D_a\omega_{db}-\D_d\omega_{ab}=\D_b\omega_{da}$ which can be verified by \eqref{gomega}. In fact,
\begin{align*}
\D_a\omega_{db}-\D_d\omega_{ab}=&\D_a[{\D_d x^i}{\D_b x^k}(\D_i v_k-\D_k v_i)]-\D_d[{\D_a x^i}{\D_b x^k}(\D_i v_k-\D_k v_i)]\\
=&{\D_d x^i}\D_a{\D_b x^k}(\D_i v_k-\D_k v_i)-{\D_a x^i}\D_d{\D_b x^k}(\D_i v_k-\D_k v_i)\\
&+{\D_d x^i}{\D_b x^k}\D_a(\D_i v_k-\D_k v_i)-{\D_a x^i}{\D_b x^k}\D_d(\D_i v_k-\D_k v_i)\\
=&\D_b[({\D_d x^i}\D_a{ x^k})(\D_i v_k-\D_k v_i)]-({\D_d x^i}\D_a{ x^k})\D_bx^l\D_l(\D_i v_k-\D_k v_i)\\
&+{\D_d x^i}{\D_b x^k}\D_ax^l\D_l(\D_i v_k-\D_k v_i)-{\D_a x^i}{\D_b x^k}\D_dx^l\D_l(\D_i v_k-\D_k v_i)\\
=&\D_b\omega_{da}-\D_l\D_iv_k({\D_d x^i}\D_a{ x^k}\D_bx^l-{\D_d x^i}{\D_b x^k}\D_ax^l+{\D_a x^i}{\D_b x^k}\D_dx^l)\\
&+\D_l\D_kv_i({\D_d x^i}\D_a{ x^k}\D_bx^l-{\D_d x^i}{\D_b x^k}\D_ax^l+{\D_a x^i}{\D_b x^k}\D_dx^l)\\
=&\D_b\omega_{da}-\D_l\D_iv_k{\D_d x^i}\D_a{ x^k}\D_bx^l+\D_l\D_kv_i{\D_a x^i}{\D_b x^k}\D_dx^l\\
=&\D_b\omega_{da}.
\end{align*}

From \eqref{A.Lie.Leib}, we get
\begin{align*}
\underline{\Lie_B^2W}_a=&g_{ea}\Lie_B^2W^e=\Lie_B^2w_a-2\Lie_B^2g_{ea}W^e-\Lie_B g_{ea}\Lie_BW^e,
\end{align*}
and then
\begin{align}\label{curlLB2W}
\curl\underline{\Lie_B^2W}_{ad}=&\curl\Lie_B^2w_{ad}-2[\D_a(\Lie_B^2g_{ed}W^e)-\D_d(\Lie_B^2g_{ea}W^e)]\no\\
&-[\D_a(\Lie_B g_{ed}\Lie_BW^e)-\D_d(\Lie_B g_{ea}\Lie_BW^e)]\no\\
=&\curl\Lie_B^2w_{ad}-2(\curl \Lie_B^2g_{e\cdot})_{ad}W^e-2[\Lie_B^2g_{ed}\D_aW^e-\Lie_B^2g_{ea}\D_dW^e]\no\\
&-(\curl \Lie_Bg_{e\cdot})_{ad}\Lie_BW^e-[\Lie_Bg_{ed}\D_a\Lie_BW^e-\Lie_Bg_{ea}\D_d\Lie_BW^e].
\end{align}

From \eqref{DtcurlDtw} and \eqref{curlLB2W}, it follows that
\begin{align}\label{DtcurlDtw'}
D_t\curl\dot{w}_{ad}
=&\curl\Lie_B^2w_{ad}-2(\curl \Lie_B^2g_{e\cdot})_{ad}W^e-(\curl \Lie_Bg_{e\cdot})_{ad}\Lie_BW^e+\D_b\omega_{da}\dot{W}^b\no\\
&-2[\Lie_B^2g_{ed}\D_aW^e-\Lie_B^2g_{ea}\D_dW^e]+[\omega_{db}\D_a\dot{W}^b-\omega_{ab}\D_d\dot{W}^b]\no\\
&+2\D_cB^b[\D_ag_{db}-\D_dg_{ab}]\Lie_BW^c+2[g_{db}\D_a\D_cB^b-g_{ab}\D_d\D_cB^b]\Lie_BW^c\no\\
&+2[\delta_{il}\D_d x^l\D_aB\D_cx^i-\delta_{il}\D_a x^l\D_dB\D_cx^i]\Lie_BW^c\no\\
&-[\Lie_Bg_{cd}\D_a\Lie_BW^c-\Lie_Bg_{ca}\D_d\Lie_BW^c]\no\\
&+2\D_cB^b[g_{db}\D_a\Lie_BW^c-g_{ab}\D_d\Lie_BW^c]\no\\
&+2\delta_{il}B\D_cx^i[\D_d x^l\D_a\Lie_BW^c-\D_a x^l\D_d\Lie_BW^c]+\curl\underline{F}_{ad}.
\end{align}

With the help of \eqref{A.comm.DtLie} and \eqref{Lie.curl}, we have
\begin{align*}
&\langle \curl\underline{\Lie_B^2W}, \curl\dot{w}\rangle\\
=&\int_{\Omega} g^{ab}g^{cd}\curl\underline{\Lie_B^2W}_{ad}\curl\dot{w}_{bc}dy\\
=&-\frac{1}{2}D_t\langle \curl\Lie_B w,\curl\Lie_B w\rangle+\int_{\Omega} \dot{g}^{ab}g^{cd}\curl\Lie_Bw_{ad}\curl\Lie_B w_{bc} dy\\
&-2\int_{\Omega} (\Lie_Bg^{ab})g^{cd}\curl\Lie_Bw_{ad}\curl\dot{w}_{bc}dy\\
&-\int_{\Omega} g^{ab}g^{cd}\curl\Lie_Bw_{ad}\Lie_B\Big\{\D_e\omega_{cb}W^e+\dot{g}_{be}\D_cW^e-\dot{g}_{ce}\D_bW^e\no\\
&+[(\dot{g}_{eb}-\omega_{eb})\D_c{\D_f x^k}-(\dot{g}_{ec}-\omega_{ec})\D_b{\D_f x^k}]\frac{\D y^e}{\D x^k}W^f\Big\}dy\\
&-\int_{\Omega} g^{ab}g^{cd}\big\{2(\curl \Lie_B^2g_{e\cdot})_{ad}W^e+2[\Lie_B^2g_{ed}\D_aW^e-\Lie_B^2g_{ea}\D_dW^e]\no\\
&+(\curl \Lie_Bg_{e\cdot})_{ad}\Lie_BW^e+[\Lie_Bg_{ed}\D_a\Lie_BW^e-\Lie_Bg_{ea}\D_d\Lie_BW^e]\big\}\curl\dot{w}_{bc}dy.
\end{align*}

Let
\begin{align*}
E_{\curl}(t)=\langle \curl w, \curl w\rangle+\langle \curl \dot{w}, \curl \dot{w}\rangle+\langle \curl \Lie_Bw, \curl \Lie_Bw\rangle.
\end{align*}
Taking the inner product of \eqref{DtcurlDtw'} with $\curl\dot{w}$, we obtain, with the help of \eqref{LieDy.commu}, that
\begin{align}\label{est.Ecurl}
\frac{1}{2}\frac{d}{dt}E_{\curl}(t)
\ls& (2\norm{\dot{g}}_{L^\infty(\Omega)}+\norm{\Lie_Bg}_{L^\infty(\Omega)})E_{\curl}\no\\
&+E_{\curl}^{1/2}[\norm{\Lie_B\D\omega}_{L^\infty(\Omega)}+\norm{\D\omega}_{L^\infty(\Omega)}]E_0\no\\
&+E_{\curl}^{1/2}[\norm{\Lie_B\dot{g}}_{L^\infty(\Omega)}+\norm{\dot{g}}_{L^\infty(\Omega)}]\no\\
&\cdot[\norm{\D W}+\norm{\D\Lie_B W}+\norm{\D^2 B}_{L^\infty(\Omega)}E_0]\no\\
&+E_{\curl}^{1/2}[\norm{\Lie_B\dot{g}}_{L^\infty(\Omega)}+\norm{\Lie_B \omega}_{L^\infty(\Omega)}]\norm{\D^2 x}_{L^\infty(\Omega)}\norm{\frac{\D y}{\D x}}_{L^\infty(\Omega)}E_0\no\\
&+E_{\curl}^{1/2}[\norm{\dot{g}}_{L^\infty(\Omega)}+\norm{\omega}_{L^\infty(\Omega)}]\big[\norm{\Lie_B\D^2 x}_{L^\infty(\Omega)}\norm{\frac{\D y}{\D x}}_{L^\infty(\Omega)}\no\\
&\qquad+\norm{\D^2 x}_{L^\infty(\Omega)}\norm{\Lie_B\frac{\D y}{\D x}}_{L^\infty(\Omega)}+2\norm{\D^2 x}_{L^\infty(\Omega)}\norm{\frac{\D y}{\D x}}_{L^\infty(\Omega)}\big]E_0\no\\
&+E_{\curl}^{1/2}[2\norm{\curl\Lie_B^2 g}_{L^\infty(\Omega)}+\norm{\curl\Lie_B g}_{L^\infty(\Omega)}]E_0\no\\
&+E_{\curl}^{1/2}[4\norm{\Lie_B^2 g}_{L^\infty(\Omega)}\norm{\D W}+2\norm{\Lie_B g}_{L^\infty(\Omega)}\norm{\D\Lie_B W}]\no\\
&+4E_{\curl}^{1/2}[\norm{\delta_{il}\D_d x^l\D_a(B\D_cx^i)+\D_a(g_{db}\D_cB^b)}_{L^\infty(\Omega)}E_0\no\\
&\qquad+\norm{\delta_{il}\D_d x^lB\D_cx^i+g_{db}\D_cB^b}_{L^\infty(\Omega)}\norm{\D\Lie_BW}]\no\\
&+2E_{\curl}^{1/2}\norm{\omega}_{L^\infty(\Omega)}\norm{\D \dot{W}}+E_{\curl}^{1/2}\norm{\curl\underline{F}}.
\end{align}
 
From \eqref{DW} and \eqref{RW}, it follows that
\begin{align}
\norm{\D W}
\ls &K_1 \left(\norm{\curl w}+\sum_{S\in\mc{S}}\norm{\Lie_S W}+\norm{[g]_1}_{L^\infty(\Omega)}\norm{W}\right)\no\\
\ls &K_1 \norm{\curl w}+K_1 E_1^{\mc{S}}+\norm{[g]_1}_{L^\infty(\Omega)}E_0,\label{est.DW}\\
\norm{\D \dot{W}}
\ls &K_1 \left(\norm{\curl \dot{w}}+\sum_{S\in\mc{S}}\norm{\Lie_S \dot{W}}+\norm{[g]_1}_{L^\infty(\Omega)}\norm{\dot{W}}\right)\no\\
\ls &K_1 \norm{\curl \dot{w}}+K_1 E_1^{\mc{S}}+\norm{[g]_1}_{L^\infty(\Omega)}E_0.\label{est.DDtW}
\end{align}
Then, from 
\begin{align*}
\underline{\Lie_BW}_a=g_{ab}\Lie_BW^b=\Lie_Bw_a-(\Lie_Bg_{ab})W^b,
\end{align*}
we have 
\begin{align}\label{DLBW}
\norm{\D \Lie_B W}\ls &K_1 \left(\norm{\curl \underline{\Lie_BW}}+\sum_{S\in\mc{S}}\norm{[\Lie_S,\Lie_B] W}+E_1^{\mc{S}}+\norm{[g]_1}_{L^\infty(\Omega)}E_0\right)\no\\
\ls &K_1 \Big(\norm{\curl\Lie_B w}+(\norm{B}_{1}^{\mc{S}}+\norm{\Lie_Bg}_{L^\infty(\Omega)}) \norm{\curl w}+E_1^{\mc{S}}\no\\
&\qquad\quad+(\norm{[g]_1}_{L^\infty(\Omega)}+\norm{\D\Lie_Bg}_{L^\infty(\Omega)})E_0\Big).
\end{align}
Combining \eqref{est.Ecurl}, \eqref{est.DW}, \eqref{est.DDtW} and \eqref{DLBW}, we obtain
\begin{align*}
\frac{d}{dt}(E_{\curl}^{1/2}(t))
\ls&n_{1,\curl}(E_{\curl}^{1/2}+E_1^{\mc{S}})
+\tilde{n}_{1,\curl}E_0+\norm{\curl\underline{F}},
\end{align*}
where
\begin{align*}
n_{1,\curl}=&K_1 (1+\norm{B}_{1}^{\mc{S}}+\norm{\Lie_Bg}_{L^\infty(\Omega)})\Big(\norm{\Lie_B\dot{g}}_{L^\infty(\Omega)}+\norm{\dot{g}}_{L^\infty(\Omega)}\\
&\quad+\norm{\Lie_B^2 g}_{L^\infty(\Omega)}+\norm{\Lie_B g}_{L^\infty(\Omega)}+\norm{\delta_{il}\D_d x^lB\D_cx^i\\
	&\quad+g_{db}\D_cB^b}_{L^\infty(\Omega)}+\norm{\omega}_{L^\infty(\Omega)}\Big),\\
\tilde{n}_{1,\curl}=&n_{1,\curl}(\norm{[g]_1}_{L^\infty(\Omega)}+\norm{\D\Lie_Bg}_{L^\infty(\Omega)})\\
&+K_1 \Big(\norm{\Lie_B\D\omega}_{L^\infty(\Omega)}+\norm{\D\omega}_{L^\infty(\Omega)}+[\norm{\Lie_B\dot{g}}_{L^\infty(\Omega)}\\
&\quad+\norm{\dot{g}}_{L^\infty(\Omega)}]\norm{\D^2 B}_{L^\infty(\Omega)}\\
&+[\norm{\Lie_B\dot{g}}_{L^\infty(\Omega)}+\norm{\Lie_B \omega}_{L^\infty(\Omega)}]\norm{\D^2 x}_{L^\infty(\Omega)}\lnorm{\frac{\D y}{\D x}}_{L^\infty(\Omega)}\no\\
&+[\norm{\dot{g}}_{L^\infty(\Omega)}+\norm{\omega}_{L^\infty(\Omega)}]\Big[\norm{\Lie_B\D^2 x}_{L^\infty(\Omega)}\lnorm{\frac{\D y}{\D x}}_{L^\infty(\Omega)}\no\\
&\quad+\norm{\D^2 x}_{L^\infty(\Omega)}\lnorm{\Lie_B\frac{\D y}{\D x}}_{L^\infty(\Omega)}+2\norm{\D^2 x}_{L^\infty(\Omega)}\lnorm{\frac{\D y}{\D x}}_{L^\infty(\Omega)}\Big]\\
&\quad+\norm{\curl\Lie_B^2 g}_{L^\infty(\Omega)}+\norm{\curl\Lie_B g}_{L^\infty(\Omega)}\\
&+\norm{\delta_{il}\D_d x^l\D_a(B\D_cx^i)+\D_a(g_{db}\D_cB^b)}_{L^\infty(\Omega)} \Big).
\end{align*}
Due to $E_{\curl}(0)=0$, the integration over $[0,t]$ in time gives
\begin{align}\label{est.1curl}
E_{\curl}^{1/2}(t)\ls \int_0^t [n_{1,\curl}(E_{\curl}^{1/2}+E_1^{\mc{S}})
+\tilde{n}_{1,\curl}E_0+\norm{\curl\underline{F}}] d\tau.
\end{align}

From \eqref{est.E1S} and \eqref{est.1curl}, we have
\begin{align*}
E_1^{\mc{S}}+E_{\curl}^{1/2}\ls& \int_0^t (n_{1,\curl}+n_1^{\mc{S}}+\bar{n}_{1,\curl})(E_{\curl}^{1/2}+E_1^{\mc{S}})d\tau+\tilde{f}_1,
\end{align*}
where
\begin{align*}
\bar{n}_{1,\curl}=&K_1 \norm{g}_{L^\infty(\Omega)}\norm{B}_1^{\mc{S}}(1+\norm{B}_1^{\mc{S}}+\norm{B}_{W^{1,\infty}(\Omega)}+\norm{\delta_{il}\D x^lB\D x^i}_{L^\infty(\Omega)}),\\
\tilde{f}_1=&\tilde{n}_1^{\mc{S}}+\int_0^t(\tilde{n}_{1,\curl}+\bar{n}_{1,\curl}\norm{\curl\Lie_B g}_{L^\infty(\Omega)})E_0 d\tau\\
&+2\int_0^t\Big[\sum_{T\in\mc{S}}f_1^T+\norm{\curl\underline{F}}\Big] d\tau.
\end{align*}
Therefore, by the Gronwall inequality, we have obtained the following estimates for both the first order tangential derivatives and the curl.

\begin{proposition}\label{prop.est.1curl}
	It holds
\begin{align*}
E_1^{\mc{S}}(t)+E_{\curl}^{1/2}(t)\ls \tilde{f}_1(t)+\int_0^t& \tilde{f}_1(s)\left(n_{1,\curl}(s)+n_1^{\mc{S}}(s)+\bar{n}_{1,\curl}(s)\right)\\
&\cdot  \exp\left(\int_s^t(n_{1,\curl}(\tau)+n_1^{\mc{S}}(\tau)+\bar{n}_{1,\curl}(\tau))d\tau\right)ds.
\end{align*}
\end{proposition}

\begin{remark}
By Lemma \ref{lem.11.3}, we have the estimates for the first-order derivative of $W$.
\end{remark}

\subsection{The higher-order estimates for the curl and the normal derivatives}\label{sec.curl.r}

Now, we need to get the equations for the curl of higher order derivatives. Since the Lie derivative commutes with $D_t$ and the curl, applying $\Lie_U^J$ to \eqref{DtcurlDtw'} and \eqref{Dtcurlw} gives
\begin{align}\label{DtcurlDtLUJw}
D_t\curl\Lie_U^J\dot{w}_{ad}
=&\curl\Lie_U^J\Lie_B^2w_{ad}+\curl\Lie_U^J\underline{F}_{ad}-2c_{J_1J_2}(\curl\Lie_U^{J_1} \Lie_B^2g_{e\cdot})_{ad}\Lie_U^{J_2}W^e\no\\
&-c_{J_1J_2}(\curl \Lie_U^{J_1}\Lie_Bg_{e\cdot})_{ad}\Lie_U^{J_2}\Lie_BW^e+c_{J_1J_2}\Lie_U^{J_1}\D_b\omega_{da}\Lie_U^{J_2}\dot{W}^b\no\\
&-2c_{J_1J_2}[\Lie_U^{J_1}\Lie_B^2g_{ed}[\Lie_U^{J_2},\D_a]W^e-\Lie_U^{J_1}\Lie_B^2g_{ea}[\Lie_U^{J_2},\D_d]W^e]\no\\
&-2c_{J_1J_2}[\Lie_U^{J_1}\Lie_B^2g_{ed}\D_a\Lie_U^{J_2}W^e-\Lie_U^{J_1}\Lie_B^2g_{ea}\D_d\Lie_U^{J_2}W^e]\no\\
&+c_{J_1J_2}[\Lie_U^{J_1}\omega_{db}[\Lie_U^{J_2},\D_a]\dot{W}^b-\Lie_U^{J_1}\omega_{ab}[\Lie_U^{J_2},\D_d]\dot{W}^b]\no\\
&+c_{J_1J_2}[\Lie_U^{J_1}\omega_{db}\D_a\Lie_U^{J_2}\dot{W}^b-\Lie_U^{J_1}\omega_{ab}\D_d\Lie_U^{J_2}\dot{W}^b]\no\\
&+2c_{J_1J_2}\Lie_U^{J_1}[\D_cB^b(\D_ag_{db}-\D_dg_{ab})]\Lie_U^{J_2}\Lie_BW^c\no\\
&+2c_{J_1J_2}\Lie_U^{J_1}[g_{db}\D_a\D_cB^b-g_{ab}\D_d\D_cB^b]\Lie_U^{J_2}\Lie_BW^c\no\\
&+2c_{J_1J_2}\Lie_U^{J_1}[\delta_{il}\D_d x^l\D_aB\D_cx^i-\delta_{il}\D_a x^l\D_dB\D_cx^i]\Lie_U^{J_2}\Lie_BW^c\no\\
&-c_{J_1J_2}[\Lie_U^{J_1}\Lie_Bg_{cd}[\Lie_U^{J_2},\D_a]\Lie_BW^c-\Lie_U^{J_1}\Lie_Bg_{ca}[\Lie_U^{J_2},\D_d]\Lie_BW^c]\no\\
&-c_{J_1J_2}[\Lie_U^{J_1}\Lie_Bg_{cd}\D_a\Lie_U^{J_2}\Lie_BW^c-\Lie_U^{J_1}\Lie_Bg_{ca}\D_d\Lie_U^{J_2}\Lie_BW^c]\no\\
&+2c_{J_1J_2}\delta_{il}[\Lie_U^{J_1}(B\D_cx^i\D_d x^l)[\Lie_U^{J_2},\D_a]\Lie_BW^c\no\\
&\qquad\qquad\qquad-\Lie_U^{J_1}(B\D_cx^i\D_a x^l)[\Lie_U^{J_2},\D_d]\Lie_BW^c]\no\\
&+2c_{J_1J_2}\delta_{il}[\Lie_U^{J_1}(B\D_cx^i\D_d x^l)\D_a\Lie_U^{J_2}\Lie_BW^c\no\\
&\qquad\qquad\qquad-\Lie_U^{J_1}(B\D_cx^i\D_a x^l)\D_d\Lie_U^{J_2}\Lie_BW^c]\no\\
&+2c_{J_1J_2}[\Lie_U^{J_1}(g_{db}\D_cB^b)[\Lie_U^{J_2},\D_a]\Lie_BW^c\no\\
&\qquad\qquad\qquad-\Lie_U^{J_1}(g_{ab}\D_cB^b)[\Lie_U^{J_2},\D_d]\Lie_BW^c]\no\\
&+2c_{J_1J_2}[\Lie_U^{J_1}(g_{db}\D_cB^b)\D_a\Lie_U^{J_2}\Lie_BW^c\no\\
&\qquad\qquad\qquad-\Lie_U^{J_1}(g_{ab}\D_cB^b)\D_d\Lie_U^{J_2}\Lie_BW^c],
\end{align}
and
\begin{align}\label{DtcurlLUJw}
D_t\curl\Lie_U^J w_{ab}
=&\curl\Lie_U^J \dot{w}_{ab}+c_{J_1J_2}\Lie_U^{J_1}\D_c\omega_{ab}\Lie_U^{J_2}W^c\no\\
&+c_{J_1J_2}(\Lie_U^{J_1}\dot{g}_{bc}[\Lie_U^{J_2},\D_a]W^c-\Lie_U^{J_1}\dot{g}_{ac}[\Lie_U^{J_2},\D_b]W^c)\no\\
&+c_{J_1J_2}(\Lie_U^{J_1}\dot{g}_{bc}\D_a\Lie_U^{J_2}W^c-\Lie_U^{J_1}\dot{g}_{ac}\D_b\Lie_U^{J_2}W^c)\no\\
&+c_{J_1J_2}\Lie_U^{J_1}\left\{[(\dot{g}_{eb}-\omega_{eb})\D_a{\D_c x^k}-(\dot{g}_{ea}-\omega_{ea})\D_b{\D_c x^k}]\frac{\D y^e}{\D x^k}\right\}\Lie_U^{J_2}W^c.
\end{align}

At this point, we have to derive the commutator $[\Lie_U^J,\D]$. If $|J|=1$, it is just the identity \eqref{LieDy.commu}. For $|J|\gs 2$, we have the following identity.

\begin{lemma}\label{lem.LUJD.commu}
	For $|J|=r\gs 1$ and $|J_r|=1$, it holds
	\begin{align}\label{comm.LieUJDa}
	[\Lie_U^{J-J_r}\Lie_U^{J_r},\D_a]W^b=W^c\Lie_U^{J-J_r}\D_c\D_aU_{J_r}^b+\sum_{J=I_1+I_2+I_3\atop |I_3|=1}\sgn(|I_1|)\Lie_U^{I_1}W^c\Lie_U^{I_2}\D_c\D_aU_{I_3}^b.
	\end{align}
\end{lemma}

\begin{proof}
For $r=1$, it follows from \eqref{def.Lie}
\begin{align}
[\Lie_U,\D_a]W^b=\D_c\D_aU^bW^c.
\end{align}
For $r\gs 2$, we prove it by induction argument. For $r=2$, we have
\begin{align*}
[\Lie_{U_1}\Lie_{U_2},\D_a]W^b=&\Lie_{U_1}[\Lie_{U_2},\D_a]W^b+[\Lie_{U_1},\D_a]\Lie_{U_2}W^b\\
=&\Lie_{U_1}(W^d\D_d\D_aU_2^b)+(\Lie_{U_2}W^d)\D_d\D_aU_1^b\\
=&(\Lie_{U_1}W^d)\D_d\D_aU_2^b+(\Lie_{U_2}W^d)\D_d\D_aU_1^b+W^d\Lie_{U_1}\D_d\D_aU_2^b,
\end{align*}
which satisfies \eqref{comm.LieUJDa}.

Now, we assume that \eqref{comm.LieUJDa} holds for $r=s$. Then, we derive the case $r=s+1$. For $|J|=s+1$ and $|J_r|=1$, one gets by using \eqref{LieDy.commu}
\begin{align*}
&[\Lie_U^{J-J_{s+1}}\Lie_U^{J_{s+1}},\D_a]W^b\\
=&\Lie_U^{J-J_{s+1}}[\Lie_U^{J_{s+1}},\D_a]W^b+[\Lie_U^{J-J_{s+1}},\D_a]\Lie_U^{J_{s+1}}W^b\\
=&\Lie_U^{J-J_{s+1}}(W^d\D_d\D_aU_{J_{s+1}}^b)+\Lie_U^{J_{s+1}}W^d\Lie_U^{J-J_s-J_{s+1}}\D_d\D_aU_{J_s}^b\\
&+\sum_{J-J_{s+1}=J_1+J_2+J_3\atop |J_1|\gs 1,\;|J_3|=1}\Lie_U^{J_1}\Lie_U^{J_{s+1}}W^d\Lie_U^{J_2}\D_d\D_aU_{J_3}^b\\
=&\sum_{J-J_{s+1}=J_1+J_2}\Lie_U^{J_1}W^d\Lie_U^{J_2}\D_d\D_aU_{J_{s+1}}^b+\Lie_U^{J_{s+1}}W^d\Lie_U^{J-J_s-J_{s+1}}\D_d\D_aU_{J_s}^b\\
&+\sum_{J-J_{s+1}=J_1+J_2+J_3\atop |J_1|\gs 1,\;|J_3|=1}\Lie_U^{J_1}\Lie_U^{J_{s+1}}W^d\Lie_U^{J_2}\D_d\D_aU_{J_3}^b\\
=&W^d\Lie_U^{J-J_{s+1}}\D_d\D_aU_{J_{s+1}}^b+\sum_{J=J_1+J_2+J_3\atop |J_1|\gs 1,\;|J_3|=1}\Lie_U^{J_1}W^d\Lie_U^{J_2}\D_d\D_aU_{J_3}^b,
\end{align*}
which is of the form in \eqref{comm.LieUJDa} with $r=s+1$. Thus, we proved the identity by induction.	
\end{proof}

For $U\in \mc{U}$ and $|J|=r-1$, let
\begin{align*}
E_{r-1,\curl}(t)=&\langle \curl\Lie_U^J w, \curl\Lie_U^J w\rangle+\langle \curl\Lie_U^J \dot{w}, \curl\Lie_U^J \dot{w}\rangle\\
&+\langle \curl \Lie_B\Lie_U^Jw, \curl \Lie_B\Lie_U^Jw\rangle.
\end{align*}
Then, from \eqref{DtcurlLUJw} it follows
\begin{align*}
&\frac{1}{2}\frac{d}{dt}E_{r,\curl}(t)\\
= &\int_{\Omega}\dot{g}^{ab}g^{cd}(\curl\Lie_U^J w_{ad}\curl\Lie_U^J w_{bc}+\curl\Lie_U^J \dot{w}_{ad}\curl\Lie_U^J \dot{w}_{bc}\\
&\qquad\qquad\qquad+\curl \Lie_B\Lie_U^Jw_{ad}\curl \Lie_B\Lie_U^Jw_{bc})dy\\
&-\int_{\Omega}\Lie_B( g^{ab}g^{cd})\curl \Lie_B\Lie_U^Jw_{ad} \curl\Lie_U^J\dot{w}_{bc} dy\\
&+\langle \curl\Lie_U^J \dot{w}, \curl\Lie_U^J w\rangle+\langle \curl\Lie_U^J \dot{w}, D_t\curl\Lie_U^J \dot{w}- \curl \Lie_B^2\Lie_U^Jw\rangle\\
&+\Big\langle \curl \Lie_U^Jw, c_{J_1J_2}\Big(\Lie_U^{J_1}\D_c\omega_{ab}\Lie_U^{J_2}W^c+(\Lie_U^{J_1}\dot{g}_{bc}[\Lie_U^{J_2},\D_a]W^c-\Lie_U^{J_1}\dot{g}_{ac}[\Lie_U^{J_2},\D_b]W^c)\no\\
&\qquad+(\Lie_U^{J_1}\dot{g}_{bc}\D_a\Lie_U^{J_2}W^c-\Lie_U^{J_1}\dot{g}_{ac}\D_b\Lie_U^{J_2}W^c)\no\\
&\qquad+\Lie_U^{J_1}\Big\{[(\dot{g}_{eb}-\omega_{eb})\D_a{\D_c x^k}-(\dot{g}_{ea}-\omega_{ea})\D_b{\D_c x^k}]\frac{\D y^e}{\D x^k}\Big\}\Lie_U^{J_2}W^c\Big)\Big\rangle\\
&+\Big\langle \curl \Lie_B\Lie_U^Jw, c_{J_1J_2}\Lie_B\Big(\Lie_U^{J_1}\D_c\omega_{ab}\Lie_U^{J_2}W^c+(\Lie_U^{J_1}\dot{g}_{bc}[\Lie_U^{J_2},\D_a]W^c\no\\
&\qquad-\Lie_U^{J_1}\dot{g}_{ac}[\Lie_U^{J_2},\D_b]W^c)+(\Lie_U^{J_1}\dot{g}_{bc}\D_a\Lie_U^{J_2}W^c-\Lie_U^{J_1}\dot{g}_{ac}\D_b\Lie_U^{J_2}W^c)\no\\
&\qquad+\Lie_U^{J_1}\Big\{[(\dot{g}_{eb}-\omega_{eb})\D_a{\D_c x^k}-(\dot{g}_{ea}-\omega_{ea})\D_b{\D_c x^k}]\frac{\D y^e}{\D x^k}\Big\}\Lie_U^{J_2}W^c\Big)\Big\rangle.
\end{align*}
Thus, by \eqref{DtcurlDtLUJw}   and Lemma \ref{lem.11.3}, we get
\begin{align}
&\frac{d}{dt}E_{r-1,\curl}^{1/2}(t)\no\\
\ls &(1+\norm{\dot{g}}_{L^\infty(\Omega)}+\norm{\Lie_Bg}_{L^\infty(\Omega)})E_{r,\curl}^{1/2}+\norm{\curl\Lie_U^J\underline{F}}\no\\
&+\norm{\curl[\Lie_U^J,\Lie_B^2] w}+Cc_{J_1J_2}\Big(\norm{\curl\Lie_U^{J_1} \Lie_B^2g}_{L^\infty(\Omega)}\norm{\Lie_U^{J_2}W}\label{Er1c.1}\\
&+\norm{\Lie_U^{J_1}\D\omega}_{L^\infty(\Omega)}\norm{\Lie_U^{J_2}\dot{W}}+\norm{\curl\Lie_U^{J_1} \Lie_Bg}_{L^\infty(\Omega)}(\norm{\Lie_B\Lie_U^{J_2}W}\no\\
&+K_1 \tilde{c}_{J_2}^{J_{21}J_{22}}[\norm{B_{J_{21}}}_{L^\infty(\Omega)}(\norm{\curl\underline{W_{J_{22}}}}+\sum_{S\in\mc{S}}\norm{\Lie_SW_{J_{22}}})+\norm{W_{J_{22}}}\label{Er1c.2}\\
&\cdot(\norm{\curl\underline{B_{J_{21}}}}_{L^\infty(\Omega)}+\sum_{S\in\mc{S}}\norm{\Lie_SB_{J_{21}}}_{L^\infty(\Omega)}+\norm{[g]_1}_{L^\infty(\Omega)}\norm{B_{J_{21}}}_{L^\infty(\Omega)})])\no\\
&+(\norm{\Lie_U^{J_1}\Lie_B^2g}_{L^\infty(\Omega)}+\norm{\Lie_U^{J_1}\dot{g}}_{L^\infty(\Omega)}+\norm{\Lie_B\Lie_U^{J_1}\dot{g}}_{L^\infty(\Omega)})\no\\
&\qquad\cdot(\norm{[\Lie_U^{J_2},\D]W}+K_1(\norm{\curl\underline{\Lie_U^{J_2}W}}+\sum_{S\in\mc{S}}\norm{\Lie_S\Lie_U^{J_2}W}\label{Er1c.3}\\
&+\norm{[g]_1}_{L^\infty(\Omega)}\norm{\Lie_U^{J_2}W}))+\norm{\Lie_U^{J_1}\omega}_{L^\infty(\Omega)}(\norm{[\Lie_U^{J_2},\D ]\dot{W}}+K_1(\norm{\curl\underline{\Lie_U^{J_2}\dot{W}}}\label{Er1c.4}\\
&\qquad\qquad+\sum_{S\in\mc{S}}\norm{\Lie_S\Lie_U^{J_2}\dot{W}}+\norm{[g]_1}_{L^\infty(\Omega)}\norm{\Lie_U^{J_2}\dot{W}}))\no\\
&+[\norm{\Lie_U^{J_1}[\D B^b(\D_ag_{db}-\D_dg_{ab})+g_{db}\D_a\D B^b-g_{ab}\D_d\D B^b]}_{L^\infty(\Omega)}\no\\
&\qquad+\norm{\Lie_U^{J_1}[\delta_{il}\D_d x^l\D_aB\D_cx^i-\delta_{il}\D_a x^l\D_dB\D_cx^i]}_{L^\infty(\Omega)}](\norm{\Lie_B\Lie_U^{J_2}W}\no\\
&+K_1 \tilde{c}_{J_2}^{J_{21}J_{22}}[\norm{B_{J_{21}}}_{L^\infty(\Omega)}(\norm{\curl\underline{W_{J_{22}}}}+\sum_{S\in\mc{S}}\norm{\Lie_SW_{J_{22}}})+(\norm{\curl\underline{B_{J_{21}}}}_{L^\infty(\Omega)}\no\\
&+\sum_{S\in\mc{S}}\norm{\Lie_SB_{J_{21}}}_{L^\infty(\Omega)}+\norm{[g]_1}_{L^\infty(\Omega)}\norm{B_{J_{21}}}_{L^\infty(\Omega)})\norm{W_{J_{22}}}])\no\\
&+(\norm{\Lie_U^{J_1}\Lie_Bg}_{L^\infty(\Omega)}+\norm{\delta_{il}\Lie_U^{J_1}(B\D x^i\D  x^l)}_{L^\infty(\Omega)}+\norm{\Lie_U^{J_1}(g_{db}\D_cB^b)}_{L^\infty(\Omega)})\no\\
&\qquad\cdot(\norm{[\Lie_U^{J_2},\D]\Lie_BW}+K_1(\norm{\curl\underline{\Lie_U^{J_2}\Lie_BW}}+\sum_{S\in\mc{S}}\norm{\Lie_S\Lie_U^{J_2}\Lie_BW}\label{Er1c.5}\\
&\qquad+\norm{[g]_1}_{L^\infty(\Omega)}(\norm{\Lie_B\Lie_U^{J_2}W}+K_1 \tilde{c}_{J_2}^{J_{21}J_{22}}[\norm{B_{J_{21}}}_{L^\infty(\Omega)}(\norm{\curl\underline{W_{J_{22}}}}\no\\
&\qquad+\sum_{S\in\mc{S}}\norm{\Lie_SW_{J_{22}}})+(\norm{\curl\underline{B_{J_{21}}}}_{L^\infty(\Omega)}+\sum_{S\in\mc{S}}\norm{\Lie_SB_{J_{21}}}_{L^\infty(\Omega)}\no\\
&\qquad+\norm{[g]_1}_{L^\infty(\Omega)}\norm{B_{J_{21}}}_{L^\infty(\Omega)})\norm{W_{J_{22}}}])))\no\\
&+ (\norm{\Lie_U^{J_1}\D \omega}_{L^\infty(\Omega)}+\norm{\Lie_B\Lie_U^{J_1}\D \omega}_{L^\infty(\Omega)})(\norm{\Lie_U^{J_2}W}+\norm{\Lie_B\Lie_U^{J_2}W})\no\\
&\qquad+\norm{\Lie_U^{J_1}\Big\{[(\dot{g}_{eb}-\omega_{eb})\D_a{\D_c x^k}-(\dot{g}_{ea}-\omega_{ea})\D_b{\D  x^k}]\frac{\D y^e}{\D x^k}\Big\}}_{L^\infty(\Omega)}\no\\
&\qquad\cdot(\norm{\Lie_U^{J_2}W}+\norm{\Lie_B\Lie_U^{J_2}W})+\norm{\Lie_U^{J_2}W}\no\\
&\qquad\cdot\lnorm{\Lie_B\Lie_U^{J_1}\Big\{[(\dot{g}_{eb}-\omega_{eb})\D_a{\D_c x^k}-(\dot{g}_{ea}-\omega_{ea})\D_b{\D  x^k}]\frac{\D y^e}{\D x^k}\Big\}}_{L^\infty(\Omega)}\Big).\no
\end{align}

We first consider the term  $|\curl[\Lie_U^I,\Lie_B^2]w|$ in the line labeled \eqref{Er1c.1}. It holds
\begin{align*}
[\Lie_U^I,\Lie_B^2]w_a=&[\Lie_U^I,\Lie_B]\Lie_B w_a+\Lie_B[\Lie_U^I,\Lie_B]w_a\no\\
=&\tilde{c}_{I}^{I_1I_2}[B^b_{I_1}\D_b\Lie_U^{I_2}\Lie_Bw_{a}+(\D_aB^c_{I_1})\Lie_U^{I_2}\Lie_Bw_{c}\no\\
&\qquad+\Lie_B[B^b_{I_1}\D_b\Lie_U^{I_2}w_{a}+(\D_aB^c_{I_1})\Lie_U^{I_2}w_{c}]]\no\\
=&\tilde{c}_{I}^{I_1I_2}[B^b_{I_1}\D_b\Lie_U^{I_2}\Lie_Bw_{a}+(\D_aB^c_{I_1})\Lie_U^{I_2}\Lie_Bw_{c}+\Lie_BB^b_{I_1}\D_b\Lie_U^{I_2}w_{a}\no\\
&\qquad+B^b_{I_1}\Lie_B\D_b\Lie_U^{I_2}w_{a}+\Lie_B(\D_aB^c_{I_1})\Lie_U^{I_2}w_{c}+(\D_aB^c_{I_1})\Lie_B\Lie_U^{I_2}w_{c}],
\end{align*}
which yields
\begin{align*}
&(\curl[\Lie_U^I,\Lie_B^2]w)_{ad}\no\\
=&\tilde{c}_{I}^{I_1I_2}[\D_aB^b_{I_1}\D_b\Lie_U^{I_2}\Lie_Bw_{d}-\D_dB^b_{I_1}\D_b\Lie_U^{I_2}\Lie_Bw_{a}+B^b_{I_1}\D_b(\curl\Lie_U^{I_2}\Lie_Bw)_{ad}\no\\
&+(\D_dB^c_{I_1})\D_a(\Lie_U^{I_2}\Lie_Bw_{c}+\Lie_B\Lie_U^{I_2}w_{c})-(\D_aB^c_{I_1})\D_d(\Lie_U^{I_2}\Lie_Bw_{c}+\Lie_B\Lie_U^{I_2}w_{c})\no\\
&+\D_a\Lie_BB^b_{I_1}\D_b\Lie_U^{I_2}w_{d}-\D_d\Lie_BB^b_{I_1}\D_b\Lie_U^{I_2}w_{a}+\D_aB^b_{I_1}\Lie_B\D_b\Lie_U^{I_2}w_{d}-\D_dB^b_{I_1}\Lie_B\D_b\Lie_U^{I_2}w_{a}\no\\
&+\Lie_BB^b_{I_1}\D_b(\curl\Lie_U^{I_2}w)_{ad}+B^b_{I_1}\D_a\Lie_B\D_b\Lie_U^{I_2}w_{d}-B^b_{I_1}\D_d\Lie_B\D_b\Lie_U^{I_2}w_{a}\no\\
&+(\D_a\Lie_B\D_dB^c_{I_1}-\D_d\Lie_B\D_aB^c_{I_1})\Lie_U^{I_2}w_{c}+(\Lie_B\D_dB^c_{I_1})\D_a\Lie_U^{I_2}w_{c}-(\Lie_B\D_aB^c_{I_1})\D_d\Lie_U^{I_2}w_{c}].
\end{align*}
Since 
\begin{align*}
[\Lie_B,\D_a]w_b=-\D_a\D_b B^cw_c,
\end{align*}
 one gets
\begin{align*}
&\abs{\curl[\Lie_U^I,\Lie_B^2]w}\no\\
\ls &K_1 \tilde{c}_{I}^{I_1I_2}[\abs{\D B_{I_1}}(|\D\Lie_U^{I_2}\Lie_Bw|+|\D\Lie_B\Lie_U^{I_2}w|+|\D^2B||\Lie_U^{I_2}w|)\no\\
&+|B_{I_1}||\D (\curl\Lie_U^{I_2}\Lie_Bw)|+|\Lie_BB_{I_1}||\D(\curl\Lie_U^{I_2}w)|+|B_{I_1}|(|\D^2\Lie_B\Lie_U^{I_2}w|\\
&+|\D^3B||\Lie_U^{I_2}w|+|\D^2B||\D\Lie_U^{I_2}w|)+(|\D\Lie_BB_{I_1}|+|\D^2B|| B_{I_1}|)|\D\Lie_U^{I_2}w|\\
&+(|\D^2\Lie_B B_{I_1}|+|\D^3B||B_{I_1}|+|\D^2B||\D B_{I_1}|)|\Lie_U^{I_2}w|].
\end{align*}
Due to 
\begin{align*}
\Lie_U^Jw_a=\Lie_U^J(g_{ab}W^b)=g_{ab}\Lie_U^JW^b+\tilde{c}_{J_1J_2}^J g_{ab}^{J_1}\Lie_U^{J_2}W^b,\quad g_{ab}^J=\Lie_U^J g_{ab},
\end{align*}
where the sum is over all $J_1+J_2=J$, and $\tilde{c}_{J_1J_2}^J=1$ for $|J_2|<|J|$, and $\tilde{c}_{J_1J_2}^J=0$ for $|J_2|=|J|$, 
from Lemma \ref{lem.11.3}, it follows that
\begin{align*}
|\Lie_U^{I_2}w|\ls&|g||\Lie_U^{I_2}W|+\tilde{c}_{I_{21}I_{22}}^{I_2}|g^{I_{21}}||\Lie_U^{I_{22}}W|, \\
|\D\Lie_U^{I_2}w|\ls&|\D g||\Lie_U^{I_2}W|+|g||\D\Lie_U^{I_2}W|+\tilde{c}_{I_{21}I_{22}}^{I_2}|\D g^{I_{21}}||\Lie_U^{I_{22}}W|+\tilde{c}_{I_{21}I_{22}}^{I_2}|g^{I_{21}}||\D\Lie_U^{I_{22}}W|\\
\ls &K_1 [|g|+|\D g|]\Big(|\curl\underline{\Lie_U^{I_2}W}|+\sum_{S\in\mc{S}}|\Lie_S \Lie_U^{I_2}W|+[g]_1|\Lie_U^{I_2}W|\Big)\\
&\!\!\!+K_1 \tilde{c}_{I_{21}I_{22}}^{I_2}[|g^{I_{21}}|+|\D g^{I_{21}}|]\Big(|\curl\underline{\Lie_U^{I_{22}}W}|+\sum_{S\in\mc{S}}|\Lie_S \Lie_U^{I_{22}}W|+[g]_1|\Lie_U^{I_{22}}W|\Big),
\end{align*}
and 
\begin{align*}
|\D\Lie_U^{I_2}\Lie_Bw_a|=& |\D\Lie_U^{I_2}(\Lie_Bg_{ab}W^b+g_{ab}\Lie_BW^b)|\\
\ls& \tilde{c}_{I_{21}I_{22}}^{I_2}(|\D \Lie_U^{I_{21}}\Lie_B g||\Lie_U^{I_{22}}W|+| \Lie_U^{I_{21}}\Lie_B g||\D\Lie_U^{I_{22}}W|\\
&+|\D \Lie_U^{I_{21}} g||\Lie_U^{I_{22}}\Lie_BW|+| \Lie_U^{I_{21}} g||\D\Lie_U^{I_{22}}\Lie_BW|)\\
\ls &K_1 \tilde{c}_{I_{21}I_{22}}^{I_2}(| \Lie_U^{I_{21}}\Lie_B g|+|\D \Lie_U^{I_{21}}\Lie_B g|)\\
&\qquad\cdot\Big(|\curl\underline{\Lie_U^{I_{22}}W}|+\sum_{S\in\mc{S}}|\Lie_S \Lie_U^{I_{22}}W|+[g]_1|\Lie_U^{I_{22}}W|\Big)\\
&\!\!\!+K_1 \tilde{c}_{I_{21}I_{22}}^{I_2}(| \Lie_U^{I_{21}} g|+|\D \Lie_U^{I_{21}} g|)\Big(|\curl\underline{\Lie_U^{I_{22}}\Lie_BW}|+\sum_{S\in\mc{S}}|\Lie_S \Lie_U^{I_{22}}W|\\
&\!\!\!\!\!+[g]_1\tilde{c}_{I_{221}I_{222}}^{I_{22}}\Big[|\Lie_U^{I_{221}}B|\Big(|\curl\underline{\Lie_U^{I_{222}}W}|+\sum_{S\in\mc{S}}|\Lie_S \Lie_U^{I_{222}}W|+[g]_1|\Lie_U^{I_{222}}W|\Big)\\
&\qquad+\Big(|\curl\underline{\Lie_U^{I_{221}}B}|+\sum_{S\in\mc{S}}|\Lie_S \Lie_U^{I_{221}}B|\Big)|\Lie_U^{I_{222}}W|\Big]\Big),
\end{align*}
since
\begin{align}\label{TIB}
\abs{[\Lie_T^I,\Lie_B]W}\ls& \tilde{c}_{I}^{I_1I_2}[\abs{B_{I_1}}\abs{\D W_{I_2}}+\abs{\D B_{I_1}}\abs{W_{I_2}}]\\
\ls &\tilde{c}_{I}^{I_1I_2}K_1\Bigg[\abs{B_{I_1}}\left( \sum_{S\in\mc{R}}|\Lie_S W_{I_2}|+|W_{I_2}|\right)+\abs{\D B_{I_1}}\abs{W_{I_2}}\Bigg].\no
\end{align}

Now, we have to express the term, like $|\curl\underline{W_{I}}|=|\curl\underline{\Lie_U^{I}W}|$ in the above inequality and in the line labeled \eqref{Er1c.2} and other lines, in term of $w$. By Lemma \ref{lem.11.3}, we have
\begin{align*}
&|\curl\underline{\Lie_U^{I}W}_a|=|\curl g_{ab}\Lie_U^{I}W^b|\\
\ls& |\curl \Lie_U^{I}w_a|+\tilde{c}^{I}_{I_1I_2}|\curl (g^{I_1}_{ab}\Lie_U^{I_2}W^b) |\\
\ls & |\curl \Lie_U^{I}w_a|+\tilde{c}^{I}_{I_1I_2}|\D_d (g^{I_1}_{ab}\Lie_U^{I_2}W^b)-\D_a(g^{I_1}_{db}\Lie_U^{I_2}W^b) |\\
\ls &|\curl \Lie_U^{I}w_a|+2\tilde{c}^{I}_{I_1I_2}(|\D g^{I_1}||\Lie_U^{I_2}W|+|g^{I_1}||\D \Lie_U^{I_2}W|)\\
\ls &|\curl \Lie_U^{I}w_a|+K_1 \tilde{c}^{I}_{I_1I_2}(|\D g^{I_1}|+|g^{I_1}|)\Big(|\curl\underline{\Lie_U^{I_{2}}W}|+\sum_{S\in\mc{S}}|\Lie_S \Lie_U^{I_{2}}W|+[g]_1|\Lie_U^{I_{2}}W|\Big).
\end{align*}
The term $|\curl\underline{\Lie_U^{I}\Lie_BW}|$ in \eqref{Er1c.5} can be estimated as a similar argument as above by regarding $B$ as a tangential vector field of form $U$.
By \eqref{comm.LieUJDa}, for the term $[\Lie_U^{J_2},\D]W$ in \eqref{Er1c.3}, we get
\begin{align*}
|[\Lie_U^{J_2},\D]W|\ls &|W||\Lie_U^{J_2-J_{2,|J_2|}}\D^2U_{J_{2,|J_2|}}|+\sum_{J_2=I_1+I_2+I_3\atop |I_3|=1}\sgn(|I_1|)|\Lie_U^{I_2}\D^2U_{I_3}||\Lie_U^{I_1}W|,
\end{align*}
and a similar estimate holds for the term $[\Lie_U^{J_2},\D]\dot{W}$ in \eqref{Er1c.4}. Similarly, for the term $[\Lie_U^{J_2},\D]\Lie_BW$ in \eqref{Er1c.5}, we have with the help of \eqref{TIB}
\begin{align*}
&|[\Lie_U^{J_2},\D]\Lie_BW|\ls |\Lie_BW||\Lie_U^{J_2-J_{2,|J_2|}}\D^2U_{J_{2,|J_2|}}|+\sum_{J_2=I_1+I_2+I_3\atop |I_3|=1}\sgn(|I_1|)|\Lie_U^{I_2}\D^2U_{I_3}|\\
&\cdot\bigg[|\Lie_B\Lie_U^{I_1}W|+K_1 \tilde{c}_{I}^{I_{11}I_{12}}\abs{B_{I_{11}}}\Big(|\curl \underline{W_{I_{12}}}|+\sum_{S\in\mc{S}}|\Lie_S W_{I_{12}}|+[g]_1|W_{I_{12}}|\Big)\no\\
&\qquad\qquad\quad+\Big( |\curl \underline{B_{I_{11}}}|+ \sum_{S\in\mc{S}}|\Lie_S B_{I_{11}}|\Big)\abs{W_{I_{12}}}\bigg].
\end{align*}
For the term $|\Lie_S\Lie_U^{J_2}\Lie_B W|$, we can use \eqref{TIB} to get estimates.

For convenience, we introduce some new norms and notation.
\begin{definition}
	For any family $\mc{V}$ of our families  of vector fields, let
\begin{align*}
	\norm{W}_r^{\mc{V}}=&\norm{W(t)}_{\mc{V}^r(\Omega)}=\sum_{|I|\ls r,\, T\in\mc{V}} \left(\int_{\Omega} |\Lie_U^I W(t,y)|^2 dy\right)^{1/2},\\	\norm{W}_{r,B}^{\mc{V}}=&\norm{W(t)}_{\mc{V}_B^r(\Omega)}=\sum_{|I|\ls r,\, T\in\mc{V}} \left(\int_{\Omega} |\Lie_B\Lie_U^I W(t,y)|^2 dy\right)^{1/2},
\end{align*}
and
\begin{align}\label{Cr}
		\begin{aligned}
	C_r^{\mc{V}}=&\sum_{|J|\ls r-1,\, J\in\mc{V}} \left(\int_{\Omega} (|\curl \Lie_U^J\dot{w}|^2+|\curl \Lie_B\Lie_U^Jw|^2+|\curl \Lie_U^Jw|^2) dy\right)^{1/2},\\
	C_0^{\mc{V}}=&0.
	\end{aligned}
\end{align}
\end{definition}
Note that the norm $\norm{W(t)}_{\mc{R}^r(\Omega)}$ is equivalent to the usual Sobolev norm in the Lagrangian coordinates.

\begin{definition}
For $\mc{V}$ any  of our families of vector fields and $\beta$ a function, a $1$-form, a $2$-form, or a vector field, let $|\beta|_s^{\mc{V}}$ be as in Definition \ref{def.betasV} and set
\begin{align*}
\begin{aligned}
\norm{\beta}_{s,\infty}^{\mc{V}}=&\norm{|\beta|_s^{\mc{V}}}_{L^\infty(\Omega)},\\
[[g]]_{s,\infty}^{\mc{V}}=&\sum_{s_1+\cdots+s_k\ls s,\, s_i\gs 1}\norm{g}_{s_1,\infty}^{\mc{V}}\cdots \norm{g}_{s_k,\infty}^{\mc{V}}, \quad [[g]]_{0,\infty}^{\mc{V}}=1,
\end{aligned}
\end{align*}
where the sum is over all combinations with $s_i\gs 1$. Furthermore, let
\begin{align}\label{mr}
\begin{aligned}
m_r^{\mc{V}}=&[[g]]_{r,\infty}^{\mc{V}},\\ 
\dot{m}_r^{\mc{V}}=&\sum_{s+u\ls r} [[g]]_{s,\infty}^{\mc{V}}(\norm{\dot{g}}_{u,\infty}^{\mc{V}}+\norm{\Lie_B g}_{u,\infty}^{\mc{V}}+\norm{\Lie_B^2 g}_{u,\infty}^{\mc{V}}+\norm{\omega}_{u,\infty}^{\mc{V}}),\\
\bar{m}_r^{\mc{V}}=&\sum_{s\ls r}\left(\norm{B}_{s+2,\infty}^{\mc{V}}+\norm{\D x}_{s,\infty}^{\mc{V}}+\norm{\D^2 x}_{s,\infty}^{\mc{V}}+\lnorm{\frac{\D y}{\D x}}_{s,\infty}^{\mc{V}}\right).
\end{aligned}
\end{align}
\end{definition}

Let $F_{r,\curl}^{\mc{U}}=\norm{\curl \underline{F}}_{\mc{U}^{r-1}(\Omega)}$. Then, it follows from the above arguments in this subsection that
\begin{align}\label{CrU}
\abs{\frac{d}{dt}C_r^{\mc{U}}}\ls K_1 \sum_{s=0}^r (\dot{m}_{r-s}^{\mc{U}}+\bar{m}_{r-s}^{\mc{U}})(C_s^{\mc{U}}+E_s^{\mc{T}})+F_{r,\curl}^{\mc{U}},
\end{align}
where $E_s^{\mc{T}}$ is the energy of the tangential derivatives defined in \eqref{EsV}. Here, we note that the same inequalities hold with $\mc{U}$ and $\mc{T}$ replaced by $\mc{R}$ and $\mc{S}$, respectively. Thus, by the Gronwall inequality, we obtain for $r\gs 1$
\begin{align*}
C_r^{\mc{U}} \ls & K_1e^{\int_0^t K_1(\dot{m}_{0}^{\mc{U}}+\bar{m}_{0}^{\mc{U}})d\tau}\int_0^t\Big(\sgn(r-1)\sum_{s=1}^{r-1}(\dot{m}_{r-s}^{\mc{U}}+\bar{m}_{r-s}^{\mc{U}})C_s^{\mc{U}}\no\\
&\qquad\qquad+\sum_{s=0}^{r}(\dot{m}_{r-s}^{\mc{U}}+\bar{m}_{r-s}^{\mc{U}})E_s^{\mc{T}}+F_{r,\curl}^{\mc{U}}\Big) d\tau.
\end{align*}
Since we have already proved a bound for $E_s^{\mc{T}}$ in Proposition \ref{prop.ErT}, it inductively follows that $C_r^{\mc{U}}$ is bounded. By Lemma \ref{lem.divcurl}, we obtain
\begin{align}\label{L12.20}
\norm{W(t)}_{\mc{U}^r(\Omega)}+\norm{\dot{W}(t)}_{\mc{U}^r(\Omega)}+\norm{\Lie_BW(t)}_{\mc{U}^r(\Omega)}\ls K_1\sum_{s=0}^r m_{r-s}^{\mc{U}}(C_s^{\mc{U}}+E_s^{\mc{T}}).
\end{align}
Therefore, we have the following estimates.
\begin{proposition}\label{prop.normal}
	Suppose that $x,P\in C^{r+2}([0,T]\times\Omega)$, $B\in C^{r+2}(\Omega)$, $P|_{\Gamma}=0$, $\nb_NP|_\Gamma\ls -c_0<0$, $B^aN_a|_\Gamma=0$ and $\dv V=0$, where $V=D_t x$. Then, there is a constant $C=C(x,P,B)$ depending only on the norm of $(x,P,B)$, a lower bound for $c_0$, and an upper bound for $T$ such that if $E_s^{\mc{T}}(0)=C_s^{\mc{U}}(0)=0$ for $s\ls r$, then
	\begin{align*}
	\norm{W}_r^{\mc{U}}+\norm{\dot{W}}_r^{\mc{U}}+\norm{\Lie_BW}_r^{\mc{U}}+E_r^{\mc{T}}\ls C\int_0^t\norm{F}_r^{\mc{U}}d\tau, \quad \text{ for } t\in [0,T].
	\end{align*}
\end{proposition}

\section{The smoothed-out equation and existence of weak solutions}\label{sec.weak}

\subsection{The smoothed-out normal operator}
In order to prove the existence of solutions, we need to replace the normal operator $\A$ by a sequence $\A^\eps$ of bounded symmetric and positive operators that converge to $\A$ as $\eps\to 0$. 

Let $\rho=\rho(d)$ be a smooth function of the distance $d=d(y)=\dist(y,\Gamma)$ such that
\begin{align*}
\rho'\gs 0, \quad \rho(d)=d \text{ for } d\ls \frac{1}{4}, \quad \text{ and } \rho(d)=\frac{1}{2} \text{ for } d\gs \frac{3}{4}.
\end{align*}
Let $\chi(\rho)$ be a smooth function such that
\begin{align*}
\chi'(\rho)\gs 0, \quad \chi(\rho)=0 \text{ for } \rho\ls \frac{1}{4}, \quad \text{ and } \chi(\rho)=1 \text{ for } \rho\gs \frac{3}{4}.
\end{align*}
For a function $f$ vanishing on the boundary, we define
\begin{align*}
\A^\eps_f W^a=\Pdv(-g^{ab}\chi_\eps(\rho)\D_b(f\rho^{-1}(\D_c\rho)W^c)),
\end{align*}
where $\chi_\eps\eps(\rho)=\chi(\rho/\eps)$. The integration by parts gives
\begin{align}\label{UAepsW}
\langle U,\A^\eps_fW\rangle=\int_{\Omega} f\rho^{-1}\chi_\eps'(\rho)(U^a\D_a\rho)(W^b\D_b\rho)dy,
\end{align}
which yields the symmetry of $\A^\eps_f$. In particular, $\A^\eps=\A^\eps_P$ is positive if $P\gs 0$, at least close to the boundary, i.e.,
\begin{align*}
\langle W,\A^\eps W\rangle \gs 0.
\end{align*}
We have another expression for $\A^\eps_f$: 
\begin{align*}
\A^\eps_f W^a=\Pdv(g^{ab}\chi_\eps'(\rho)(\D_b\rho)f\rho^{-1}(\D_c\rho)W^c).
\end{align*}
Since the projection is continuous on $H^r(\Omega)$, if the metric and pressure are sufficiently regular, one gets, as in \cite{L1,L2}, that
\begin{align}\label{Hrest}
\sum_{j=0}^k \norm{D_t^j \A^\eps W}_{H^r(\Omega)}\ls C_{\eps,r,k}\sum_{j=0}^k\norm{D_t^jW}_{H^r(\Omega)}.
\end{align} 
Moreover, we have
\begin{align*}
\norm{\A^\eps W-\A W}^2=&\langle \A^\eps W-\A W,\A^\eps W-\A W\rangle\\
=&\langle \A^\eps W-\A W,\A^\eps W\rangle-\langle \A^\eps W-\A W,\A W\rangle\\
=&-\int_{\Omega} (\A^\eps W^a-\A W^a)\chi_\eps(\rho)\D_a(P\rho^{-1}(\D_c\rho)W^c)dy\\
&+\int_{\Omega}(\A^\eps W^a-\A W^a)\D_a(\D_cPW^c)dy\\
=&-\int_{\Omega} (\A^\eps W^a-\A W^a)\chi_\eps(\rho)\D_a[(P\rho^{-1}\D_c\rho-\D_cP)W^c]dy\\
&+\int_{\Omega}(\A^\eps W^a-\A W^a)(1-\chi_\eps(\rho))\D_a(\D_cPW^c)dy\\
=&\int_{\Omega} (\A^\eps W^a-\A W^a)\chi_\eps'(\rho)\D_a\rho[(P\rho^{-1}\D_c\rho-\D_cP)W^c]dy\\
&+\int_{\Omega}(\A^\eps W^a-\A W^a)(1-\chi_\eps(\rho))\D_a(\D_cPW^c)dy
\end{align*}
due to $P\rho^{-1}\D_c\rho=\D_cP$ on the boundary, which yields 
\begin{align*}
\norm{\A^\eps W-\A W}\ls& \norm{\chi_\eps'(\rho)}_{L^\infty(\Omega)}\norm{\D_a\rho(P\rho^{-1}\D_c\rho-\D_cP)W^c}\\
&+\norm{(1-\chi_\eps(\rho))}_{L^\infty(\Omega)}\norm{\D_a(\D_cPW^c)}\to 0, \text{ as } \eps\to 0,
\end{align*}
since $\chi_\eps'(\rho)\to 0$ and $\chi_\eps(\rho)\to 1$ in $L^\infty(\Omega)$ as $\eps\to 0$. Thus, we obtain 
\begin{align}\label{Aepslimit}
\A^\eps U\to \A U \text{ in } L^2(\Omega),\quad \text{if } U\in H^1(\Omega).
\end{align}

As in \eqref{est.A}, it holds
\begin{align}\label{est.Aeps}
\abs{\langle U,\A_{fP}^\eps W\rangle}\ls & \langle U,\A_{fP}^\eps U\rangle^{1/2}\langle W,\A_{fP}^\eps W\rangle^{1/2}\no\\
\ls &\norm{f}_{L^\infty(\Omega\setminus \Omega^\eps)}\langle U,\A^\eps U\rangle^{1/2}\langle W,\A^\eps W\rangle^{1/2},
\end{align}
where 
\begin{align}\label{Omegaeps}
\Omega^\eps=\{y\in\Omega: \dist(y,\Gamma)>\eps\}.
\end{align}
In fact, it suffices to take the supremum over the set where $d(y)\ls \eps$ since $\chi_\eps'=0$ when $d(y)\gs \eps$. The only difference with \eqref{est.A} is that the supremum is over a small neighborhood of the boundary instead of on the boundary. Since $P$ vanishes on the boundary, $P>0$ in the interior, and $\nb_NP\ls -c_0<0$ on the boundary, it follows that $\dot{P}=D_tP$ vanishes on the boundary and $\dot{P}/P$ is a smooth function. Let $\dot{\A}^\eps=\A_{\dot{P}}^\eps$ be the time derivative of the operator $\A^\eps$, which satisfies
\begin{align}\label{dotAeps}
\abs{\langle W,\dot{\A}^\eps W\rangle}\ls \norm{\dot{P}/P}_{L^\infty(\Omega\setminus \Omega^\eps)}\langle W,\A^\eps W\rangle.
\end{align}

The commutators between $\A_f^\eps$ and the Lie derivatives $\Lie_T$ with respect to tangential vector fields $T$ are basically the same as for $\A$. Since $Td=0$ for $T\in\mc{T}_0=\mc{S}_0\cup\{D_t\}$, we have
\begin{align}\label{comm.AL}
\Pdv(g^{ca}\Lie_T(g_{ab}\A_f^\eps W^b))=\A_f^\eps \Lie_T W^c+\A_{Tf}^\eps W^c.
\end{align}
In order to get additional regularity in the interior, we include the vector fields $\mc{S}_1$ that span the tangent space in the interior. The vector fields in $\mc{S}_1$ satisfy $S\rho=\Lie_S\rho=0$ when $d\ls d_0/2$. Due to $\chi_\eps'(\rho)=0$ when $d\gs \eps$, the above relation \eqref{comm.AL} holds for these as well if we assume that $\eps\ls d_0/2$.

Now, it remains to estimate the curl of $\A^\eps$. Although the curl of $\A$ vanishes, it is not the case for the curl of $\A^\eps$. However, it will vanish away from the boundary. Let $(\underline{\A}^\eps W)_a=g_{ab}\A^\eps W^b$,  we have
\begin{align*}
(\underline{\A}^\eps W)_a=&g_{ab}\Pdv(-g^{bd}\chi_\eps(\rho)\D_d(P\rho^{-1}(\D_c\rho)W^c))\no\\
=&-\chi_\eps(\rho)\D_a(P\rho^{-1}(\D_c\rho)W^c)-\D_a q_1,
\end{align*}
for some function $q_1$ vanishing on the boundary and determined so that the divergence vanishes. Then, when $d(y)\gs \eps$, we get $\chi_\eps'(\rho)=0$ and
\begin{align}\label{curlAeps}
(\curl \underline{\A}^\eps W)_{ab}=&\D_a(\underline{\A}^\eps W)_b-\D_b(\underline{\A}^\eps W)_a\\
=&-\D_a(\chi_\eps(\rho)\D_b(P\rho^{-1}(\D_c\rho)W^c))+\D_b(\chi_\eps(\rho)\D_a(P\rho^{-1}(\D_c\rho)W^c))\no\\
=&-\chi_\eps'(\rho)[\D_a\rho\D_b(P\rho^{-1}(\D_c\rho)W^c)-\D_b\rho\D_a(P\rho^{-1}(\D_c\rho)W^c)]\no\\
=&0.\no
\end{align}

\subsection{The smoothed-out equation and existence of weak solutions} \label{sec.weaksol}

We introduce the following $\eps$ smoothed-out linear equation
\begin{subequations}\label{epsIVP}
	\begin{align}
	&\ddot{W}_\eps^a-\Lie_B^2W_\eps^a+\A^\eps W_\eps^a+\dot{\mc{G}}\dot{W}_\eps^a-\mc{C}\dot{W}_\eps^a+\mc{X}\Lie_BW_\eps^a=F^a,\label{eps.IVP}\\
	&W_\eps|_{t=0}=0,\quad \dot{W}_\eps|_{t=0}=0.\label{eps.data}
	\end{align}
\end{subequations}
It is a wave equation with variable coefficients, one can get the existence of weak solutions in $H^1(\Omega)$ by standard methods and noticing that $B^aN_a=0$ on the boundary, or in $H^r(\Omega)$ by \eqref{Hrest}, since all operators are bounded and $\Lie_B$ can be regarded as the first-order derivative with respect to spatial variables.

In order to obtain the additional regularity in time as well, we need to apply more time derivatives using \eqref{Hrest} and \eqref{PdvHr}, the initial data for these vanish as well since we constructed $F$ in \eqref{epsIVP} vanishing to any given order. If the initial data, encoded in $F$, are smooth, we hence have a smooth solution of the $\eps$ approximate linear equation. 

We will show that $W_\eps\to W$ weakly in $L^2$, where $W\in H^r(\Omega)$ for some large $r$. From the weak convergence, it will follow that $W$ is a weak solution, and then from the additional regularity of $W$, we can obtain that it is indeed a classical solution; hence the a priori bounds in the earlier section hold.

The norm of $\A^\eps$ tends to infinity as $\eps\to 0$, but we can include it in the energy because it is a positive operator. The energy will be the same as before with $\A$ replaced by $\A^\eps$, so \eqref{energy} becomes
\begin{align}\label{energyeps}
E^\eps(t)=\langle \dot{W}_\eps, \dot{W}_\eps\rangle+\langle (\A^\eps +I)W_\eps,W_\eps\rangle+\langle \Lie_BW_\eps,\Lie_BW_\eps\rangle.
\end{align}
Since $D_td=0$, it follows from taking the time derivative of \eqref{UAepsW}, with $f=P$, that
\begin{align*}
\frac{d}{dt}\langle \A^\eps W_\eps,W_\eps\rangle=2\langle \A^\eps W_\eps,\dot{W}_\eps\rangle+\langle \A^\eps_{\dot{P}} W_\eps,W_\eps\rangle,
\end{align*}
where the last term is bounded by \eqref{dotAeps}. Thus, by \eqref{dotE0}, one has
\begin{align*}
|\dot{E}^\eps|\ls \left(1+\lnorm{\frac{\dot{P}}{ P}}_{L^\infty(\Omega)}+2\norm{\dot{g}}_{L^\infty(\Omega)}+2\norm{g}_{L^\infty(\Omega)}\norm{\D B}_{L^\infty(\Omega)}\right)E^\eps+2\sqrt{E^\eps}\norm{F},
\end{align*}
from which we obtain a uniform bound for $t\in[0,T]$ independent of $\eps$, i.e., $E^\eps(t)\ls C$.

Since $\norm{W_\eps}\ls C$, we can choose a subsequence $W_{\eps_n}\rightharpoonup W$ weakly in the inner product. Now, we show that the limit $W$ is a weak solution of the equation. Multiplying \eqref{eps.IVP} by a smooth divergence-free vector field $U$ that vanished for $t\gs T$ and integrating by parts, we have
\begin{align*}
&\int_0^T \int_{\Omega}g_{ab}U^bF^a dydt\\
=&\int_0^T\int_{\Omega} g_{ab}(\ddot{W}_\eps^a-\Lie_B^2W_\eps^a+\A^\eps W_\eps^a+\dot{\mc{G}}\dot{W}_\eps^a-\mc{C}\dot{W}_\eps^a+\mc{X}\Lie_BW_\eps^a) U^b dydt\no\\
=&\int_0^T\int_{\Omega} (D_t(g_{ab}\dot{W}_\eps^a) U^b-g_{ab}\Lie_B^2W_\eps^a U^b+\chi_\eps'(\rho)(\D_b\rho)P\rho^{-1}(\D_c\rho)W^c U^b-\omega_{bc}\dot{W}_\eps^c U^b\\
&-2\delta_{il}\D_b x^lB\D_cx^i\Lie_BW_\eps^c U^b-2g_{ab}\D_c B^a\Lie_BW_\eps^c U^b) dydt\\
=&-\int_0^T\int_{\Omega}g_{ab}\dot{W}_\eps^a\dot{U}^b dydt+\int_0^T \int_{\Omega}((\Lie_Bg_{ab})U^b+g_{ab}\Lie_BU^b)\Lie_B W_\eps^a  dydt\\
&+\int_0^T \int_{\Omega} g_{ab}\A^\eps U^b W_\eps^a  dydt+\int_0^T \int_{\Omega}(\dot{\omega}_{bc}U^b+\omega_{bc}\dot{U}^b)W_\eps^c dydt\\
&+2\int_0^T \int_{\Omega}\delta_{il}\D_b x^lB\D_cx^iW_\eps^c \Lie_BU^b dydt+2\int_0^T \int_{\Omega}\Lie_B(\delta_{il}\D_b x^lB\D_cx^i)W_\eps^c U^b dydt\\
&+2\int_0^T \int_{\Omega}\Lie_B(g_{ab}\D_c B^a)U^bW_\eps^c dydt+2\int_0^T \int_{\Omega}g_{ab}\D_c B^a\Lie_BU^bW_\eps^c dydt\\
=&\int_0^T\int_{\Omega}g_{ab}(\ddot{U}^b-\Lie_B^2U^b+\A^\eps U^b +\dot{\mc{G}}\dot{U}^b-\mc{C}\dot{U}^b+\mc{X}\Lie_BU^b) W_\eps^a  dydt\\
&+\int_0^T \int_{\Omega}\dot{\omega}_{bc}U^bW_\eps^c dydt+\int_0^T \int_{\Omega}\Lie_B(\delta_{il}\D_b x^lB\D_cx^i-\delta_{il}B\D_b x^l\D_cx^i)U^bW_\eps^c  dydt\\
&+\int_0^T \int_{\Omega}\Lie_B(g_{ab}\D_c B^a-g_{ac}\D_b B^a)U^bW_\eps^c dydt.
\end{align*}
From \eqref{Aepslimit}, we know that $\A^\eps U$ converges to $\A U$ strongly in the norm if $U\in H^1$. Because $W_{\eps_n}\rightharpoonup W$ weakly, this proves that we have a weak solution $W$ of the equation
\begin{align*}
&\int_0^T\int_{\Omega}g_{ab}(\ddot{U}^b-\Lie_B^2U^b+\A U^b +\dot{\mc{G}}\dot{U}^b-\mc{C}\dot{U}^b+\mc{X}\Lie_BU^b)  dydt\no\\
&\quad+\int_0^T \int_{\Omega}\dot{\omega}_{bc}U^bW^c dydt+\int_0^T \int_{\Omega}\Lie_B(\delta_{il}\D_b x^lB\D_cx^i-\delta_{il}B\D_b x^l\D_cx^i)U^bW^c  dydt\no\\
&\quad+\int_0^T \int_{\Omega}\Lie_B(g_{ab}\D_c B^a-g_{ac}\D_b B^a)U^bW^c dydt=\int_0^T \int_{\Omega}g_{ab}U^bF^a dydt
\end{align*}
for any smooth divergence-free vector field $U$ that vanishes for $t\gs T$. Moreover, due to $\dv W_\eps=0$, we get
\begin{align*}
\int_0^T\int_{\Omega}(\D_aq)W_\eps^a dydt=0
\end{align*}
for any smooth $q$ that vanishes on the boundary and thus
\begin{align}\label{DqW}
\int_0^T\int_{\Omega}(\D_aq)W^a dydt=0.
\end{align}
Therefore, $W$ is weakly divergence-free.

\section{Existence of smooth solutions for the linearized equation} \label{sec.exist}

In order to show that $W$ is divergence-free classical solution, we need to prove the additional regularity, i.e., $W,\dot{W}\in H^r(\Omega)$ for any $r\gs 0$. Then, the integration by parts for \eqref{DqW} yields
\begin{align*}
\int_0^T\int_{\Omega}q\D_aW^a dydt=0
\end{align*} 
for any smooth function $q$ that vanishes on the boundary. Thus, $W$ is divergence-free. 

Moreover,
\begin{align}\label{sol}
&\int_0^T\int_{\Omega} g_{ab}U^b (\ddot{W}^a-\Lie_B^2W^a+\A W^a+\dot{\mc{G}}\dot{W}^a-\mc{C}\dot{W}^a+\mc{X}\Lie_BW^a) dydt\\
&\qquad=\int_0^T \int_{\Omega}g_{ab}U^bF^a dydt\no
\end{align}
for any smooth, divergence-free vector field $U$ that vanished for $t\gs T$. Since $W$ is divergence-free, it follows that $\ddot{W}^a-\Lie_B^2W^a+\A W^a+\dot{\mc{G}}\dot{W}^a-\mc{C}\dot{W}^a+\mc{X}\Lie_BW^a$ is divergence-free. By construction, $F$ is also divergence-free, it follows that \eqref{sol} holds for any smooth vector field $U$ that vanishes for $t\gs T$. Thus, we conclude that 
\begin{align*}
\ddot{W}^a-\Lie_B^2W^a+\A W^a+\dot{\mc{G}}\dot{W}^a-\mc{C}\dot{W}^a+\mc{X}\Lie_BW^a=F^a, \quad \dv W=0.
\end{align*}

Therefore, it only remains to prove that $W\in H^r(\Omega)$. We must show that we have uniform bounds for the $\eps$ smoothed-out equation similar to the a priori bounds for the linearized equation. The uniform tangential bounds for the $\eps$ smoothed-out equation follow the proof of the a priori tangential bounds in Section \ref{sec.tang.r}, which is just a change of notation. Let
\begin{align*}
E_I^\eps=&\langle \dot{W}_{\eps I}, \dot{W}_{\eps I}\rangle +\langle W_{\eps I},(\A +I)W_{\eps I}\rangle+\langle \Lie_B W_{\eps I},\Lie_B W_{\eps I}\rangle, \quad W_{\eps I}=\Lie_T^I W_\eps.
\end{align*}
If $\eps<d_0$, then the commutator relation for $\A^\eps$, \eqref{comm.AL}, is exactly the same as for $\A$, \eqref{LTA'}. Moreover, the positivity property for $\A_f^\eps$ only differs from the one for $\A_f$ in which the supremum over the boundary in \eqref{est.A} is replaced by the supremum over a neighborhood of the boundary where $d(y)<\eps$ in \eqref{est.Aeps}. Thus, all the calculations and inequalities in Sections \ref{sec.tang.r}, \ref{sec.curl} and \ref{sec.curl.r} hold with $\A$ replaced by $\A^\eps$ if we replace the supremum of $\nb_Nq/\nb_NP$ over the boundary in \eqref{Dq} by the supremum of $q/P$ over the domain $\Omega\setminus\Omega^\eps$, where $\Omega^\eps$ is given by \eqref{Omegaeps}. Hence, we will reach the energy bound  \eqref{ErT} for $E_r^{\mc{T}}$ replaced by
\begin{align}\label{ErTeps}
E_r^{\mc{T},\eps}=\sum_{|I|\ls r, I\in\mc{T}} \sqrt{E_I^\eps},
\end{align}
namely, Proposition \ref{prop.ErT} holds for  $E_r^{\mc{T}}$ replaced by $E_r^{\mc{T},\eps}$ with a constant independent of $\eps$. It is where we need to have vanishing initial data and an inhomogeneous term that vanishes to higher order when $t=0$ so that the higher-order time derivatives of the solution of \eqref{eps.IVP} also vanish when $t=0$. If the initial data for higher-order time derivatives were obtained from the $\eps$ smoothed-out equation, then they would depend on $\eps$, and so we would not have been able to get a uniform bound for the energy $E_r^{\mc{T},\eps}$.

The estimate for the curl is simple since the curl of $\A_\eps$ vanishes in $\Omega^\eps$ by \eqref{curlAeps}, it follows that all the formula in Sections \ref{sec.curl} and \ref{sec.curl.r} hold when $d(y)\gs \eps$. This follows from replacing $\A$ in \eqref{Dtw} by $\A^\eps$ and vanishing of its curl  for $d(y)\gs \eps$. Let
\begin{align}
C_r^{\mc{U},\eps}=&\sum_{|J|\ls r-1,\, J\in\mc{U}}\left(\int_{\Omega^\eps}(|\curl \Lie_U^J w_\eps|^2+|\curl \Lie_U^J w_\eps|^2)dy\right)^{1/2},\\
\norm{W(t)}_{\mc{U}^r(\Omega^\eps)}=&\sum_{|I|\ls r, I\in\mc{U}}\left(\int_{\Omega^\eps}|\Lie_U^I W(t,y)|^2 dy\right)^{1/2}.\label{Omegaepsnorm}
\end{align}
Because all the used estimates from Section \ref{sec.tang} are pointwise estimates, we conclude that the inequality in Proposition \ref{prop.normal} holds with a constant $C$ independent of $\eps$ if we replace $C_s^{\mc{U}}$ by $C_s^{\mc{U},\eps}$ and the norms by \eqref{Omegaepsnorm}, as follows. 
\begin{proposition}\label{prop.normal.eps}
	Suppose that $x,P\in C^{r+2}([0,T]\times\Omega)$, $B\in C^{r+2}(\Omega)$, $P|_{\Gamma}=0$, $\nb_NP|_\Gamma\ls -c_0<0$, $B^aN_a|_\Gamma=0$ and $\dv V=0$, where $V=D_t x$. Suppose that $W_\eps$ is a solution of \eqref{eps.IVP} where $F$ is divergence-free and vanishing to order $r$ as $t\to 0$. Let $E_s^{\mc{T},\eps}$ be defined by \eqref{ErTeps}. Then, there is a constant $C=C(x,P,B)$ depending only on the norm of $(x,P,B)$, a lower bound for $c_0$, and an upper bound for $T$, but independent of $\eps$ such that if $E_s^{\mc{T},\eps}(0)=C_s^{\mc{U},\eps}(0)=0$ for $s\ls r$, then for $t\in [0,T]$
	\begin{align*}
	\norm{W_\eps}_{\mc{U}^r(\Omega^\eps)}+\norm{\dot{W}_\eps}_{\mc{U}^r(\Omega^\eps)}+\norm{\Lie_BW_\eps}_{\mc{U}^r(\Omega^\eps)}+E_r^{\mc{T},\eps}\ls C\int_0^t\norm{F}_r^{\mc{U}}d\tau.
	\end{align*}
\end{proposition}

Therefore, it implies that the limit $W$ satisfies the same bound with $\Omega^\eps$ replaced by $\Omega$, and so the weak solution in Section \ref{sec.weaksol} is indeed a smooth solution.

\section{The energy estimate with inhomogeneous initial data} \label{sec.inhomo}

In this section, we consider the original equations with inhomogeneous initial data and an inhomogeneous term:
\begin{align}\label{W}
\ddot{W}^a-\Lie_B^2W^a+\mc{A}W^a+\dot{\mc{G}}\dot{W}^a-\mc{C}\dot{W}^a+\mc{X}\Lie_B W^a=F^a.
\end{align}

We need some estimates of the commutators with the operator $\mc{A}$, $\dot{\mc{G}}$, $\mc{C}$ and tangential vector fields. We recall them from \cite{L1},
\begin{align*}
[\Lie_S,\mc{A}]W=&(\mc{A}_S-\mc{G}_S\mc{A})W,\quad 
&[\Lie_T,\mc{G}_S]W= &(\mc{G}_{TS}-\mc{G}_T\mc{G}_S)W,\\
[\Lie_S,\mc{C}]W=&(\mc{C}_S-\mc{G}_S\mc{C})W,\quad
&[\Lie_S^I,\mc{A}]W= &\tilde{d}_I^{I_1I_k}\mc{G}_{I_1}\cdots \mc{G}_{I_{k-2}}\mc{A}_{I_{k-1}}W_{I_k},\\
[\Lie_S^I,\dot{\mc{G}}]W=&\tilde{e}_I^{I_1I_k}\mc{G}_{I_1}\cdots \mc{G}_{I_{k-2}}\dot{\mc{G}}_{I_{k-1}}W_{I_k},\quad
&[\Lie_S^I,\mc{C}]W= &\tilde{e}_I^{I_1I_k}\mc{G}_{I_1}\cdots \mc{G}_{I_{k-2}}\mc{C}_{I_{k-1}}W_{I_k},
\end{align*} 
where $\mc{A}_S=\mc{A}_{SP}$, $\mc{G}_S=\mc{M}_{g^S}$ defined by $\mc{G}_SW^a=\Pdv(g^{ac}g_{cb}^SW^b)$, $g_{bc}^{TS}=\Lie_T\Lie_S g_{bc}$, $\mc{G}_{TS}W^a=\Pdv(g^{ab}g_{bc}^{TS}W^c)$,  $\mc{C}_TW^a=\Pdv(g^{ab}\omega_{bc}^TW^c)$, $\omega_{bc}^T=\Lie_T\omega_{bc}$, $\mc{G}_JW^a=\mc{M}_{g^J}W^a=\Pdv(g^{ac}g_{cb}^JW^b)$, $g_{ac}^J=\Lie_S^Jg_{ac}$, $\mc{A}_J=\mc{A}_{S^JP}$, and $W_J=\Lie_S^JW$, the sum is over all combinations with $I_1+I_2+\cdots+I_k=I$ in last three identities, with $k\gs 2$ and $|I_k|<|I|$, $\tilde{d}_I^{I_1I_k}$ and $\tilde{e}_I^{I_1I_k}$ are some constants. For the operator $\mc{X}$, we have similar equality for its commutator with tangential vector fields. Denote $\beta_{bc}=\delta_{il}\D_bx^l\Lie_B\D_cx^i$, we have
\begin{align*}
\Lie_T\mc{X}W^a=&\Lie_T(g^{ab}(-2\beta_{bc}W^c+\D_b q))\\
=&(\Lie_Tg^{ab})(-2\beta_{bc}W^c+\D_b q)-2g^{ab}(\Lie_T\beta_{bc})W^c-2g^{ab}\beta_{bc}\Lie_TW^c+g^{ab}\D_b Tq,
\end{align*}
where
\begin{align*}
\Delta q=&2\D_a(g^{ab}\beta_{bc}W^c), \quad q|_\Gamma=0.
\end{align*}
Projecting each term onto the divergence-free vector fields, we get
\begin{align*}
[\Lie_T,\mc{X}]W=(\mc{X}_T-\mc{G}_T\mc{X})W,
\end{align*}
where $\mc{X}_TW^a=\Pdv(-2g^{ab}\Lie_T(\delta_{il}\D_bx^l\Lie_B\D_cx^i)W^c)$.
Similarly, we have
\begin{align*}
[\Lie_S^I,\mc{X}]W=&\tilde{e}_I^{I_1I_k}\mc{G}_{I_1}\cdots \mc{G}_{I_{k-2}}\mc{X}_{I_{k-1}}W_{I_k}.
\end{align*}
These commutators are bounded operator and lower order since $|I_k|<|I|$. In addition, $[\Lie_T^I,\Lie_B^2]$ is also a bounded operator since $B$ is a tangential vector field.
Thus, we obtain
\begin{align*}
L_IW=\ddot{W}_I-\Lie_B^2W_I+\mc{A}W_I+\dot{\mc{G}}\dot{W}_I-\mc{C}\dot{W}_I+\mc{X}\Lie_B W_I=H_I,
\end{align*}
with
\begin{align}
H_I=&F_I+\tilde{d}_I^{I_1I_k}\mc{G}_{I_1}\cdots \mc{G}_{I_{k-2}}\mc{A}_{I_{k-1}}W_{I_k}\label{HI.1}\\
&+[\Lie_T^I,\Lie_B^2]W+\tilde{e}_I^{I_1I_k}\mc{G}_{I_1}\cdots \mc{G}_{I_{k-2}}\mc{X}_{I_{k-1}}[\Lie_T^{I_k},\Lie_B]W\label{HI.2}\\
&+\tilde{e}_I^{I_1I_k}\mc{G}_{I_1}\cdots \mc{G}_{I_{k-2}}\dot{\mc{G}}_{I_{k-1}}\dot{W}_{I_k}+\tilde{e}_I^{I_1I_k}\mc{G}_{I_1}\cdots \mc{G}_{I_{k-2}}\mc{C}_{I_{k-1}}\dot{W}_{I_k},\label{HI.3}
\end{align}
where $|I_k|<|I|$ and $F_I=\Lie_T^I F$.
We consider only $W_I=\Lie_S^I W$ with $S\in\mc{S}$, as before, let
\begin{align*}
E_I=&\langle \dot{W}_I, \dot{W}_I\rangle +\langle W_I,(\A +I)W_I\rangle+\langle \Lie_B W_I,\Lie_B W_I\rangle.
\end{align*}
The energy estimate is as similar as before, and we only need to estimate the $L^2$-norm of the $H_I$. It is obvious that \eqref{HI.2} and \eqref{HI.3} are bounded by $E_J$ for some $|J|\ls |I|$. Since $A_{I_k}$ is of order $1$, it contains derivatives in any direction, thus the term has to be estimated by $\norm{\D W_{I_k}}_{L^2(\Omega)}$, and then it does not directly get an estimate for $\norm{\Lie_S W_{I_k}}_{L^2(\Omega)}$ for all tangential derivatives $S$. However, we can combine the estimates for the curl to get the desired estimate.

Let $C_r^{\mc{R}}$ and $E_r^{\mc{S}}$ be defined as in \eqref{Cr} and \eqref{EsV}, respectively. Let $m_r^{\mc{V}}$,  $\dot{m}_r^{\mc{V}}$ and $\bar{m}_r^{\mc{V}}$ be defined as in \eqref{mr}. Then, we have by \eqref{L12.20}
\begin{align}\label{WR}
\norm{W}_r+\norm{\dot{W}}_r+\norm{\Lie_BW}_r\ls K_1\sum_{s=0}^r m_{r-s}^{\mc{R}}(C_s^{\mc{R}}+E_s^{\mc{S}}),
\end{align}
where $\norm{W}_r=\norm{W(t)}_{\mc{R}^r(\Omega)}$. Since the projection $\Pdv$ has norm $1$, and $\norm{G_JW}\ls \norm{g^J}_{L^\infty(\Omega)}\norm{W}$, it follows that
\begin{align}
\norm{\mc{G}_{I_1}\cdots \mc{G}_{I_{k-2}}\dot{\mc{G}}_{I_{k-1}}\dot{W}_{I_k}}\ls & \norm{g^{I_1}}_{L^\infty(\Omega)}\cdots\norm{g^{I_{k-2}}}_{L^\infty(\Omega)}\norm{\dot{g}^{I_{k-1}}}_{L^\infty(\Omega)}\norm{\dot{W}}_s\label{CW}\\
\ls& \dot{m}_{r-s}^{\mc{R}}\norm{\dot{W}}_s,\no\\
\norm{\mc{G}_{I_1}\cdots \mc{G}_{I_{k-2}}\dot{\mc{C}}_{I_{k-1}}\dot{W}_{I_k}}\ls & \norm{g^{I_1}}_{L^\infty(\Omega)}\cdots\norm{g^{I_{k-2}}}_{L^\infty(\Omega)}\norm{\omega^{I_{k-1}}}_{L^\infty(\Omega)}\norm{\dot{W}}_s\\
\ls& \dot{m}_{r-s}^{\mc{R}}\norm{\dot{W}}_s,\no
\end{align}
and
\begin{align*}
&\norm{\mc{G}_{I_1}\cdots \mc{G}_{I_{k-2}}\mc{X}_{I_{k-1}}[\Lie_T^{I_k},\Lie_B]W}\no\\
\ls& \norm{g^{I_1}}_{L^\infty(\Omega)}\cdots\norm{g^{I_{k-2}}}_{L^\infty(\Omega)}\norm{\Lie_T^{I_{k-1}}(\delta_{il}\D_bx^l\Lie_B\D_cx^i)}_{L^\infty(\Omega)}\norm{[\Lie_T^{I_k},\Lie_B]W}\no\\
\ls &(\dot{m}_{r-s}^{\mc{R}}+\bar{m}_{r-s}^{\mc{R}})\norm{W}_{s,B},
\end{align*}
where $s=|I_k|<r=|I|$. Denote
\begin{align*}
P_r^{\mc{R}}=\sum_{s=0}^r[[g]]_{r-s,\infty}^{\mc{R}}\sum_{|J|\ls s+1,J\in\mc{S}}\norm{\D S^JP}_{L^\infty(\D\Omega)}.
\end{align*}
Then, we have
\begin{align}\label{AW}
\norm{\tilde{d}_I^{I_1I_k}\mc{G}_{I_1}\cdots \mc{G}_{I_{k-2}}\mc{A}_{I_{k-1}}W_{I_k}}\ls & \norm{g^{I_1}}_{L^\infty(\Omega)}\cdots\norm{g^{I_{k-2}}}_{L^\infty(\Omega)}\norm{\mc{A}^{I_{k-1}}W_{I_k}}_{L^\infty(\Omega)}
\\
\ls &P_{r-s}^{\mc{R}}\norm{W}_s+P_{r-s-1}^{\mc{R}}\norm{W}_{s+1}.\no
\end{align}
Similar to \eqref{dotE0}, we can get
\begin{align}\label{EIS}
|\dot{E}_I|\ls& \left(1+2\norm{\dot{g}}_{L^\infty(\Omega)}+\frac{\norm{\D \dot{P}}_{L^\infty(\Gamma)}}{c_0}+2\norm{\D x}_{L^\infty(\Omega)}\norm{\Lie_B\D x}_{L^\infty(\Omega)}\right)E_I\\
&+2\sqrt{E_I}\norm{H_I},\no
\end{align}
where $c_0$ is the constant in the condition \eqref{nbNPcond}. From \eqref{CW}-\eqref{AW} and \eqref{WR}, we get
\begin{align}\label{HIest}
\norm{H_I}\ls& C\sum_{s=0}^{r-1}(\dot{m}_{r-s}^{\mc{R}}\norm{\dot{W}}_s+(\dot{m}_{r-s}^{\mc{R}}+\bar{m}_{r-s}^{\mc{R}})\norm{W}_{s,B}+P_{r-s}^{\mc{R}}\norm{W}_s)\\
&+P_0^{\mc{R}}\norm{W}_r+\norm{F}_r\no\\
\ls &K_1\sum_{s=0}^{r-1}(\dot{m}_{r-s}^{\mc{R}}+\bar{m}_{r-s}^{\mc{R}}+P_{r-s}^{\mc{R}})(C_s^{\mc{R}}+E_s^{\mc{R}})\no\\
&+K_1P_0^{\mc{R}}(C_r^{\mc{R}}+E_r^{\mc{R}})+\norm{F}_r.\no
\end{align}

Summing \eqref{EIS} over all $I\in\mc{S}$ with $|I|=r$ and using \eqref{HIest}, we have
\begin{align}\label{ErS}
\abs{\frac{dE_r^{\mc{S}}}{dt}}\ls& K_1\bigg(1+2\norm{\dot{g}}_{L^\infty(\Omega)}+\frac{\norm{\D \dot{P}}_{L^\infty(\Gamma)}}{c_0}+2\norm{\D x}_{L^\infty(\Omega)}\norm{\Lie_B\D x}_{L^\infty(\Omega)}\\
&\qquad\quad+\sum_{S\in\mc{S}}\norm{\D SP}_{L^\infty(\Gamma)} \bigg)(C_r^{\mc{R}}+E_r^{\mc{R}})\no\\
&+K_1\sum_{s=0}^{r-1}(\dot{m}_{r-s}^{\mc{R}}+\bar{m}_{r-s}^{\mc{R}}+P_{r-s}^{\mc{R}})(C_s^{\mc{R}}+E_s^{\mc{R}})+\norm{F}_r.\no
\end{align}
Since \eqref{CrU} holds with $\mc{U}$ and $\mc{T}$ replaced by $\mc{R}$ and $\mc{S}$, respectively, we get
\begin{align}\label{CrR}
\abs{\frac{dC_r^{\mc{R}}}{dt}}\ls & K_1 (\dot{m}_0^{\mc{R}}+\bar{m}_0^{\mc{R}})(C_r^{\mc{R}}+E_r^{\mc{S}})\\
&+K_1 \sum_{s=0}^{r-1} (\dot{m}_{r-s}^{\mc{R}}+\bar{m}_{r-s}^{\mc{R}})(C_s^{\mc{R}}+E_s^{\mc{S}})+\norm{F}_r.\no
\end{align}
Thus, \eqref{ErS} and \eqref{CrR} yield a bound for $C_r^{\mc{R}}+E_r^{\mc{S}}$ in terms of $C_s^{\mc{R}}+E_s^{\mc{S}}$ for $s<r$, namely
\begin{align*}
C_r^{\mc{R}}(t)+E_r^{\mc{S}}(t)\ls &K_1e^{K_1\int_0^t n d\tau}\bigg(C_r^{\mc{R}}(0)+E_r^{\mc{S}}(0)\no\\
&+\int_0^t \Big(\sum_{s=0}^{r-1} (\dot{m}_{r-s}^{\mc{R}}+\bar{m}_{r-s}^{\mc{R}})(C_s^{\mc{R}}+E_s^{\mc{S}})+\norm{F}_r\Big)d\tau\bigg),
\end{align*}
where 
\begin{align*}
n=&1+2\norm{\dot{g}}_{L^\infty(\Omega)}+\norm{\D \dot{P}}_{L^\infty(\Gamma)}/c_0+2\norm{\D x}_{L^\infty(\Omega)}\norm{\Lie_B\D x}_{L^\infty(\Omega)}\\
&+\sum_{S\in\mc{S}}\norm{\D SP}_{L^\infty(\Gamma)}+\norm{\omega}_{L^\infty(\Omega)}+\norm{\Lie_Bg}_{L^\infty(\Omega)}+\norm{\Lie_B^2g}_{L^\infty(\Omega)}+\norm{B}_{2,\infty}^{\mc{R}}\\
&+\norm{\D x}_{L^\infty(\Omega)}+\norm{\D^2 x}_{L^\infty(\Omega)}+\norm{\frac{\D y}{\D x}}_{L^\infty(\Omega)}.
\end{align*}
Because the bound for $E_0^{\mc{S}}=E_0$ have been already proven, we can get the bound for $C_r^{\mc{R}}+E_r^{\mc{S}}$ inductively.
Therefore, from \eqref{WR}, we obtain the following estimates.

\begin{proposition}
	Suppose that $x,P\in C^{r+2}([0,T]\times\Omega)$, $B\in C^{r+2}(\Omega)$, $P|_{\Gamma}=0$, $\nb_NP|_\Gamma\ls -c_0<0$, $B^aN_a|_\Gamma=0$ and $\dv V=0$, where $V=D_t x$. Let $W$ be the solution of \eqref{W} where $F$ is divergence-free. Then, there is a constant $C$ depending only on the norm of $(x,P,B)$, a lower bound for the constant $c_0$, and an upper bound for $T$, such that, for $s\ls r$, we have
	\begin{align}\label{WW}
	&\norm{W(t)}_r+\norm{\dot{W}(t)}_r+\norm{\Lie_BW(t)}_r+\langle W(t)\rangle_{\mc{A},r}\\
	\ls & C\bigg(\norm{W(0)}_r+\norm{\dot{W}(0)}_r+\norm{\Lie_BW(0)}_r+\langle W(0)\rangle_{\mc{A},r}+\int_0^t\norm{F}_rd\tau\bigg),\no
	\end{align}
	where
	\begin{align*}
	\norm{W(t)}_r=&\sum_{|I|\ls r, I\in\mc{R}} \norm{\Lie_U^IW(t)}_{L^2(\Omega)},\\
	\langle W(t)\rangle_{\mc{A},r}=&\sum_{|I|\ls r, I\in\mc{S}}\langle \Lie_S^IW(t),\mc{A}\Lie_S^IW(t) \rangle^{1/2}.
	\end{align*}
\end{proposition}

\section{The main result}\label{sec.mainresult}

As the same as in \cite{L1}, $\norm{W(t)}_r$ is equivalent to the usual time-independent Sobolev norm;  $\langle W(t)\rangle_{\mc{A},r}$ is only a seminorm on divergence-free vector fields, which is not only equivalent to a time-independent seminorm given by \eqref{Afdef} with $f$ the distance function $d(y)$ due to $0<c_0\ls -\nb_NP\ls C$, but also equivalent to the normal component of the vector field $W_N=N_aW^a$ being in $H^r(\Gamma)$ in view of \eqref{Afdef}, up to lower-order terms that can be controlled by $\norm{W(t)}_r$, since we only apply tangential vector fields.

We define $H^r(\Omega)$ to be the completion of $C^\infty(\Omega)$ in the norm $\norm{W(t)}_r$ and define $N^r(\Omega)$ to be the completion of the  $C^\infty(\Omega)$ divergence-free vector fields in the norm $\norm{W}_{N^r}=\norm{W(t)}_r+\langle W(t)\rangle_{\mc{A},r}$. Because the projection $\Pdv$ is continuous in the $H^r$ norm, it implies that $H^r$ is also the completion of the $C^\infty(\Omega)$ divergence-free  vector fields in the $H^r$ norm. We state the main result as follows.

\begin{theorem}
Suppose that $x,P\in C^{r+2}([0,T]\times\Omega)$, $B\in C^{r+2}(\Omega)$, $P|_{\Gamma}=0$, $\nb_NP|_\Gamma\ls -c_0<0$, $B^aN_a|_\Gamma=0$ and $\dv D_t x=0$. Then, if initial data and the inhomogeneous term in \eqref{data} are divergence-free and satisfy 
\begin{align*}
(W_0,W_1,\Lie_BW_0)\in N^r(\Omega)\times H^r(\Omega)\times H^r(\Omega), \quad F\in L^1\left([0,T],\, H^r(\Omega)\right),
\end{align*}
the linearized equations \eqref{IVP} have a solution
\begin{align}
(W, \dot{W}, \Lie_BW)\in C\left([0,T],\, N^r(\Omega)\times H^r(\Omega)\times H^r(\Omega)\right).\label{WWspace}
\end{align}
\end{theorem}

\begin{proof}
    If $W_0$, $W_1$ and $F$ are divergence-free and $C^\infty$, and $F$ is supported in $t>0$, then there exists a solution by the arguments in Section \ref{sec.exist}. It follows, by approximating $W_0$, $W_1$ and $F$ with $C^\infty(\Omega)$ divergence-free  vector fields and applying the estimate \eqref{WW} to the differences, that we obtain a convergent sequence in \eqref{WWspace}, thus the limit must be in the same space.
\end{proof}

\appendix
\section{Lie derivatives}\label{App.Lie}

Let us review the Lie derivative of the vector field $W$
with respect to the vector field $T$ constructed in the previous section,
\begin{align}\label{A.Lie}
\Lie_T W^a=TW^a -(\D_c T^a) W^c.
\end{align}
For those vector fields, it holds $\dv T=0$, so $\dv W=0$ implies that
$$
\dv \Lie_{T} W=T\dv W-W\dv T=0. 
$$

The Lie derivative of a $1$-form is defined by
$$
\Lie_T \alpha_a =T\alpha_a+(\D_a T^c) \alpha_c.
$$
The Lie derivatives also commute with the exterior differentiation,
$[\Lie_T, d]=0$, so if $q$ is a function, then
\begin{align}\label{A.Lieq}
\Lie_T \D_a q=\D_a T q.
\end{align}

The Lie derivative of a $2$-form is given by 
\begin{align}\label{Lie.2form}
\Lie_T \beta_{ab} =T\beta_{ab}
+(\D_a T^c) \beta_{cb}+(\D_b T^c) \beta_{ac}.
\end{align}

In general, in local coordinate notation, for a type $(r, s)$ tensor field $\beta$, the Lie derivative along $T$ is given by
\begin{align}\label{def.Lie}
\Lie_{T}\beta^{a_{1}\ldots a_{r}}{}_{b_{1}\ldots b_{s}}=&T\beta^{a_{1}\ldots a_{r}}{}_{b_{1}\ldots b_{s}}\no\\
&-(\partial _{c}T^{a_{1}})\beta^{ca_{2}\ldots a_{r}}{}_{b_{1}\ldots b_{s}}-\ldots -(\partial _{c}T^{a_{r}})\beta^{a_{1}\ldots a_{r-1}c}{}_{b_{1}\ldots b_{s}}\no\\
&+(\partial _{b_{1}}T^{c})\beta^{a_{1}\ldots a_{r}}{}_{cb_{2}\ldots b_{s}}+\ldots +(\partial _{b_{s}}T^{c})\beta^{a_{1}\ldots a_{r}}{}_{b_{1}\ldots b_{s-1}c}.
\end{align}
It follows that the Lie derivative satisfies the Leibnitz rule, e.g. 
\begin{align}\label{A.Lie.Leib}
\begin{aligned}
\Lie_T (\alpha_{c} W^c)=&
(\Lie_T \alpha_{c}) W^c+\alpha_{c} \Lie_T W^c,\\ 
\Lie_T (\beta_{ac} W^c)=&
(\Lie_T\beta_{ac}) W^c+\beta_{ac} \Lie_T W^c,\\
\Lie_T(g^{ab}\alpha_b)=&\Lie_Tg^{ab}\alpha_b+g^{ab}\Lie_T\alpha_b,
\end{aligned}
\end{align}
and
\begin{align}
\Lie_T (W^c\D_c\beta_{ab} )
=&\Lie_T W^c\D_c\beta_{ab}+W^c\Lie_T\D_c\beta_{ab}.
\end{align}

If $w$ is a $1$-form and 
$\curl w_{ab}= dw_{ab}=\D_a w_b-\D_b w_a$, 
then  
\begin{align}\label{Lie.curl}
\Lie_T \curl w_{ab}=\curl \Lie_T w_{ab},
\end{align}
since the Lie derivative commutes with exterior differentiation.

From \eqref{A.Lie}, we have the following relation on the commutator of two Lie derivatives
\begin{align}\label{Lie.commut}
[\Lie_T,\Lie_B]W^a=\Lie_{[T,B]}W^a.
\end{align}
From \eqref{def.Lie}, we get the commutator of Lie derivative and $\D_a$
\begin{align}\label{LieDy.commu}
[\Lie_T, \D_a]W^b=W^d\D_d\D_aT^b.
\end{align}
Furthermore, we also treat $D_t$ as if it were a Lie derivative and we set
\begin{align}\label{A.LieDt}
\Lie_{D_t}=D_t.
\end{align}
Of course, this is not a space Lie derivative but rather could be interpreted 
as a space-time Lie derivative in the domain $[0,T]\times \Omega$. 
It is important that it satisfies all the properties of the 
other Lie derivatives considered, such as $\dv W=0$
implies that $\dv D_t W=0$ and $D_t \curl w=\curl D_t w$,
because it commutes with partial differentiation with respect to the
$y$ coordinates. It is more 
efficient with the same notation, since 
we will apply products of Lie derivatives and \eqref{A.LieDt}.
Moreover,
\begin{align}\label{A.comm.DtLie}
[\Lie_{D_t},\Lie_T]=0,
\end{align}
because this quantity is $\Lie_{[D_t,T]}$ and $[D_t,T]=0$ for the 
vector fields we considered, or it follows from \eqref{A.Lie} and that 
$T^a=T^a(y)$ is independent of $t$. 

\bigskip

{\bf Acknowledgments.} 
Hao's research was supported by National Natural Science Foundation of China (Grant No. 11671384).  Luo's research was  supported by a grant from the Research Grants Council of  Hong Kong  (Project No. 11305818).


\begin{thebibliography}{10}
\addcontentsline{toc}{section}{References} 

      \bibitem{AZ}Alazard, T.; Burq, N.; Zuily, C.
      \newblock On the water waves equations with surface tension. \emph{Duke Math. J.} \textbf{158} (2011), no. 3, 413--499. 
      
    \bibitem{ABZ} Alazard, T.; Burq, N.; Zuily, C.
     On the Cauchy problem for gravity water waves. \emph{Invent. Math.} \textbf{198} (2014), no. 1, 71--163.
     
\bibitem{AM}Ambrose, D.M.; Masmoudi, N.
The zero surface tension limit of two-dimensional water waves. 
\newblock \emph{Comm. Pure Appl. Math.} \textbf{58} (2005), no. 10, 1287--1315. 

\bibitem{BG}Beyer, K.; G${\rm \ddot{u}}$nther, M.
 On the Cauchy problem for a capillary drop. I. Irrotational motion.
  \emph{Math. Methods Appl. Sci.} \textbf{21} (1998), no. 12, 1149--183. 
  
	\bibitem{chenwang}
	 Chen, G.-Q.; Wang, Y.-G. 
	\newblock Existence and stability of compressible current-vortex sheets in
	three-dimensional magnetohydrodynamics.
	\newblock {\em Arch. Ration. Mech. Anal.} \textbf{187} (2008), no. 3, 369--408.
	
	\bibitem{CD19}
	Chen, P.; Ding, S.
	\newblock Inviscid limit for the free-boundary problems of mhd equations with
	or without surface tension.
	\newblock Preprint (2019), arXiv:1905.13047.
	
	\bibitem{CHS1} Christianson, H.; Hur, V. M.; Staffilani, G.  
	Strichartz estimates for the water- wave problem with surface tension. 
	\emph{Comm. Partial Differential Equations} \textbf{35} (2010), 2195–-2252.


	
	\bibitem{CL}
	 Christodoulou, D.; Lindblad, H.
	\newblock On the motion of the free surface of a liquid.
	\newblock {\em Comm. Pure Appl. Math.} \textbf{53} (2000), no. 12, 1536--1602.
	
	\bibitem{Coutand}
Coutand, D.; Shkoller, S.
Well-posedness of the free-surface incompressible
  {E}uler equations with or without surface tension.
\newblock \emph{J. Amer. Math. Soc.} \textbf{20} (2007), no. 3, 829--930.

\bibitem{CHS} Coutand, D.; Hole, J.; Shkoller, S.
  Well-posedness of the free-boundary compressible 3-D Euler equations with surface tension and the zero surface tension limit. 
  \emph{SIAM J. Math. Anal.} \textbf{45} (2013), no. 6, 3690--3767.

\bibitem {6} Cox, J.P.; Giuli, R.T.
\newblock  Principles of stellar structure, I.,II. Gordon and Breach, New
York, 1968.

\bibitem{Ebin}
Ebin, D.G. The equations of motion of a perfect fluid with free boundary are
  not well posed.
\newblock \emph{Comm. Partial Differential Equations} \textbf{12} (1987), no. 10, 1175--1201.
  


	\bibitem{GW}
	Gu, X.; Wang, Y.
	\newblock On the construction of solutions to the free-surface incompressible
	ideal magnetohydrodynamic equations.
	\newblock {\em J. Math. Pures Appl. } \textbf {128}  (2019), 1--41.
	
	\bibitem{Hao17}
	Hao, C.
	\newblock On the motion of free interface in ideal incompressible {MHD}.
	\newblock {\em Arch. Ration. Mech. Anal.} \textbf{224} (2017), no. 2, 515--553.
	
	\bibitem{HLarma}
	Hao, C.; Luo, T.
	\newblock A priori estimates for free boundary problem of incompressible
	inviscid magnetohydrodynamic flows.
	\newblock {\em Arch. Ration. Mech. Anal.} \textbf{212} (2014), no. 3, 805--847.
	
	\bibitem{HLip}
	Hao, C.; Luo, T.
	\newblock Ill-posedness of free boundary problem of the incompressible ideal
	{MHD}.
	\newblock Preprint  (2018), arXiv:1810.07465.
	
	\bibitem{LD2}Lannes, D. The water waves problem: mathematical analysis and asympototics. In: Mathematical Surveys and Monographs, vol. 188. American Mathematical Society Providence, RI, 2013.

	\bibitem{Lee}
	Lee, D.
	\newblock Uniform estimate of viscous free-boundary magnetohydrodynamics with
	zero vacuum magnetic field.
	\newblock {\em SIAM J. Math. Anal.} \textbf{49} (2017), no. 4, 2710--2789.
	
	
	\bibitem{Lee2}
	Lee, D.
	\newblock Initial value problem for the free-boundary magnetohydrodynamics with
	zero magnetic boundary condition.
	\newblock {\em Commun. Math. Sci.} \textbf{16} (2018), no. 3, 589--615.
	
	\bibitem{L1}
    Lindblad, H.
	\newblock Well-posedness for the linearized motion of an incompressible liquid
	with free surface boundary.
	\newblock {\em Comm. Pure Appl. Math.} \textbf{56} (2003), no. 2, 153--197.
	
	\bibitem{L2}
	 Lindblad, H.
	\newblock Well-posedness for the motion of an incompressible liquid with free
	surface boundary.
	\newblock {\em Ann. of Math. (2)} \textbf{162} (2005), no. 1, 109--194.
	
	\bibitem{LN}
Lindblad, H.; Nordgren, K.H. A priori estimates for the motion of a
  self-gravitating incompressible liquid with free surface boundary.
\newblock \emph{J. Hyperbolic Differ. Equ.} \textbf{6}(2009), no. 2, 407--432. 


	\bibitem{LZ19}
    Luo C.; Zhang, J.
	\newblock A regularity result for the incompressible magnetohydrodynamics
	equations with free surface boundary.
	\newblock Preprint (2019), arXiv:1904.05444.
	
	\bibitem{LZ}
	Luo, T.;  Zeng, H. 
	\newblock On the Free Surface Motion of Highly Subsonic Heat-conducting Inviscid Flows.
	\newblock Preprint  (2017), arXiv:1709.06925.
	
	\bibitem{MTT14}
	Morando, A.; Trakhinin, Y.; Trebeschi, P.
	\newblock Well-posedness of the linearized plasma-vacuum interface problem in
	ideal incompressible {MHD}.
	\newblock {\em Quart. Appl. Math.} \textbf{72} (2014), no. 3, 549--587.
	
	\bibitem{PS10}
	Padula, M.; Solonnikov, V.~A.
	\newblock On the free boundary problem of magnetohydrodynamics.
	\newblock {\em Zap. Nauchn. Sem. S.-Peterburg. Otdel. Mat. Inst. Steklov.
		(POMI)}, 385:135--186, 2010.
	\newblock Translation in J. Math. Sci. (N. Y.) \textbf{178} (2011), no. 3, 313--344.
	.
	
	\bibitem{PN}  Poyferre, T.; Nguyen, Q.-H. 
	A paradifferential reduction for the gravity-capillary waves system at low regularity and applications. 
	\emph{Bull. Soc. Math. France} \textbf{145} (2017), 643--710.


	
	\bibitem{ST}
Secchi, P.; Trakhinin, Y.
	\newblock Well-posedness of the plasma-vacuum interface problem.
	\newblock {\em Nonlinearity} \textbf{27} (2014), no. 1, 105--169.
	
\bibitem{26}
Shapiro, S.H.; Teukolsky, S.A.
\newblock Black Holes, White Dwarfs, and Neutron Stars.
WILEY-VCH, 2004.
	
	\bibitem{SZ}
Shatah, J.; Zeng, C. Geometry and a priori estimates for free boundary
  problems of the {E}uler equation.
\newblock \emph{Comm. Pure Appl. Math.} \textbf{61} (2008), no. 5, 698--744.



\bibitem{SWZ} Sun, Y.; Wang, W.; Zhang, Z. Well-posedness of the plasma-vacuum interface problem for ideal incompressible MHD. Preprint (2017), arXiv:1705.00418.


	
	\bibitem{Trakhinin}
	Trakhinin, Y.
	\newblock On well-posedness of the plasma-vacuum interface problem: the case of
	non-elliptic interface symbol.
	\newblock {\em Commun. Pure Appl. Anal.} \textbf{15}(2016), no. 4, 1371--1399.
	
	\bibitem{35} Trakhinin, Y.
	Local existence for the free boundary problem for the non-relativistic and relativistic compressible Euler equations with a vacuum boundary condition. 
	\emph{Commun. Pure Appl. Math.} \textbf{62} (2009), 1551--1594.

	
	\bibitem{WangXin17}
	Wang, Y.; Xin, Z.
	\newblock Incompressible inviscid resistive mhd surface waves in 2d.
	\newblock Preprint (2018), arXiv:1801.04694.
	
	\bibitem{WangYu13}
	Wang, Y.-G.; Yu, F.
	\newblock Stabilization effect of magnetic fields on two-dimensional
	compressible current-vortex sheets.
	\newblock {\em Arch. Ration. Mech. Anal.} \textbf{208} (2013), no. 2, 341--389.
	
	\bibitem{wu1}
Wu, S.
\newblock Well-posedness in {S}obolev spaces of the full water wave problem in
  {$2$}-{D}.
\newblock {\em Invent. Math. } \textbf{130}(1997), no. 1, 39--72.

\bibitem{wu2}
Wu, S. 
\newblock  Well-posedness in {S}obolev spaces of the full water wave problem in
  3-{D}.
\newblock \emph{J. Amer. Math. Soc.} \textbf{12} (1999), no. 2, 445--495.

\bibitem{zhang}
Zhang, P.; Zhang, Z.
 On the free boundary problem of three-dimensional
  incompressible {E}uler equations.
\newblock \emph{Comm. Pure Appl. Math.} \textbf{61}  (2008), no. 7, 877--940.

\bibitem{27}
Zirin, H.
\newblock  Astrophysics of the Sun. 
\newblock Cambridge University Press, Cambridge, 1988.

\end{thebibliography}

\vspace*{1cm}
\noindent\parbox[t]{0.5\textwidth}{
	 \textsc{Chengchun Hao}\\
HLM, Institute of Mathematics\\
 Academy of Mathematics and
 Systems Science\\
Chinese Academy of Sciences\\
Beijing 100190\\
P.R. CHINA\\[1mm]
\hspace*{3mm} and
 
School of Mathematical Sciences\\
University of Chinese Academy of Sciences\\
Beijing 100049\\
P.R. CHINA\\
E-mail: hcc@amss.ac.cn
}\quad
\parbox[t]{0.5\textwidth}{
\noindent \textsc{Tao Luo}\\
Department of Mathematics\\
City University of Hong Kong\\
Tat Chee Avenue, Kowloon\\
HONG KONG\\
E-mail: taoluo@cityu.edu.hk
}

\end{document}